\documentclass[11pt,reqno]{amsproc}
\DeclareMathOperator\supp{supp}

\usepackage{amsmath,amsfonts,amssymb,amsthm}
\usepackage[abbrev,lite,nobysame]{amsrefs}
\usepackage{mathrsfs}
\usepackage{graphics,graphicx}
\usepackage[usenames,dvipsnames]{color}
\usepackage{times}
\usepackage{bbm}
\usepackage[margin=1in]{geometry}

\usepackage[colorlinks=true, pdfstartview=FitV, linkcolor=BrickRed,citecolor=black, urlcolor=black]{hyperref}

\usepackage{mathtools}
\mathtoolsset{showonlyrefs}

\newtheorem{theorem}{Theorem}
\newtheorem{proposition}{Proposition}[section]
\newtheorem{lemma}[proposition]{Lemma}
\newtheorem{corollary}[proposition]{Corollary}

\theoremstyle{definition}

\newtheorem{remark}[proposition]{Remark}

\numberwithin{equation}{section}

\newcommand\eps{\varepsilon}
\newcommand\e{{\rm e}}
\newcommand\dd{{\rm d}}
\newcommand\ddt{{\frac{\dd}{\dd t}}}
\def\Re{{\rm Re}}

\def\l {\langle}
\def\r {\rangle}
\newcommand\de{{\partial}}

\newcommand\cA{{\mathcal A}}
\newcommand\cB{{\mathcal B}}
\newcommand\cF{{\mathcal F}}
\newcommand\cI{{\mathcal I}}
\newcommand\cL{{\mathcal L}}
\newcommand\cM{{\mathcal M}}
\newcommand\cN{{\mathcal N}}
\newcommand\cP{{\mathcal P}}
\newcommand\cR{{\mathcal R}}
\newcommand\cS{{\mathcal S}}
\newcommand\cT{{\mathcal T}}
\newcommand\cV{{\mathcal V}}
\newcommand\sfE{{\mathsf E}}

\newcommand{\norm}[1]{\left\lVert #1 \right\rVert}

\newcommand{\abs}[1]{\left\lvert #1 \right\rvert}

\makeatletter

\renewcommand\subsubsection{\@startsection{subsubsection}{3}%
\normalparindent{.5\linespacing\@plus.7\linespacing}{-.5em}
{\normalfont\bfseries}}

\def\@tocline#1#2#3#4#5#6#7{\relax
  \ifnum #1>\c@tocdepth % then omit
  \else
    \par \addpenalty\@secpenalty\addvspace{#2}%
    \begingroup \hyphenpenalty\@M
    \@ifempty{#4}{%
      \@tempdima\csname r@tocindent\number#1\endcsname\relax
    }{%
      \@tempdima#4\relax
    }%
    \parindent\z@ \leftskip#3\relax \advance\leftskip\@tempdima\relax
    \rightskip\@pnumwidth plus4em \parfillskip-\@pnumwidth
    #5\leavevmode\hskip-\@tempdima
      \ifcase #1
       \or\or \hskip 1em \or \hskip 2em \else \hskip 3em \fi%
      #6\nobreak\relax
    \dotfill\hbox to\@pnumwidth{\@tocpagenum{#7}}\par
    \nobreak
    \endgroup
  \fi}
\makeatother

\newcommand{\CC}{\mathbb{C}}
\newcommand{\NN}{\mathbb{N}}
\newcommand{\ZZ}{\mathbb{Z}}
\newcommand\TT {{\mathbb T}}
\newcommand\RR {{\mathbb R}}

\renewcommand{\circ}{\Gamma}
\renewcommand{\star}{\GG}

\begin{document}
\title[Stability of viscous $3d$ stratified Couette flow via dispersion and mixing]{\vspace*{-3cm}Stability of viscous three-dimensional stratified Couette flow \\ via dispersion and mixing} 

\author[M. Coti Zelati]{Michele Coti Zelati}
\address{Department of Mathematics, Imperial College London}
\email{m.coti-zelati@imperial.ac.uk}

\author[A. Del Zotto]{Augusto Del Zotto}
\address{Department of Mathematics, Imperial College London}
\email{a.del-zotto20@imperial.ac.uk}

\author[K. Widmayer]{Klaus Widmayer}
\address{Faculty of Mathematics, University of Vienna \& Institute of Mathematics, University of Zurich}
\email{klaus.widmayer@math.uzh.ch}



\begin{abstract}
This article explores the stability of stratified Couette flow in the viscous $3d$ Boussinesq equations. In this system, mixing effects arise from the shearing background, and gravity acts as a restoring force leading to dispersive internal gravity waves. These mechanisms are of fundamentally different nature and relevant in complementary dynamical regimes. Our study combines them to establish a bound for the nonlinear transition threshold, which is quantitatively larger than the inverse Reynolds number $\nu$, and increases with stronger stratification resp.\ gravity.

\end{abstract}

\maketitle

\newcommand\q{q}
\newcommand\Q{Q}
\newcommand\GG{G}

\setcounter{tocdepth}{2}
\tableofcontents

\section{Introduction}
This article is devoted to the study of stable dynamics in the three-dimensional Boussinesq equations
\begin{equation}\label{eq:3DBoussinesq}
\begin{cases}
\de_t v+(v\cdot\nabla)v+\nabla p=\nu\Delta v+\mathfrak{g}\vartheta \vec{e}_y,\qquad \nabla\cdot v=0,\\
\de_t\vartheta+v\cdot\nabla\vartheta=\nu\Delta \vartheta,
% \nabla\cdot v=0.
\end{cases}
\end{equation}
This system describes the evolution of an incompressible, viscous and inhomogeneous fluid with velocity $v(t,x,y,z)\in\RR^3$, pressure $p(t,x,y,z)\in\RR$ and temperature $\vartheta(t,x,y,z)\in\RR$. The dynamics of the fluid are given by the classical momentum equation of Navier-Stokes and are coupled to an advection-diffusion equation for the temperature through buoyancy forces due to gravity (acting here in the $y$ direction with constant of gravity $\mathfrak{g}>0$). The parameter $\nu\in(0,1)$ is the kinematic viscosity coefficient, proportional to the inverse Reynolds number of the fluid, which for convenience we take equal to the diffusivity parameter in the $\vartheta$-equation. 

One of the main reasons for our interest in \eqref{eq:3DBoussinesq} is the fact that it provides a comparatively simple yet highly relevant setting to study the interplay of different stabilizing mechanisms in fluid systems. In particular, in addition to damping due to viscosity, \eqref{eq:3DBoussinesq} exhibits both dispersive and mixing effects, which are frequently at the origin of stable fluid flows. A crucial challenge hereby lies in the fact that although in general one may expect stability to arise from a combination of such mechanisms, their distinct nature makes it difficult to treat them in a combined fashion, and the quantitative analysis often (and necessarily) relies on rather different tools -- we refer to the discussion below for more details. The goal of this article is to implement an approach that overcomes this difficulty: we demonstrate how mixing and dispersion combine to yield an improved stability threshold for dynamics near the classical stationary structure given by a linearly stratified Couette flow.

More precisely, we study \eqref{eq:3DBoussinesq} on the spatial domain $(x,y,z)\in\TT\times\RR\times\TT$. Such three-dimensional channels are a natural setting to study dynamics of stratified flow in the absence of boundaries, and admit a large family of stationary states: shear flows $v=(f(y),0,0)$ with linearly stratified temperature profiles $\vartheta=a+by$, $a,b\in\RR$. Amongst them, the stably stratified Couette flow stands out as a particularly simple yet relevant example:
\begin{equation}\label{eq:steadystate}
v^s=(y,0,0),\qquad \de_y p^s=\mathfrak{g}(1+\alpha y), \qquad \vartheta^s=1+\alpha y,\qquad \alpha>0.
\end{equation}
The choice of sign $\alpha>0$ hereby assures that with respect to the direction of gravity, warmer fluid is on top of colder fluid. This is referred to as stable stratification, since gravity acts as a restoring force for perturbations in the $y$-direction,\footnote{Contrast this with the case $\alpha<0$, where gravity gives rise to instability of Rayleigh-B\'enard type.} which gives rise to internal gravity waves. 

Writing $v=v^s+u$, $\vartheta=\vartheta^s-\sqrt{\alpha/\mathfrak{g}}\,\theta$, the perturbations $(u,\theta)$ of \eqref{eq:steadystate} in \eqref{eq:3DBoussinesq} satisfy
\begin{equation}\label{eq:3DBoussinesqPert}
\begin{cases}
\de_t u+y\de_xu+u^{2}\vec{e}_x+ (u\cdot\nabla)u+\nabla p=\nu\Delta u-\beta\theta \vec{e}_y,\quad \nabla\cdot u=0, \\
\de_t\theta+y\de_x\theta -\beta u^{2}+u\cdot\nabla\theta=\nu\Delta\theta,
\end{cases}
\end{equation}
where $\beta=\sqrt{\alpha \mathfrak{g}}$ is the Brunt-V\"ais\"al\"a frequency, reflecting the strength of the response of the fluid to displacements in the direction of gravity.

As may be apparent, the dynamics of solutions to \eqref{eq:3DBoussinesqPert} which are independent of $x$ resp.\ $x$ and $z$ are qualitatively and quantitatively different from those that depend on the $x$ and $z$ variables. To reflect this in the analysis, we decompose functions $\varphi:\mathbb{T}\times\mathbb{R}\times\mathbb{T}\to \mathcal{V}$, $\cV\in\{\RR,\RR^3\}$, as
\begin{equation}
 \varphi(x,y,z)=\varphi_0(y,z)+\varphi_{\neq}(x,y,z),\qquad \varphi_0(y,z):=\frac{1}{2\pi}\int_{\mathbb{T}}\varphi(x,y,z)\dd x,
\end{equation}
where we call $\varphi_0$, the mean in $x$ of $\varphi$, the \emph{zero mode} of $\varphi$.
Moreover, we let 
\begin{equation}\label{eq:zmean-notation}
 \overline{\varphi}_0(y):=\frac{1}{2\pi}\int_{\TT}\varphi_0(y,z)\dd z,\qquad \widetilde{\varphi}_0(y,z):=\varphi_0(y,z)-\overline{\varphi}_0(y),
\end{equation}
denote the mean resp.\ mean-free components in $z$ of $\varphi_0$. We refer to $\overline{\varphi}_0$ as the \emph{double zero} and $\widetilde{\varphi}_0$ as the \emph{simple zero} mode of $\varphi$.

Our main result (Theorem \ref{thm:transitionthreshold}) provides a lower bound for the size of the basin of attraction of the stably stratified Couette flow \eqref{eq:steadystate} in \eqref{eq:3DBoussinesq} in Sobolev regularity:
\begin{theorem}[Transition threshold]\label{thm:transitionthreshold}
  Let $m\geq 3$ and $0<\nu<1$. There exist universal constants $c_1,c_2>0$ such that the following holds true. For initial data $(u(0),\theta(0))$ with vanishing $x$-$z$-mean
  \begin{equation}\label{eq:zero-mean-xz-id}
   \iint_{\TT\times\TT}u(0)\dd x\dd z=\iint_{\TT\times\TT}\theta(0)\dd x\dd z=0, 
  \end{equation}
  and size
  \begin{equation}\label{eq:id_size}
   \norm{u(0)}_{H^{2m+1}\cap W^{2m+5,1}}+\norm{\theta(0)}_{H^{2m+1}\cap W^{2m+5,1}}\leq\eps_0,
  \end{equation}
  there exists a unique, global solution $(u(t),\theta(t))\in C_tH^{2m}([0,\infty)\times\TT\times\RR\times\TT)$, which moreover satisfies an inviscid damping and enhanced dissipation estimate
  \begin{equation}\label{eq:basicbounds}
      \|u_{\neq}^1(t)\|_{L^2}+\l t\r^{\frac32}\|u_{\neq}^2(t)\|_{L^2}+\|u_{\neq}^3(t)\|_{L^2}+\l t\r^{\frac12}\|\theta_{\neq}(t)\|_{L^2}\lesssim \eps_0\lambda^{-\frac12}\e^{-\lambda\nu^{\frac13} t},\qquad \lambda(\beta):=\frac{2\beta-1}{2\beta+1},
  \end{equation} 
  as well as dispersive bounds
  \begin{equation}\label{eq:dispersion}
    \|u^2_0(t)\|_{L^\infty}+\|\widetilde{u}^3_0(t)\|_{L^\infty}+\|\widetilde{\theta}_0(t)\|_{L^\infty}\lesssim \eps_0\beta^{-\frac13}\left(t^{-\frac13}\e^{-\nu t} + \nu^{-\frac23}\eps_0\right), 
  \end{equation}
  provided one of the following two options holds true:
  \begin{equation}\label{eq:transthresh1}
      \beta>\frac12 \quad \textnormal{and}\quad \eps_0\leq c_1\nu^{\frac{11}{12}}
  \end{equation}
  or
  \begin{equation}\label{eq:transthresh2}
      \beta>c_2\nu^{-\frac12} \quad \textnormal{and}\quad  \eps_0\leq c_1\nu^{\frac89}.
  \end{equation}
  \end{theorem}
  
Besides being the first such transition threshold established for the $3d$ Boussinesq equations, the key novelty in our approach lies in its reliance and exploitation of both mixing effects (enhanced dissipation and inviscid damping around the Couette flow) and a dispersive, oscillatory mechanism (due to buoyancy forces around the linear stratification). These qualitatively and quantitatively different stabilizing dynamics are captured e.g.\ in \eqref{eq:basicbounds} and \eqref{eq:dispersion}. In particular, they allow us (see \eqref{eq:transthresh1}--\eqref{eq:transthresh2}) to quantify the size of the stability transition as at least $\nu^p\gg\nu$, for a $p<1$. To the best of our knowledge, this is the first instance of a threshold that is quantitatively larger than $\nu$ in a three-dimensional hydrodynamic stability problem. Moreover, the influence of the strength of the coupling is quantitatively tracked through the Brunt-V\"ais\"al\"a frequency $\beta$ and allows to further increase the threshold, provided $\beta$ is sufficiently large.

\begin{remark}
 Theorem \ref{thm:transitionthreshold} is a simplified version of the full result we establish. 
 \begin{itemize}
     \item[$-$] In Propositions \ref{prop:GGamma}--\ref{prop:doublezero} we obtain more precise information on the thresholds for the various components of \eqref{eq:3DBoussinesqPert}. In particular, the transition threshold for the nonzero modes is $\nu^{\frac56}$, while the threshold for the simple and double zero modes continuously improves with increasing $\beta$ from \eqref{eq:transthresh1} to \eqref{eq:transthresh2}. The largest contributions hereby are due to the dynamics of $u^1$ -- see also the discussion below. 
     \item[$-$] The mean-zero condition \eqref{eq:zero-mean-xz-id} is automatic for $u^2(0)$ by incompressibility, and in fact only required for $u^3$ and $\theta$. It can be relaxed in a quantified fashion, but an assumption of this nature is necessary for a threshold larger than $\nu$. 
     \item[$-$] Furthermore, with minor adaptions of our methods it seems possible to prove our result also in the case of kinematic viscosity $\nu>0$ differing from diffusivity $\kappa>0$ (as in \cite{CZDZ23}), provided they satisfy
\begin{equation}
    \frac{\max\{\nu,\kappa\}}{\min\{\nu,\kappa\}}<4\beta-1.
\end{equation}
    \item[$-$] The classical spectral stability for $2d$ inviscid stratified flows, known as the Miles-Howard criterion \cites{M61, H61}, requires the Richardson number $\beta^2$ to be greater than $\frac14$. This condition is reflected in \eqref{eq:transthresh1} and used to ensure a coercivity condition of a certain energy functional (see \eqref{eq:GGamma_energy} below), although in the $3d$ viscous setting we consider, this may not be necessary. 
    \item[$-$] We expect our approach to also yield (improved) thresholds in other settings that combine mixing and dispersive dynamics, such as the $3d$ magnetohydrodynamics setting of \cite{L18}. This will be the subject of future study.
 \end{itemize}

\end{remark}

\subsubsection*{Context}
The study of the stability of laminar flow -- and in particular its most basic example, the Couette flow -- in viscous flows at high Reynolds number has a long history, dating back to the end of the nineteenth century \cites{rayleigh1879stability,kelvin1887stability,reynolds1883xxix}. Since then, countless articles have been devoted to estimate, in terms of relevant parameters, the maximal size of perturbations for which a flow avoids a transition to a turbulent state, and instead (asymptotically) retains essential, stable features. That in general the size of such a \emph{transition threshold} may depend on the Reynolds number was clear to O.\ Reynolds himself in 1883, in view of his famous experiments \cite{reynolds1883xxix}, which demonstrated the transition to turbulence of certain laminar configurations when increasing the flow rates in a pipe. However, proving optimal size bounds requires a deep and quantitative understanding of the dynamics. To date, the best understood setting in $3d$ is that of the homogeneous Navier-Stokes equations near Couette flow (in certain channel-like geometries) \cites{BGM17,BGM20,BGM22,WZ21,CWZ20}. Recently, important progress has been made to extend such results to more general, non-monotone shears \cites{CDLZ23,LWZ20Kolmo}. However, due to their relation with atmospheric and oceanic sciences, as well as engineering applications involving heat transfer \cites{pedlosky2013geophysical, vallis2017atmospheric,wyngaard2010turbulence}, these questions are not only relevant in the case of homogeneous fluids governed by the Navier-Stokes equations, but also in non-homogeneous fluids as described by \eqref{eq:3DBoussinesq}.
\paragraph{\emph{Mixing effects in homogeneous fluids}}
The presence of a background shear flow is responsible for fluid mixing, a mechanism that produces small scales and causes inviscid damping and enhanced dissipation. In our setting, these two effects are sharply quantified in \eqref{eq:basicbounds}, the former in the $\nu$-independent algebraic decay rate of $u^2,\theta$, and the latter via the exponential decay on a time-scale of order $\nu^{-\frac13}$, which is much shorter than the dissipative one proportional to $\nu^{-1}$. 

Our understanding of these effects is best in $2d$. There inviscid damping near shear flows has been studied in the $2d$ homogeneous Euler equations both at the linear level \cites{BCZV19,CZZ19,GNRS20,Jia20,JiaGev20,WZZ18,WZZ19,WZZKolmo20,Zillinger16,IIJ24} and at the nonlinear level \cites{BM15,IJ22,IJ20,IJnon20,MZ20}. When dissipation is present, the nonlinear stability of shear flows in the $2d$ Navier-Stokes equations and related models have been studied in various contexts \cites{BMV16,BVW18,DL23,WZZKolmo20,WZ23,DelZotto23,DL22,CZEW20,CLWZ20,WZ21}. We only highlight that for the $2d$ Couette flow on $\TT\times\RR$, the stability threshold depends on the regularity of the perturbation: it is at least of order $\nu^\frac13$ in Sobolev regularity \cite{MZ22}, and independent of viscosity in Gevrey regularity \cite{BMV16}. This dependence on the topology is due to the so-called Orr mechanism, a transient growth of the stream function that can be suppressed by regularity and (in part) by viscosity. 

In $3d$, there are to date no stability results of shear flows in the inviscid setting. However, in the Navier-Stokes equations, the nonlinear stability of Couette flow in a channel was investigated in \cites{BGM17,BGM20,BGM22,CWZ20,WZ21}. Here one encounters a severe instability mechanism, known as the lift-up effect: this involves the stretching and tilting of vortices by the Couette flow, causing complex flow patterns. This forces a growth of order $\nu^{-1}$ in the first component of the velocity, and implies that the stability thresholds of order $\nu$ in both Gevrey \cites{BGM20,BGM22} and Sobolev \cites{WZ21,CWZ20} are sharp. More general monotone shears are still awaiting exploration, but recent progress has been made for other prototypical non-monotone shear flows \cites{CDLZ23,LWZ20Kolmo}. We also refer to \cite{CDLZ23}*{page 3} for a quick overview of the state of the art.

\paragraph{\emph{The inhomogeneous setting}}
In the inviscid $2d$ setting, early linear studies date back to Hartman \cite{Hartman}, with more recent work addressing the linear stability of stably stratified Couette flow \cites{BCZD22,CZN23strip,CZN23chan,YL18}. For a nonlinear result giving an extended time-span of stability, see \cite{BBCZD23}. One of the distinctive features of the interaction between shearing and stratification in $2d$ is an oscillatory coupling that induces an instability that slows down inviscid damping rates. This is captured in \cites{BCZD22,BBCZD23} thanks to specific \emph{symmetric variables}, which play a key role in the present article as well (see \eqref{eq:SimVar} below). In the presence of dissipation, a nonlinear stability threshold was established for the $2d$ stably stratified Couette flow in \cites{ZZ23,Zillinger21bouss}.

In $3d$, the lift-up effect is suppressed by the coupling with the temperature equation, as can already be seen in the linearized dynamics of the zero modes \cite{CZDZ23}, the so-called streaks. Although this was already noticed in other coupled systems (such as in the MHD equations \cite{L18}), the suppression of lift-up does not directly imply a quantitative improvement of the stability threshold over the Navier-Stokes setting. Indeed, in the corresponding MHD problem it is still only known to be at least of order $\nu$, as in the case of homogeneous Navier-Stokes \cite{BGM20}. It is thus one of the key aspects of our result (Theorem \ref{thm:transitionthreshold}) that not only a transition threshold is established, but that it improves over that of the corresponding Navier-Stokes setting (see \eqref{eq:transthresh1}, \eqref{eq:transthresh2}). The decisive novelty to obtain such a quantitative improvement is the use of dispersive mechanisms in the nonlinear analysis.

\paragraph{\emph{The role of stable stratification and dispersive effects}}
From a (geo-) physical viewpoint, a suitable \emph{stable} stratification is widely considered a stabilizing mechanism: deviations from such a configuration are subject to a restoring force due to gravity. This is best understood without background flow, i.e.\ near $v=0$. Here the linearization of \eqref{eq:3DBoussinesq} yields a constant coefficient system, which in the inviscid case features purely dispersive (in $2d$, see \cite{EW15}) resp.\ stationary and dispersive behaviors (in $3d$, see \cite{W19}). This highlights a parallel with rotational effects, another key aspect of many geophysical flows, as both stratification and rotation typically feature anisotropic, degenerate dispersion relations.\footnote{There is a vast literature on many related models, see e.g.\ \cite{GS2007} for an overview. In terms of dispersion relations, a particularly close connection exists with the $3d$ Euler equations near a rigid rotation, a swirling configuration which was recently shown to be nonlinearly stable \cites{GHPW21,GPW23} in axisymmetry.}
On the full space $\RR^2$ resp.\ $\RR^3$, this leads to amplitude decay. The role of the Richardson number $\beta^2$ can then be simply traced as that of a large parameter, which increases the strength of the dispersive effects. In the presence of viscosity, it is moreover possible to show that a perturbation of the stable stratification without background shear will give rise to a global solution of \eqref{eq:3DBoussinesq} provided $\beta$ is sufficiently large \cites{T21,TS17}.\footnote{This proceeds in the same spirit as works on fast rotation in the Navier-Stokes equations with Coriolis force, see e.g.\ \cites{CDGG06,GRM09}, and also related geophysical settings \cite{GIQT22}.}

With Theorem \ref{thm:transitionthreshold}, we provide a quantitative validation of the aforementioned physical intuition that stratification helps to stabilize a flow also in the presence of a shearing background in $3d$. The latter leads to structural changes in the system (see also the discussion of our proof below) and introduces variable coefficients, so that it is not straightforward to understand what of the dispersive effects survives. In particular, it is worth mentioning that in the analogous $2d$ situation stabilization is not at all the case: when $\nu=0$, the interplay of stratification and shearing generates an instability mechanism that weakens inviscid damping rates \cites{CZN23strip,CZN23chan,BBCZD23,BCZD22}, and there does not seem to be room for dispersive effects to improve this. Moreover, for $\nu>0$ the best known stability threshold is of order $\nu^\frac12$ \cite{ZZ23}, and thereby significantly smaller than that of order $\nu^\frac13$ proved for the homogeneous case \cite{MZ22}.

\subsubsection*{About the proof of Theorem \ref{thm:transitionthreshold}}
To give a rough idea of the ideas and techniques underlying Theorem \ref{thm:transitionthreshold}, we review here the dynamics it captures. 
We refer to Sections \ref{ssec:unknowns} and \ref{ssec:proofsetup} for the details regarding the choice of unknowns, the setup of our main bootstrap argument and a more elaborate overview of the proof. The proper choice of unknowns also includes translation to a moving frame -- for the sake of simplicity, we ignore this more delicate point in this preliminary presentation.

\emph{Linearized Dynamics.} Friction forces act (via kinematic viscosity/thermal diffusivity) on the full system \eqref{eq:3DBoussinesqPert} isotropically, with strength proportional to $\nu$, and are relevant on the time scale $O(\nu^{-1})$, characteristic of the heat equation. In addition, $u^1$ and $u^3$ are forced by $u^2$ and $\theta$, while $u^2$ and $\theta$ are coupled in an oscillatory fashion. As already observed in \cite{CZDZ23} (see also \cite{L18} for the related MHD case), this suppresses the classical lift-up instability mechanism in the Navier-Stokes equations near the Couette flow. Moreover, the same oscillatory coupling between $\theta$ and $u^2$ also reflects a restoring mechanism due to internal gravity waves: at its core, this is a dispersive mechanism with zero-homogeneous dispersion relation $\frac{\abs{k,l}}{\abs{k,\eta,l}}$, where $(k,\eta,l)\in\ZZ\times\RR\times\ZZ$ are the Fourier variables on $\TT\times\RR\times\TT$ (see also \cite{W19} for the corresponding analysis in the inviscid case without background shear). However, this mechanism is witnessed here in general in a moving frame (and thus with time dependent coefficients),  which makes it difficult to exploit.

The dominant dynamics are then as follows:
\begin{enumerate}
    \item (Nonzero modes $k\neq 0$) The effect of friction is enhanced by the shearing of the background Couette flow, which leads to the enhanced dissipation \eqref{eq:basicbounds} of the nonzero (i.e.\ $x$-dependent) modes in the system. This effect is relevant on a time scale $O(\nu^{-\frac13})\ll O(\nu^{-1})$. It is easily witnessed for $u^1_{\neq}$ and $u^3_{\neq}$, but due to the oscillatory coupling between $u^2_{\neq}$ and $\theta_{\neq}$ it is only apparent in suitable \emph{symmetric variables} $\GG,\Gamma$ (see \eqref{eq:SimVar}) replacing $u^2_{\neq},\theta_{\neq}$.\footnote{These symmetric variables go back to at least \cites{BCZD22,BBCZD23}, in the $2d$ inviscid setting, and have also been used in prior work \cite{CZDZ23} on the linear dynamics in \eqref{eq:3DBoussinesqPert}.} 
    
    \item (Simple zero modes $k=0$, $l\neq $0) The zero (i.e.\ $x$-independent) modes do not witness a dissipation enhancement. However, the simple zero modes form a constant coefficient system with three degrees of freedom (due to incompressibility), which features one purely dissipating mode (closely related to $u^1_0$), and two modes which are also dispersing via internal gravity waves on a time scale $O(1)$. The latter are bona fide dispersive waves with dispersion relation $\frac{\abs{l}}{\abs{\eta,l}}$, and closely related to the symmetric variables. In particular, this leads to amplitude decay of the simple zero modes $\widetilde{u}^2_0, \widetilde{u}^3_0$ and $\widetilde{\theta}_0$
    as in Proposition \ref{prop:disp_est} (compare the first term on the right-hand side of \eqref{eq:dispersion}).
    
    \item (Double zero modes $k=l=0$) The double zero modes, depending neither on $x$ nor on $z$, simply obey a linear heat equation on $\RR$. (Their nonlinear interactions motivate our assumption \eqref{eq:zero-mean-xz-id}.)
\end{enumerate}

\emph{Nonlinear Behavior.}
The nonlinear interactions determine the transition threshold we find for the full system \eqref{eq:3DBoussinesqPert}. This is established in a perturbative spirit from the linear dynamics via a nonlinear bootstrap argument -- see Theorem \ref{thm:bootstrap_step}. Hereby, a precise analysis of the quadratically nonlinear interactions is indispensable. It relies on some key features, listed here by the type of outputs they produce:
\begin{enumerate}
    \item (Nonzero modes) Nonlinear interactions leading to a nonzero mode output need to include at least one nonzero mode. As a consequence, enhanced dissipation of such modes can be propagated with a comparatively large transition threshold of $\nu^{\frac56}$ (see Propositions \ref{prop:GGamma} and \ref{prop:U13}). This is relatively direct for $u^1_{\neq}$ and $u^3_{\neq}$, but relies on a specific energy functional that uses the skew-symmetric structure of the coupling of the modes $u^2_{\neq}$ and $\theta_{\neq}$ in the symmetric variables. The overall approach is similar in spirit to prior works \cites{BGM17,WZ21,L18}.

    \item (Simple zero modes $k=0$, $l\neq 0$) Obtaining a threshold that is quantitatively larger than $\nu$ also for the simple zero modes uses on the one hand that $u^1_0$ does not force any zero modes, and on the other hand requires the propagation of dispersive features on $\widetilde{u}^2_0,\widetilde{u}^3_0,\widetilde{\theta}_0$ (see Proposition \ref{prop:LinftyGGamma}). In particular, we establish a decomposition of $\widetilde{u}^2_0$ into a piece (emanating from the initial data) that decays in amplitude and a nonlinear contribution, which is comparatively small (see Lemma \ref{lemma:innl_decomp}), and use that $\widetilde{u}^2_0$ and $\widetilde{u}^3_0$ are connected by incompressibility. Several further structural features of the Euler nonlinearity then play an important role, e.g.\ that $\overline{u}^2_0=0$ and that $\overline{u}^3_0$ cannot force simple zero modes. However, nonlinear interactions including double zero modes exist, and it is from those that the weakest control on $\widetilde{u}^2_0,\widetilde{u}^3_0$ and $\widetilde{\theta}_0$ with a threshold of $\max\{\nu^{\frac56},\beta^{\frac13}\nu^{\frac89}\}$ derives.\footnote{It is here that assumption \eqref{eq:zero-mean-xz-id} is used to prevent exceedingly large contributions.} In all of this, the non-dispersing contribution gives the weakest bound for $\widetilde{u}^1_0$ (which fortunately is only a passive dynamic), and thus the largest contribution $\max\{\nu^{\frac89},\beta^{\frac13}\nu^{\frac{11}{12}}\}$ to the threshold -- see Proposition \ref{prop:zero}.

     \item (Double zero modes) Due to the comparatively slow effect of the linear heat equation dynamic, the nonlinear control of double zero modes crucially relies on the absence of self-interactions of such modes, which are avoided by the Euler nonlinearity. Under assumption \eqref{eq:zero-mean-xz-id} (which can be quantitatively relaxed), favorable bounds (as they are needed for the simple zero modes) on $\overline{u}^3_0$ and $\overline{u}^1_0$ can be established. As in the simple zero modes, the double zero mode $\overline{u}^1_0$ is only transported, leading to a similar threshold as in the simple zero case -- see Proposition \ref{prop:doublezero}.

     \item (The role of $\beta$) As can be seen from these arguments, the role of $\beta$ is to increase the strength of the dispersive effect, thus allowing for a larger threshold if $\beta$ is sufficiently large.
\end{enumerate}

\subsection{Structure of the equations and choice of unknowns}\label{ssec:unknowns}
The linearized transport structure of \eqref{eq:3DBoussinesqPert} makes it
convenient to adopt the  change variables 
\begin{equation}\label{eq:changelin}
    (x,y,z)\mapsto (x-yt,y,z).
\end{equation}
In this way, the differential operators become
\begin{equation}\label{eq:nablaL}
    \nabla=(\de_x,\de_y,\de_z)\mapsto \nabla_L:=(\de_x,\de_y^L,\de_z), \qquad \Delta\mapsto \Delta_L
\end{equation}
where
\begin{equation}
    \de_y^L:= \de_y -t\de_x, \qquad \Delta_L:= \de^2_x+\left(\de_y^L\right)^2+\de_z^2
\end{equation}
As a convention, a function $f$ in the original coordinates will be capitalized to $F$ in the moving frame \eqref{eq:changelin}. We will write $\nabla_{x,z}=(\de_x,\de_z)$ to emphasize the gradient with respect to $(x,z)$ only. In particular, we have that $\nabla_L\cdot U=\nabla\cdot u=0$, and from this it follows that $\iint_{\TT\times\TT}U^2\dd x\dd z$ is a constant, integrable function of $y\in\RR$, and thus
\begin{equation}\label{eq:xz-meanU2}
 \overline{U}^2_0(y)=\iint_{\TT\times\TT}U^2(x,y,z)\dd x\dd z=0.
\end{equation}
As noted in \cite{CZDZ23}, the coupling between $u^2$ and $\theta$ is conveniently captured at the linearized level by the symmetric variables 
\begin{equation}\label{eq:SimVar}
 \GG:=-|\nabla_{x,z}|^{-\frac12}|\nabla_L|^{\frac32}U^2, \qquad  \Gamma:=|\nabla_{x,z}|^{\frac12}|\nabla_L|^{\frac12}\Theta.
\end{equation}
We highlight that $\GG$ is well-defined by \eqref{eq:xz-meanU2}, and by construction both $\GG$ and $\Gamma$ have zero mean over $(x,z)\in\TT\times\TT$. These new variables play a crucial role in the analysis of the linearized system, as they admit a favorable energy structure (see the proof of Proposition \ref{prop:GGamma}).

Under the change of coordinates \eqref{eq:changelin} and the change of variables \eqref{eq:SimVar}, we  obtain from \eqref{eq:3DBoussinesqPert} that
\begin{subequations}\label{eq:UGGamma}
\begin{alignat}{4}
\de_tU^1&=\nu\Delta_LU^1-U^2 +\de_xP  + \cT(U,U^1) +\de_x\cP(U,U), \\
\de_tU^3&=\nu\Delta_LU^3 +\de_zP + \cT(U,U^3) +\de_z\cP(U,U),   \\
\de_t\GG &=\nu\Delta_L \GG +\frac{1}{2}\de_x\de_y^L|\nabla_L|^{-2}\GG +\beta |\nabla_{x,z}||\nabla_L|^{-1}\Gamma + \cT_\star(U,U^2)+\de_y^L\cP_\star(U,U),\\
\de_t\Gamma & = \nu\Delta_L\Gamma-\frac{1}{2}\de_x\de_y^L|\nabla_L|^{-2}\Gamma -\beta |\nabla_{x,z}||\nabla_L|^{-1}\GG+\cT_\circ (U,\Theta).
\end{alignat}
\label{eq:full_shorthand}
\end{subequations}
However, $\overline{\Theta}_0=\iint_{\TT\times\TT}\Theta \dd x\dd z$ is not conserved, and thus \eqref{eq:UGGamma} needs to be complemented with
\begin{equation}
\de_t \overline{\Theta}_0+\de_y(\overline{U^2\Theta})_0=\nu\de_{yy}\overline{\Theta}_0,
\end{equation}
to ensure equivalence with \eqref{eq:3DBoussinesqPert}.
In \eqref{eq:UGGamma}, the linear part of the pressure is denoted by 
\begin{equation}
P:=-2\de_x|\nabla_L|^{-2}U^2-\beta\de_y^L|\nabla_L|^{-2}\Theta,
\end{equation}
% For a vector-valued function $U$ and a scalar function $F$, the bilinear forms defining the transport nonlinearities as
% \begin{align}
%     &\cT(U,F)=-(U\cdot\nabla_L F)  \\
%     & \cT_\star(U,F)=|\nabla_{x,z}|^{-1/2}|\nabla_L|^{\frac32}\cT(U,F)\\
%     &\cT_\circ (U,F)=|\nabla_{x,z}|^{1/2}|\nabla_L|^{1/2}\cT(U,F),
% \end{align}
% while those corresponding to nonlinear pressure terms are defined for vector-valued functions $U,V$ as
% \begin{align}
%     &\cP(U,V)=-|\nabla_L|^{-2}(\nabla_L \otimes \nabla_L) (U\otimes V),\\    
%     &\cP_\star(U,V)= |\nabla_{x,z}|^{-1/2}|\nabla_L|^{\frac32}\cP(U,V).
% \end{align}
% 
% 
and for  a scalar function $F$, the bilinear forms defining the transport and pressure nonlinearities are written as
\begin{equation}
    \cT(U,F)=-(U\cdot\nabla_L F), \qquad 
    \cP(U,U)=-|\nabla_L|^{-2}(\nabla_L \otimes \nabla_L) (U\otimes U),  
\end{equation}
with modified transport terms
\begin{equation}
\cT_\star(U,F)=-|\nabla_{x,z}|^{-\frac12}|\nabla_L|^{\frac32}\cT(U,F),\qquad 
\cT_\circ (U,F)=|\nabla_{x,z}|^{\frac12}|\nabla_L|^{\frac12}\cT(U,F),
\end{equation}
and modified pressure 
\begin{equation}
    \cP_\star(U,U)= |\nabla_{x,z}|^{-\frac12}|\nabla_L|^{\frac32}\cP(U,U).
\end{equation}
In particular, with \eqref{eq:xz-meanU2} $U^2$ and $\Theta$ can be recovered from $\GG$ and $\Gamma$ as
\begin{equation}\label{eq:U2Theta_recover}
 U^2=-\abs{\nabla_{x,z}}^{\frac12}\abs{\nabla_L}^{-\frac32}\GG,\qquad \Theta=\abs{\nabla_{x,z}}^{-\frac12}\abs{\nabla_L}^{-\frac12}\Gamma+\overline{\Theta}_0.   
\end{equation}

\subsection{Setup and overview of the proof of Theorem \ref{thm:transitionthreshold}}\label{ssec:proofsetup}
We will work with energies that are defined through Fourier multipliers. The main multiplier $\cA $ combines regularity, time decay and \emph{ghost weights} $\cM_j$ ($1\leq j\leq 3$), and is of the form
\begin{equation}\label{def:cA}
    \cA =\e^{\lambda\nu^{\frac13}t}\l\nabla\r^{2m} \cM,\qquad \cM=\prod_{j=1}^3\cM_j,\qquad \lambda=\lambda(\beta):=\frac{2\beta-1}{2\beta +1}.
\end{equation}
See Section \ref{ssec:neq_prelim} for further details. 
We will show the following bootstrap result:
\begin{theorem}[Bootstrap step]\label{thm:bootstrap_step}
Under the hypothesis of Theorem \ref{thm:transitionthreshold}, assume that for some $T>0$ we have the following bounds for $t\in [0,T]$, where $\eps=C_\beta^{-1}\eps_0$ with $C_\beta:=\sqrt{\lambda(\beta)}=\sqrt{\frac{2\beta-1}{2\beta+1}}$ and $C_0\geq 10^4$ :
    \begin{align}
    \norm{\cA \GG_{\neq}}^2_{L^\infty_tL^2}+\nu\norm{ \nabla_L \cA\GG_{\neq}}^2_{L^2_tL^2}+\norm{\sqrt{-\frac{\dot{\cM}}{\cM}}\cA \GG_{\neq}}^2_{L^2_tL^2}  &\leq 100\varepsilon^2, \label{eq:bootstrap_nonzero_G}\\
    \norm{\cA \Gamma_{\neq}}^2_{L^\infty_tL^2}+\nu\norm{\nabla_L\cA  \Gamma_{\neq}}^2_{L^2_tL^2}+\norm{\sqrt{-\frac{\dot{\cM}}{\cM}}\cA \Gamma_{\neq}}^2_{L^2_tL^2}  &\leq 100\varepsilon^2,\label{eq:bootstrap_nonzero_Gamma}\\
    \norm{\cA  U^j_{\neq}}^2_{L^\infty_tL^2}+\nu\norm{\nabla_L\cA  U^j_{\neq}}^2_{L^2_{t}L^2}+\norm{\sqrt{-\frac{\dot{\cM}}{\cM}}\cA  U^j_{\neq}}^2_{L^2_tL^2}  &\leq 100 C_0^2\varepsilon^2, \qquad j\in\{1,3\},\label{eq:bootstrap_nonzero_U}
\end{align}
and 
\begin{align}
    \norm{\GG_0}^2_{L^\infty_tH^{2m}}+\nu\norm{\nabla\GG_0}^2_{L^2_tH^{2m}} &\leq 100\varepsilon^2,\label{eq:bootstrap_zero_g}\\
    \norm{\Gamma_0}^2_{L^\infty_tH^{2m}}+\nu\norm{\nabla\Gamma_0}^2_{L^2_tH^{2m}} &\leq 100\varepsilon^2,\label{eq:bootstrap_zero_gamma}\\
    \norm{\overline{\Theta}_0}^2_{L^\infty_tH^{2m+1}}+\nu\norm{\de_y\overline{\Theta}_0}^2_{L^2_tH^{2m+1}} &\leq 100\varepsilon^2,\label{eq:bootstrap_doublezero_Theta}\\
    % &\norm{\overline{U}^3_0}^2_{L^\infty_tH^{2m}}+\nu\norm{\de_y\overline{U}_0^3}^2_{L^2_tH^{2m}} \leq 100\varepsilon^2,\label{eq:bootstrap_doublezero_U3}\\
    \norm{U^r_0}^2_{L^\infty_tH^{2m}}+\nu\norm{\nabla U^r_0}^2_{L^2_tH^{2m}} &\leq 100C_0^2\varepsilon^2, \qquad r\in\{1,3\},\label{eq:bootstrap_zero_u}
    % &\norm{V_0}^2_{L^\infty_tH^{2m}}+\nu\norm{\nabla V_0}^2_{L^2_tH^{2m}} \leq 100C_0^2\varepsilon^2.\label{eq:bootstrap_zero_v}   
\end{align}
Then \eqref{eq:bootstrap_nonzero_G}--\eqref{eq:bootstrap_zero_u} hold in $[0,T]$, with $100$ replaced by $50$.
\end{theorem}

The proof of Theorem \ref{thm:transitionthreshold} follows directly from this theorem: By standard local well-posedness we can assume that there exists $T>0$ such that the assumptions \eqref{eq:bootstrap_nonzero_G}--\eqref{eq:bootstrap_zero_u} hold on $[0,T]$. Theorem \ref{thm:bootstrap_step} and continuity of the norms then imply that the set of times on which \eqref{eq:bootstrap_nonzero_G}--\eqref{eq:bootstrap_zero_u} holds is closed, open and non-empty in $[0,\infty)$. This implies Theorem \ref{thm:transitionthreshold}, upon also collecting the dispersive bounds from Proposition \ref{prop:LinftyGGamma}.

\begin{proof}[Proof of Theorem \ref{thm:bootstrap_step}]
 The below Propositions \ref{prop:GGamma}--\ref{prop:doublezero} combine to give the claim.
\end{proof}

\begin{proposition}[Proposition for $G$, $\Gamma$]\label{prop:GGamma}
    There exists $C_1>0$ such that under the bootstrap assumptions \eqref{eq:bootstrap_nonzero_G}--\eqref{eq:bootstrap_zero_u} there holds that
    \begin{align}
    \norm{\cA \GG_{\neq}}^2_{L^\infty_tL^2}+\nu\norm{ \nabla_L \cA\GG_{\neq}}^2_{L^2_tL^2}+\norm{\sqrt{-\frac{\dot{\cM}}{\cM}}\cA \GG_{\neq}}^2_{L^2_tL^2} &\leq C_\beta^{-2}\eps_0^2+C_1(C_\beta^{-2}\nu^{-\frac56}\eps)\eps^2,\\%\leq 3\varepsilon^2+C_1(\nu^{-\frac56}\varepsilon)\cdot \varepsilon^2, \\
    % &\blue{\leq C_\beta\eps_0^2+C_\beta C_1\nu^{-\frac56}\eps^3}\\
    \norm{\cA \Gamma_{\neq}}^2_{L^\infty_tL^2}+\nu\norm{\nabla_L \cA \Gamma_{\neq}}^2_{L^2_tL^2}+\norm{\sqrt{-\frac{\dot{\cM}}{\cM}}\cA \Gamma_{\neq}}^2_{L^2_tL^2}  &\leq C_\beta^{-2}\eps_0^2+C_1(C_\beta^{-2}\nu^{-\frac56}\eps)\eps^2,
    %3\varepsilon^2+C_1(\nu^{-\frac56}\varepsilon)\cdot \varepsilon^2,
\end{align}
\end{proposition}

\begin{proof}[About the proof of Proposition \ref{prop:GGamma}]
We exploit the symmetric structure of the equations for the symmetrized variables $G$ and $\Gamma$ through a combined energy functional 
\begin{equation}\label{eq:GGamma_energy}
    \sfE(t)=\frac12\left[\|\cA \GG_{\neq}(t)\|^2+\|\cA \Gamma_{\neq}(t)\|^2+\frac{1}{\beta}\l\de_x\de_y^L|\nabla_{x,z}|^{-1}|\nabla_L|^{-1}\cA \GG_{\neq}(t) ,\cA \Gamma_{\neq}(t)\r\right].
\end{equation}
Here the Fourier multiplier $\cA$ encodes time decay and spatial regularity (see also the discussion in Section \ref{ssec:neq_prelim}). We highlight the coercivity of $\sfE$ for $\beta>\frac12$ in the sense that
\begin{equation}\label{eq:coercive-L2}
\frac12\left(1-\frac{1}{2\beta}\right)\left[\|\cA \GG_{\neq}(t)\|^2+\|\cA \Gamma_{\neq}(t)\|^2\right] \leq \sfE(t) \leq\frac12\left(1+\frac{1}{2\beta}\right)\left[\|\cA \GG_{\neq}(t)\|^2+\|\cA \Gamma_{\neq}(t)\|^2\right].
\end{equation}
From \eqref{eq:coercive-L2} we then obtain the claim by establishing suitable bounds on the time derivative of $\sfE(t)$. The multiplier $\cA$ is hereby crafted to absorb certain linear terms, which allows us to reduce to estimating trilinear expressions -- see Lemma \ref{lem:GGamma_reduction}. These arise naturally from energy estimates and are of the form
\begin{equation}
 \l\cA F_{\neq},\cA\cB(U,H)\r,\qquad \cB\in\{\cT_\star,\cT_\circ,\cP_\star\},\quad F\in\{\GG,\Gamma\},\quad H\in\{U,\Theta\},
\end{equation}
expressed in terms of the variables in \eqref{eq:UGGamma}. In particular, we note that the quadratic nonlinear terms must always involve at least one nonzero mode, i.e.\
\begin{equation}
 \cB(U,H)_{\neq}=\cB(U_{\neq},H_0)_{\neq}+\cB(U_0,H_{\neq})_{\neq} +\cB(U_{\neq},H_{\neq})_{\neq}. 
\end{equation}
To overcome the derivative loss in these nonlinearities, we utilize the energy structure and dissipation, distributing derivatives in $\cA$ and symmetric variable multipliers across the three terms to ensure that the final bound relies on dissipative estimates to the smallest possible extent. 
%includes only bootstrap assumption terms from Theorem \ref{thm:bootstrap_step}.
 %\blue{Hereby the largest contributions (in $\nu$) arise when the dissipative bounds $\norm{\nabla_L\cA F_{\neq}}_{L^2_tL^2}\lesssim\nu^{-\frac12}\eps$, $F\in\{U^1,U^3,\GG,\Gamma\}$, are invoked.} 
 The bound for the threshold then follows by tracing powers of $\nu$ needed for the various terms that appear: from the bootstrap assumptions it is clear that bounds in $L^2_tL^2$ with maximal order of derivatives incur a loss of $\nu^{-\frac12}$ for the dissipative contributions, whereas for one order of derivatives less enhanced dissipation yields $L^2_tL^2$ bounds of order $\nu^{-\frac16}$ (see Corollary \ref{cor:ED_bounds}), and terms involving only $L^\infty_t$ norms and ghost multipliers are uniformly bounded in $\nu$.
To give an example, a simple such bound appearing after suitably distributing derivatives and the multipliers is 
\begin{equation}
    \int_0^\infty\norm{\cA F_{\neq}}\norm{\cA \GG_{\neq}}\norm{\nabla_L\cA \GG_{\neq}}\lesssim \norm{\cA F_{\neq}}_{L^\infty_tL^2}\norm{\cA \GG_{\neq}}_{L^2_tL^2}\norm{\nabla_L\cA \GG_{\neq}}_{L^2_tL^2}\lesssim \nu^{-\frac23}\eps^3,
\end{equation}
where $F\in\{\GG,\Gamma\}$.
We direct the reader to Section \ref{ssec:GGamma_neq} for the detailed proof of the proposition.
\end{proof}

\begin{proposition}[Proposition for $U^r$, $r=1,3$]\label{prop:U13}
    There exists $C_2>0$ such that under the bootstrap assumptions \eqref{eq:bootstrap_nonzero_G}--\eqref{eq:bootstrap_zero_u} there holds that
\begin{align}
    &\norm{\cA U^r_{\neq}}^2_{L^\infty_tL^2}+\nu\norm{\nabla_L \cA U^r_{\neq}}^2_{L^2_tL^2}+\norm{\sqrt{-\frac{\dot{\cM}}{\cM}}\cA U^r_{\neq}}^2_{L^2_tL^2} \leq 4C_0^2\eps_0^2+C_2(\nu^{-\frac23}\varepsilon) \varepsilon^2.
\end{align}
\end{proposition}
The proof of this proposition proceeds directly via energy estimates and trilinear bounds, for which techniques as in the proof of Proposition \ref{prop:GGamma} are employed -- see Section \ref{ssec:U13neq}.

\begin{proposition}[Proposition for zero modes]\label{prop:zero}
 There exists $C_3>0$ such that under the bootstrap assumptions \eqref{eq:bootstrap_nonzero_G}--\eqref{eq:bootstrap_zero_u} there holds that
 \begin{align}
    \norm{\GG_0}^2_{L^\infty_tH^{2m}}+\nu\norm{\nabla\GG_0}^2_{L^2_tH^{2m}} &\leq \eps_0^2+C_3(\nu^{-\frac56}\eps+\beta^{-\frac13}\nu^{-\frac89}\eps)\eps^2,\\
    \norm{\Gamma_0}^2_{L^\infty_tH^{2m}}+\nu\norm{\nabla\Gamma_0}^2_{L^2_tH^{2m}} &\leq \eps_0^2+C_3(\nu^{-\frac56}\eps+\beta^{-\frac13}\nu^{-\frac89}\eps)\eps^2,\\
    \norm{\widetilde{U}^3_0}^2_{L^\infty_tH^{2m}}+\nu\norm{\nabla \widetilde{U}^3_0}^2_{L^2_tH^{2m}} &\leq \eps_0^2+C_3(\nu^{-\frac56}\eps+\beta^{-\frac13}\nu^{-\frac89}\eps)\eps^2,\\
    \norm{\widetilde{U}^1_0}^2_{L^\infty_tH^{2m}}+\nu\norm{\nabla \widetilde{U}^1_0}^2_{L^2_tH^{2m}} &\leq \eps_0^2+C_3(\nu^{-\frac89}\eps+\beta^{-\frac13}\nu^{-\frac{11}{12}}\eps)\eps^2.
\end{align}
\end{proposition}

\begin{proposition}[Proposition for double zero modes]\label{prop:doublezero}
  There exists $C_4>0$ such that under the bootstrap assumptions \eqref{eq:bootstrap_nonzero_G}--\eqref{eq:bootstrap_zero_u} there holds that
  \begin{align}
   \norm{\overline{\Theta}_0}^2_{L^\infty_tH^{2m+1}}+\nu\norm{\de_y\overline{\Theta}_0}^2_{L^2_tH^{2m+1}} &\leq \eps^2 (\nu^{-\frac23}\eps+\beta^{-\frac13}\nu^{-\frac56}\eps)^2,\\
   %\eps_0^2+C_4(\nu^{-\frac89}\eps)\eps^2,\\
    \norm{\overline{U}^3_0}^2_{L^\infty_tH^{2m}}+\nu\norm{\de_y\overline{U}_0^3}^2_{L^2_tH^{2m}} &\leq \eps^2 (\nu^{-\frac23}\eps+\beta^{-\frac13}\nu^{-\frac56}\eps)^2,\\
   %\eps_0^2+C_4(\nu^{-\frac89}\eps)\eps^2,\\
    \norm{\overline{U}^1_0}^2_{L^\infty_tH^{2m}}+\nu\norm{\de_y\overline{U}_0^1}^2_{L^2_tH^{2m}} &\leq \eps^2 (\nu^{-\frac89}\eps+\beta^{-\frac13}\nu^{-\frac{11}{12}}\eps)^2.
   %\eps_0^2+C_4(\nu^{-\frac{11}{12}}\eps)\eps^2,
  \end{align}
\end{proposition}
\begin{proof}[About the proof of Propositions \ref{prop:zero} -- \ref{prop:doublezero}]
From \eqref{eq:UGGamma} we see that the dynamics of the (simple) zero modes are governed by 
% \begin{subequations}\label{eq:zero_modes}
% \begin{alignat}{4}
% \de_tU^1_0&=\nu\Delta U^1_0-U^2_0 + \cT(U,U^1)_0, \\
% \de_tU^2_0&=\nu\Delta U^2_0-\beta\de_z^2\abs{\nabla_{y,z}}^{-1}\Theta_0 + \cT(U,U^2)_0, \\
% \de_tU^3_0&=\nu\Delta U^3_0 + \beta\de_y\de_z\abs{\nabla_{y,z}}^{-1}\Theta_0 + \cT(U,U^3)_0,   \\
% \de_t\Theta_0&=\nu\Delta \Theta_0 +\beta U^2_0+ \cT(U,\Theta)_0.
% \end{alignat}
% \end{subequations}
\begin{subequations}\label{eq:zero_mode_eqs}
\begin{alignat}{4}
\de_tU^1_0&=\nu\Delta U^1_0-U^2_0 + \cT(U,U^1)_0, \\
\de_tU^3_0&=\nu\Delta U^3_0 +\de_zP_0 + \cT(U,U^3)_0 +\de_z\cP(U,U)_0,   \\
\de_t\GG_0 &=\nu\Delta \GG_0 +\beta |\de_z||\nabla_{y,z}|^{-1}\Gamma_0 + \cT_\star(U,U^2)_0+\de_y\cP_\star(U,U)_0,\\
\de_t\Gamma_0 & = \nu\Delta\Gamma_0 -\beta |\de_z||\nabla_{y,z}|^{-1}\GG_0+\cT_\circ (U,\Theta)_0.
\end{alignat}
\end{subequations}
At the linearized level, this is a constant coefficient system which can easily be diagonalized using the Fourier transform (see also \cite{W19}). Up to dissipation (i.e.\ formally setting $\nu=0$), one finds the following picture for the simple zero dynamics: there are two zero eigenvalues, and two dispersive modes with dispersion relation $\pm i\frac{\abs{l}}{\abs{\eta,l}}$, where $(\eta,l)\in\RR\times\ZZ$ are the Fourier variables on $\RR\times\TT$. Incompressibility reduces the dimension of the zero eigenvalue space to a single mode, the amplitude of which turns out to be given by the function 
\begin{equation}\label{eq:def_V0}
    V_0(t):=U^1_0(t)+\beta^{-1}\Theta_0(t).
\end{equation}
The two dispersive modes combine oscillations in all components of the system, and have amplitude given by (rescalings of) $\GG_0\pm i\Gamma_0$, i.e.\ combinations of the symmetric variables. Due to the structure of the associated eigenvectors, the simple zero mode $\widetilde{U}^3_0$ is naturally recovered by incompressibility, namely $\widetilde{U}^3_0=-\de_z^{-1}\de_yU^2_0$ (where we recall that by \eqref{eq:xz-meanU2} we have $\overline{U}^2_0=0$).

Due to the degenerate nature of the dispersion relation, the amplitude decay it entails is rather weak: for the linearized dynamics we obtain an $L^\infty$ decay rate of only $t^{-\frac13}$ (see Proposition \ref{prop:disp_est}). We do not propagate this decay in the nonlinear problem, but instead rely on its interplay with the heat equation dynamic to yield improved bounds for amplitudes (see Proposition \ref{prop:LinftyGGamma}). Since these are to be used in order to improve the threshold for other nonlinear terms, we state them in the original variables $U^2_0,\widetilde{U}^3_0,\widetilde{\Theta}_0$. The most critical piece hereby is a decomposition of $U^2_0$, at the highest level of derivatives (order $2m$) in our norms, into a piece that tracks the evolution of the initial data in $L^\infty$, and an $H^{2m}$ estimate for the nonlinear contributions, which improves (thanks to dispersion) over the obvious bounds -- see Lemma \ref{lemma:innl_decomp}. This is vital for obtaining a large threshold for the double zero modes, which are only forced by $U^2_0$ (see the discussion below).

To obtain the claimed control on the simple zero modes, we combine these ideas with the techniques developed in the context of Section \ref{sec:neqs} for the nonzero modes. In particular, we use the symmetric energy structure to estimate $\GG_0$ and $\Gamma_0$ via trilinear estimates. As a general rule, whenever an interaction involves a nonzero mode this leads to strong bounds. All other interactions involve at least one simple zero mode (in particular, $U^1_0$ does not force any zero modes), so that we can appeal to their dispersive features. The most delicate terms involve the interaction of double zero with simple zero modes. By incompressibility, control of $\widetilde{U}^3_0$ follows from that of $U^2_0$. Finally, as is natural from the above discussion, we track $U^1_0$ through $V_0$ from \eqref{eq:def_V0}. As one checks directly, this plays the role of a passive scalar, being dissipated and advected by $U$ -- see \eqref{eq:V_0_dynamics}.

\smallskip
On the other hand, the dynamics of the double zero modes are given by
\begin{equation}\label{eq:doublezero_mode_eqs}
\de_t \overline{F}_0+\de_y(\overline{U^2F})_0=\nu\de_{yy}\overline{F}_0,\qquad F\in\{U^1,U^3,\Theta\}.
\end{equation}
Two points are worth emphasizing: self-interactions are absent in this system, and the double zero modes are exclusively forced by $U^2$ (for which we observe dispersive decay and a favorable decomposition of the zero mode, Proposition \ref{prop:LinftyGGamma} resp.\ Lemma \ref{lemma:innl_decomp}). Here the vanishing condition \eqref{eq:zero-mean-xz-id} (or its less stringent form \eqref{eq:double0id}, or a suitably quantified relaxation thereof) ensures additional smallness of the evolution of $\overline{U}^3_0$ and $\overline{\Theta}_0$. This is vital in order to achieve a threshold larger than $\nu$: to first order the initial data of the double zero modes are simply propagated by the heat equation on $\RR$, which in our setting of symmetric energy estimates would lead to an excessively large contribution to the zero modes (leading to a threshold of order $\nu$). This point is mute for $\overline{U}^1_0$, since it does not force any zero or double zero modes.
\end{proof}

\subsection{Notation}\label{ssec:notation}
We will write $a\lesssim b$ to mean that there exists a constant $C>0$ (independent of other parameters of relevance) such that $a\leq Cb$.

The Fourier transform of a function $\varphi(x,y,z)$ will be denoted as follows: for $(k,\eta,l)\in \mathbb{Z}\times\mathbb{R}\times\mathbb{Z}$, define
\begin{equation}
     \hat\varphi_{k,l}(\eta):=\frac{1}{4\pi^2}\int_{\mathbb{T}\times\mathbb{R}\times\mathbb{T}}\mathrm{e}^{-i(kx+\eta y+lz)}\varphi(x,y,z)\,\mathrm{d}x\,\mathrm{d}y\,\mathrm{d}z.
\end{equation}
The real-variable function $\varphi$ can then be reconstructed as
\begin{equation}
\varphi(x,y,z)=\sum_{(k,l)\in\mathbb{Z}^2}\int_{\mathbb{R}}\mathrm{e}^{i(kx+\eta y+lz)} \hat\varphi_{k,l}(\eta)\,\mathrm{d}\eta.
\end{equation}
We denote the $k=0$ mode with a single index, $\varphi_0$, without distinguishing between the original function and its Fourier transform.

We use $\langle\cdot,\cdot\rangle$ to denote the $L^2$ inner product and $\|\cdot\|$ for the $L^2$ norm. Other norms are denoted with a subscript, e.g.\ writing $\|\cdot\|_{L^\infty}$ for the $L^\infty$ norm.
Additionally, we define
\begin{equation}
    |k,\eta,l|^2:=k^2+\eta^2+l^2, \quad \langle k,\eta,l\rangle:=\sqrt{1+|k,\eta,l|^2}.
\end{equation}
Based on this, the $H^s$ norm for $s>0$ is denoted as
\begin{equation}
    \|\varphi\|_{H^s}^2:=\sum_{k,l\in\ZZ}\int_{\RR} \langle k,\eta,l \rangle^{2s} |\hat{\varphi}_{k,l}(\eta)|^2\dd\eta.
\end{equation}
For a Hilbert space $H$ of functions on $\mathbb{T}\times\mathbb{R}\times\mathbb{T}$ we denote the natural norm of the space $L^p\left(0,T;H\right)$, $1\leq p \leq \infty$, as
\begin{equation}
    \|F\|_{L^p\left(0,T;H\right)}=\|F\|_{L^p_tH}.
\end{equation}
(Typically, the space $H$ will be either $L^2$ or $H^s$.)

\section{Analysis of the nonzero modes}\label{sec:neqs}
We begin by setting up the precise structure of the main multiplier $\cA$ in Section \ref{ssec:neq_prelim}. Sections \ref{ssec:GGamma_neq} and \ref{ssec:U13neq} then establish bounds for $\GG_{\neq},\Gamma_{\neq}$ and $U^1_{\neq},U^3_{\neq}$, proving Proposition \ref{prop:GGamma} and \ref{prop:U13}.
\subsection{Preliminaries}\label{ssec:neq_prelim}
As sketched in \eqref{def:cA}, the energy functionals we employ include a weight function $\e^{\lambda\nu^{\frac13}t}$ to encode time decay (at enhanced dissipation rate) and derivatives to capture spatial regularity. In addition to the derivatives, we will include so-called ghost weights to deal with certain linear terms that arise.

Regarding derivatives, we recall here the standard Sobolev product estimate
\begin{equation}\label{eq:prod-ineq}
\|FG\|_{H^s}\lesssim \|F\|_{H^s}\|G\|_{L^\infty}+\|F\|_{L^\infty}\|G\|_{H^s}\qquad s\geq 0,
\end{equation}
and record that by definition \eqref{eq:nablaL} there holds that
\begin{equation}\label{eq:time_extraction}
    \||\nabla_L|^{-a}F\|\lesssim \l t\r^{-a} \||\nabla|^aF\|,\qquad a\geq 0.
\end{equation}

\subsubsection{Ghost weights}
The ghost weights are constructed through symbols $M_j(t,k,\eta,l)$, $1\leq j\leq 3$, satisfying that
\begin{equation}
 M_j(0,k,\eta,l)=1,\quad 1\leq j\leq 3,
\end{equation}
and defined through the convenient (as will become clear below) relations
\begin{equation}
-\frac{\dot{M}_1}{M_1}=\frac{\nu^{\frac13}k^2}{k^2+\nu^{\frac23}|\eta-tk|^2},\qquad
-\frac{\dot{M}_2}{M_2}=\frac{2}{2\beta-1}\frac{|k,l|k^2}{|k,\eta-kt,l|^{3}},\qquad
-\frac{\dot{M}_3}{M_3}=\frac{|k||k,l|^{\frac12}}{|k,\eta-kt,l|^{\frac32}},\label{def:M3}
\end{equation}
where the dot denotes a time derivative in the variable $t$. 
In particular, we note that $M_j(t,0,\eta,l)\equiv 1$. The denomination of \emph{ghost} multipliers is due to the following fact:
\begin{lemma}\label{lem:mult_size}
 For $1\leq j\leq 3$ consider $M_j$ as defined above. Then there exists $\kappa_0>0$ such that
 \begin{equation}\label{eq:ghost_constant}
 0<\kappa_0\leq M_j\leq 1,\quad 1\leq j\leq 3.
 \end{equation}
\end{lemma}
As a direct consequence, denoting by $\cM_j$ the associated Fourier multiplication operators
\begin{equation}
 \cM_j F:=\cF^{-1}(M_j\widehat{F}),   
\end{equation}
we note that
\begin{equation}\label{eq:ghost_bd}
 \kappa_0 \|F\|\leq \|\cM_j F\|\leq \|F\|.
\end{equation}

\begin{proof}[Proof of Lemma \ref{lem:mult_size}]
By construction we have that $\dot{M}_j\leq 0$ and thus $M_j\leq 1$. Moreover, for any $m>1$ we bound
\begin{align}
\int_0^{+\infty}\frac{1}{|k,\eta-kt,l|^{m}}\dd t&\leq \frac{1}{2|k|} \int_{-\infty}^{+\infty}\frac{1}{(k^2+l^2+s^2)^{\frac{m}{2}}}\dd s =  \frac{\sqrt{\pi}}{2}\frac{\Gamma\left(\frac{m-1}{2}\right)}{\Gamma\left(\frac{m}{2}\right)}\frac{1}{|k||k,l|^{m-1}}.
\end{align}
Integrating \eqref{def:M3} this gives
\begin{equation}
    1\geq M_3 = \exp\left(-\int_0^t\frac{|k||k,l|^{\frac12}}{|k,\eta-ks,l|^{\frac32}}\dd s\right)\geq \exp\left(- \frac{\sqrt{\pi}}{2}\frac{\Gamma\left(\frac{1}{4}\right)}{\Gamma\left(\frac{3}{4}\right)}\right) >0
\end{equation}
and analogously for $j=1,2$.
\end{proof}
More specifically, these multipliers have the following different roles:

\medskip

\noindent $-$ $\cM_1$ is designed to capture the transition between the enhanced dissipation and the dissipation regimes. This is encoded in the following lemma:
\begin{lemma}\label{lem:enhanced_dissip_estim}
 There holds that   
 \begin{equation}
 \nu^{\frac16}\norm{F_{\neq}}_{L^2_tL^2}\leq 2\norm{\sqrt{-\frac{\dot{\cM}_1}{\cM_1}}F_{\neq}}_{L^2_tL^2}+\nu^{\frac12}\norm{\nabla_LF_{\neq}}_{L^2_tL^2}.
\end{equation}
\end{lemma}
\begin{proof}
 This is a direct consequence of the pointwise inequality
 \begin{equation}
  \nu^{\frac16}\leq 2\sqrt{-\frac{\dot{M}_1}{M_1}}+\frac12\nu^{\frac12}|k|^{-1}|\eta-t k|,   
 \end{equation}
 valid for $k\neq 0$ (see also \cite{L18}*{Lemma 4.2}) and Parseval's identity.
\end{proof}
We directly obtain the following bounds:
\begin{corollary}\label{cor:ED_bounds}
 Under the bootstrap assumptions \eqref{eq:bootstrap_nonzero_G}--\eqref{eq:bootstrap_nonzero_U} there holds that  
 \begin{equation}
        \| \cA  \GG_{\neq} \|_{L^2_tL^2}^2+\| \cA  \Gamma_{\neq} \|_{L^2_tL^2}^2
        +\| \cA  U^1_{\neq} \|_{L^2_tL^2}^2+\| \cA  U^3_{\neq} \|_{L^2_tL^2}^2\leq 20\varepsilon^2 \nu^{-\frac13}.\label{eq:enhanced_diss}
    \end{equation}
\end{corollary}

\noindent $-$ $\cM_2$ precisely captures the time derivative of the derivative operator appearing in the inner product of the energy functional $\sfE$, as defined in \eqref{eq:GGamma_energy}. Its application can be found, for example, in the proof of Lemma \ref{lem:GGamma_reduction}.

\medskip

\noindent $-$
$\cM_3$ will be extensively used in nonlinear terms exhibiting sufficient time decay to be $L^2_t$ integrable. In practical terms, whenever we encounter a Fourier multiplier of the form $|\de_x|^{\frac12}|\nabla_{x,z}|^{n-\frac12}|\nabla_L|^{-n}, n\geq \frac12$, it will be possible to bound it via $-\dot \cM_3 \cM_3^{-1}$.

\medskip

Henceforth we let
\begin{equation}
    \cM:=\cM_1\cM_2\cM_3.
\end{equation}
The usefulness of such a ghost weight can be hinted at as follows: Thanks to \eqref{eq:ghost_bd}, we have the freedom of controlling an $L^2$ norm of a function through a weighted version of it, which in turn satisfies a possibly favorable time derivative equation, namely
\begin{equation}
    \frac12\ddt \|\cM F\|^2=- \norm{\sqrt{-\frac{\dot \cM}{\cM}}\cM F}^2 +\l \cM\de_t F, \cM F\r.
\end{equation}
Here we see that the term with time derivative on the ghost weight has a favorable sign, while the inner product involving $\de_tF$ gives some extra freedom that can be used depending on the equations satisfied by $F$ and in particular on the nonlinear terms (see Section \ref{ssec:GGamma_neq} for more details and explicit computations).

Moreover, thanks to the product structure $M=M_1M_2M_3$, we have 
\begin{equation}
    -\frac{\dot M}{M}= -\frac{\dot M_1}{M_1}-\frac{\dot M_2}{M_2}-\frac{\dot M_3}{M_3},
\end{equation}
so that
\begin{align}
    \norm{\sqrt{-\frac{\dot \cM}{\cM}}F}^2= \int -\frac{\dot \cM}{\cM}|F|^2=\int -\sum_{j=1}^3\frac{\dot \cM_j}{\cM_j}|F|^2 = \sum_{j=1}^3 \norm{\sqrt{-\frac{\dot \cM_j}{\cM_j}}F}^2.
\end{align} 

\subsubsection{The main multiplier $\cA$}
The weight functions in our energy estimates are collected in the main multiplier $\cA $ defined in \eqref{def:cA}, the symbol $A$ of which is
\begin{equation}
 A(t,k,\eta,l):=\e^{\lambda\nu^{\frac13}t}\langle k,\eta,l\rangle^{2m}\Pi_{i=1}^3M_i(t,k,\eta,l),\qquad m\in\NN,m\geq 2.
\end{equation}
We record some key properties of $\cA $:
\begin{lemma}
 We have that
 \begin{equation}\label{eq:prod_ineq_A}
 \|\cA (FG)\|\lesssim \|\cA F\|\|G\|_{L^\infty}+\|F\|_{L^\infty}\|\cA G\|,
\end{equation}
and
\begin{align} \label{eq:property_A}
    \| \nabla_L F\|_{H^n}  \lesssim \nu^{-\frac13} \| \cA  F\|\qquad n\leq 2m-1.
\end{align}
\end{lemma}
\begin{proof}
 The first inequality \eqref{eq:prod_ineq_A} follows from \eqref{eq:prod-ineq} and \eqref{eq:ghost_bd}: with this we have that 
 \begin{align}
 \|\cA (FG)\|&\leq \e^{\lambda \nu^{\frac13} t}\|FG\|_{H^{2m}}\notag\\
 &\lesssim \e^{\lambda \nu^{\frac13} t} \|F\|_{H^{2m}}\|G\|_{L^\infty}+\|F\|_{L^\infty} \e^{\lambda \nu^{\frac13} t} \|G\|_{H^{2m}}\notag\\
 &\leq \kappa_0^{-1} (\|\cA F\|\|G\|_{L^\infty}+\|F\|_{L^\infty}\|\cA G\|).
\end{align}
For the second one it suffices to use the crude bound $ \| \nabla_L F\| \lesssim \l t \r \|\l \nabla \r F\|$ to conclude with \eqref{eq:prod_ineq_A} that
\begin{align}
    \| \nabla_L F\|_{H^n} \lesssim \l t \r \e^{-\lambda \nu^{\frac13}t} \|\l \nabla \r \e^{\lambda \nu^{\frac13} t }F\|_{H^{2m-1}} \lesssim \nu^{-\frac13} \| \cA  F\|,\qquad n\leq 2m-1.
\end{align}
\end{proof}

\subsection{Control of $\GG_{\neq},\Gamma_{\neq}$ -- proof of Proposition \ref{prop:GGamma}}\label{ssec:GGamma_neq}
The proof of Proposition \ref{prop:GGamma} combines two main ingredients: the energy  functional $\sfE$ from \eqref{eq:GGamma_energy} to exploit the symmetric relation between $\GG$ and $\Gamma$ and its coercivity, as well as trilinear estimates. As hinted at above, here the precise choice of multipliers in our norms plays an important role and accounts for enhanced dissipation ($\cM_1$), certain lower order time derivatives ($\cM_2$) and time decay of some unknowns ($\cM_3$).

We begin by recalling that by \eqref{eq:full_shorthand} the unknowns $\GG_{\neq}$ and $\Gamma_{\neq}$ satisfy the equations
\begin{subequations}
\begin{alignat}{2}
\de_t\GG_{\neq} &=\nu\Delta_L \GG_{\neq} +\frac{1}{2}\de_x\de_y^L|\nabla_L|^{-2}\GG_{\neq} +\beta |\nabla_{x,z}||\nabla_L|^{-1}\Gamma_{\neq} + \cT_\star(U,U^2)_{\neq}+\de_y^L\cP_\star(U,U)_{\neq},\\
\de_t\Gamma_{\neq} & = \nu\Delta_L\Gamma_{\neq}-\frac{1}{2}\de_x\de_y^L|\nabla_L|^{-2}\Gamma_{\neq} -\beta |\nabla_{x,z}||\nabla_L|^{-1}\GG_{\neq}+\cT_\circ (U,\Theta)_{\neq}.
\end{alignat}
\end{subequations}
As a first step we show:
\begin{lemma}\label{lem:GGamma_reduction}
 Under the assumptions of Theorem \ref{thm:bootstrap_step}, we have that
\begin{equation}
\begin{aligned}
    &\sum_{F\in\{\GG,\Gamma\}}\norm{\cA F_{\neq}(t)}^2 +2\nu\norm{\nabla_L\cA F_{\neq}}^2_{L^2_tL^2}+\norm{\sqrt{-\frac{\dot \cM }{\cM }}\cA F_{\neq}}^2_{L^2_tL^2}\leq 45\eps^2 + \frac{4\beta}{2\beta-1}\int_0^\infty N_{\sfE},
\end{aligned}
\end{equation}
where 
\begin{align}\label{def:NE}
    N_{\sfE}&:=\langle \cA\GG_{\neq},\cA\,\cT_\star(U,U^2)_{\neq}+ \cA\, \de_y^L\cP_{\star}(U,U)_{\neq}\rangle +\langle \cA\,  \cT_\circ (U,\Theta)_{\neq},\cA\Gamma_{\neq}\rangle \notag\\
    &\quad +\frac{1}{2\beta}\langle  \de_x\de_y^L|\nabla_{x,z}|^{-1} |\nabla_L|^{-1}\cA\GG_{\neq},\cA\,  \cT_\circ (U,\Theta)_{\neq}\rangle \notag\\
    & \quad +\frac{1}{2\beta}\langle   \cA\,\cT_\star(U,U^2)_{\neq}+ \cA\, \de_y^L\cP_{\star}(U,U)_{\neq},\de_x\de_y^L|\nabla_{x,z}|^{-1} |\nabla_L|^{-1}\cA\Gamma_{\neq}\rangle.
\end{align}
\end{lemma}
This follows by computing the time derivative of $\sfE$, using the definition of $\cM_2$ and invoking the bootstrap assumptions. 
\begin{proof}[Proof of Lemma \ref{lem:GGamma_reduction}]
 From the definition \eqref{def:cA} of $\cA$ we have that
\begin{equation}\label{eq:dotA}
    \dot \cA = \lambda \nu^{\frac13}\cA + \frac{\dot \cM}{\cM} \cA,
\end{equation}
and hence by a direct computation from \eqref{eq:GGamma_energy}
\begin{align}
    \ddt \sfE (t) &= -\nu\norm{\nabla_L\cA\GG_{\neq}}^2-\nu\norm{\nabla_L\cA\Gamma_{\neq}}^2 -\frac{\nu}{\beta}\left\langle  \de_x\de_y^L|\nabla_{x,z}|^{-1} |\nabla_L|^{-1} \nabla_L\cA\GG_{\neq},\nabla_L\cA\Gamma_{\neq}\right\rangle \notag\\
    &\quad -\norm{\sqrt{-\frac{\dot \cM }{\cM }}\cA \GG_{\neq} }^2 - \norm{\sqrt{-\frac{\dot \cM }{\cM }}\cA \Gamma_{\neq} }^2 +\lambda \nu^{\frac13}\norm{\cA \GG_{\neq}}^2+\lambda \nu^{\frac13}\norm{\cA \Gamma_{\neq}}^2\notag\\
    &\quad + \frac{1}{\beta}\left\langle  \de_x\de_y^L|\nabla_{x,z}|^{-1} |\nabla_L|^{-1} \frac{\dot \cM}{\cM}\cA\GG_{\neq},\cA\Gamma_{\neq}\right\rangle + \frac{\lambda \nu^{\frac13}}{\beta}\left\langle  \de_x\de_y^L|\nabla_{x,z}|^{-1} |\nabla_L|^{-1} \cA\GG_{\neq},\cA\Gamma_{\neq}\right\rangle \notag \\
    &\quad 
    +\frac{1}{2\beta}\left\langle \frac{\dd }{\dd t }\left( \de_x\de_y^L|\nabla_{x,z}|^{-1} |\nabla_L|^{-1}\right) \cA\GG_{\neq},\cA\Gamma_{\neq}\right\rangle 
    +N_{\sfE}(t),
    \end{align}
By Plancherel's theorem we have
\begin{equation}\label{eq:Riesz_bounds}
 \norm{\de_x|\nabla_{x,z}|^{-1} F_{\neq}} \leq \norm{F_{\neq}},\qquad \norm{\de_y^L |\nabla_L|^{-1} F_{\neq}} \leq \norm{F_{\neq}},
\end{equation}
and it follows that 
\begin{equation}
    \frac{\nu}{\beta}\left\langle  \de_x\de_y^L|\nabla_{x,z}|^{-1} |\nabla_L|^{-1} \nabla_L\cA\GG_{\neq},\nabla_L\cA\Gamma_{\neq}\right\rangle \leq \frac{\nu}{\beta}\norm{\nabla_L\cA\GG_{\neq}}\norm{\nabla_L\cA\Gamma_{\neq}},
\end{equation}
and
\begin{equation}
    \frac{1}{\beta}\left\langle  \de_x\de_y^L|\nabla_{x,z}|^{-1} |\nabla_L|^{-1} \frac{\dot \cM}{\cM}\cA\GG_{\neq},\cA\Gamma_{\neq}\right\rangle\leq \frac{1}{\beta}\norm{\sqrt{-\frac{\dot\cM}{\cM}}\cA\GG_{\neq}}\norm{\sqrt{-\frac{\dot\cM}{\cM}}\cA\Gamma_{\neq}}.
\end{equation}
We can thus bound the time derivative of $\sfE(t)$ as
\begin{align}
    \ddt \sfE(t)&\leq  -\left(1-\frac{1}{2\beta}\right)\left(\norm{\sqrt{-\frac{\dot \cM }{\cM }}\cA \GG_{\neq}}^2+\norm{\sqrt{-\frac{\dot \cM }{\cM }}\cA \Gamma_{\neq}}^2\right)\notag\\
    &\quad -\nu\left(1-\frac{1}{2\beta}\right)\left(\norm{\nabla_L\cA \GG_{\neq}}^2+\norm{\nabla_L\cA\Gamma_{\neq}}^2\right)\notag\\
    &\quad +\lambda\nu^{\frac13}\left(1+\frac{1}{2\beta}\right)\left(\norm{\cA \GG_{\neq}}^2+\norm{\cA \Gamma_{\neq}}^2\right)\notag\\
    &\quad  +\frac{1}{2\beta}\left\langle \frac{\dd }{\dd t }\left(\de_x\de_y^L|\nabla_{x,z}|^{-1} |\nabla_L|^{-1}\right) \cA \GG_{\neq},\cA \Gamma_{\neq}\right\rangle + N_{\sfE}(t).
\end{align}    
Explicit computation of the time derivative in the second to last term, integration in time and coercivity of $\sfE(t)$ as in \eqref{eq:coercive-L2} lead to 
\begin{align}\label{eq:dtE_integrated}
    &\norm{\cA \GG_{\neq}(t)}^2+\norm{\cA \Gamma_{\neq}(t)}^2 +2\nu \left(\norm{\nabla_L\cA\GG_{\neq}}^2_{L^2_tL^2}+\norm{\nabla_L\cA\Gamma_{\neq}}^2_{L^2_tL^2}\right)\notag\\
    &\qquad +2\left(\norm{\sqrt{-\frac{\dot \cM }{\cM }}\cA \GG_{\neq}}^2_{L^2_tL^2}+\norm{\sqrt{-\frac{\dot \cM }{\cM }}\cA \Gamma_{\neq}}^2_{L^2_tL^2}\right) \notag\\
    &\leq \frac{2\beta+1}{2\beta-1}\left(\norm{\cA \GG_{\neq}(0)}^2+\norm{\cA \Gamma_{\neq}(0)}^2\right) +2  \lambda \nu^{\frac13}\frac{2\beta+1}{2\beta-1}\left(\norm{\cA \GG_{\neq}}^2_{L^2_tL^2}+\norm{\cA \Gamma_{\neq}}^2_{L^2_tL^2}\right)\notag\\
    & \qquad -\frac{2}{2\beta-1}\int_0^\infty \left\langle \de^2_x|\nabla_{x,z}| |\nabla_L|^{-3}\cA \GG_{\neq},\cA \Gamma_{\neq}\right\rangle + \frac{4\beta}{2\beta-1}\int_0^\infty N_{\sfE}.
\end{align}
By smallness of the initial data (see e.g.\ \eqref{eq:bootstrap_nonzero_G}--\eqref{eq:bootstrap_nonzero_Gamma}) and since $\beta>\frac12$, we have
\begin{equation}
    \frac{2\beta+1}{2\beta-1}\left(\norm{\cA \GG_{\neq}(0)}^2+\norm{\cA \Gamma_{\neq}(0)}^2\right) \leq C_\beta^{-2}\eps_0^2\leq\eps^2,
\end{equation}
and from Corollary \ref{cor:ED_bounds} and the choice of $\lambda$ in \eqref{def:cA} we obtain
\begin{equation}
    2\lambda \nu^{\frac13}\frac{2\beta+1}{2\beta-1}\left(\norm{\cA \GG_{\neq}}^2_{L^2_tL^2}+\norm{\cA \Gamma_{\neq}}^2_{L^2_tL^2}\right) \leq 40 \eps^2.
\end{equation}
To prove the claim it thus suffices to observe that by construction, the second to last term in \eqref{eq:dtE_integrated} precisely recovers the ghost multiplier $\cM_2$ (as defined in \eqref{def:M3}), and thus
\begin{equation}
\begin{aligned}
    \abs{-\frac{2}{2\beta-1}\int_0^\infty \left\langle \left(\de^2_x|\nabla_{x,z}| |\nabla_L|^{-3}\right) \cA \GG_{\neq},\cA \Gamma_{\neq}\right\rangle}&=\abs{\int_0^\infty \left\langle -\frac{\dot\cM_2}{\cM_2} \cA \GG_{\neq},\cA \Gamma_{\neq}\right\rangle}\\
    &\leq \norm{\sqrt{-\frac{\dot \cM_2 }{\cM_2 }}\cA \GG_{\neq}}_{L^2_tL^2}\norm{\sqrt{-\frac{\dot \cM_2 }{\cM_2 }}\cA \Gamma_{\neq}}_{L^2_tL^2},
\end{aligned}    
\end{equation}
which can be absorbed by the correspondent terms appearing in the LHS of \eqref{eq:dtE_integrated}. 
\end{proof}

To conclude the proof of Proposition \ref{prop:GGamma}, by Lemma \ref{lem:GGamma_reduction} it thus suffices to show that
\begin{equation}\label{eq:intNsfE}
 \int_0^\infty N_{\sfE}\lesssim (\nu^{-\frac56}\eps)\eps^2.
\end{equation}
More precisely, we will establish the following bounds, which by \eqref{def:NE} imply \eqref{eq:intNsfE}:
\begin{lemma}\label{lem:GGammanonlin}
 Under the assumptions of Theorem \ref{thm:bootstrap_step}, there holds that
 \begin{align}
  \int_0^\infty|\l \cA\GG_{\neq},\cA\,\cT_\star(U,U^2)_{\neq} \r| +|\langle   \de_x\de_y^L|\nabla_{x,z}|^{-1} |\nabla_L|^{-1}\cA\Gamma_{\neq},\cA\,\cT_\star(U,U^2)_{\neq}\rangle|  
    &\lesssim (\nu^{-\frac56}\eps)\eps^2, \qquad\label{eq:T_star_bds}\\
  \int_0^\infty|\l\cA\Gamma_{\neq}, \cA\,\cT_\circ(U,\Theta)_{\neq} \r| +|\langle   \de_x\de_y^L|\nabla_{x,z}|^{-1} |\nabla_L|^{-1}\cA\GG_{\neq},\cA\,\cT_\circ(U,U^2)_{\neq}\rangle|   &\lesssim (\nu^{-\frac23}\eps)\eps^2, \qquad\label{eq:T_circ_bds}\\
  \int_0^\infty |\l \cA\GG_{\neq},\cA \, \de_y^L\cP_\star(U,U)_{\neq}\r| + |\l \de_x\de_y^L|\nabla_{x,z}|^{-1} |\nabla_L|^{-1}\cA\Gamma_{\neq},\cA \, \de_y^L\cP_\star(U,U)_{\neq}\r| &\lesssim (\nu^{-\frac56}\eps)\eps^2.\qquad
  \label{eq:P_star_bds}
 \end{align} 
\end{lemma}

\begin{proof}
The bound \eqref{eq:T_star_bds} is established in Section \ref{sssec:cTstarGGamma}, \eqref{eq:T_circ_bds} in Section \ref{sssec:cTcircGGamma}, and for \eqref{eq:P_star_bds} we refer to Section \ref{sssec:cPGGamma}.
\end{proof}

\subsubsection{Nonlinear terms analysis: $\cT_\star(U,U^2)_{\neq}$}\label{sssec:cTstarGGamma}
Here we establish \eqref{eq:T_star_bds}. Since $\GG_{\neq}$ and $\Gamma_{\neq}$ satisfy analogous bootstrap assumptions \eqref{eq:bootstrap_nonzero_G} resp.\ \eqref{eq:bootstrap_nonzero_Gamma}, by Plancherel (see also \eqref{eq:Riesz_bounds}), it suffices to show the bound \eqref{eq:T_star_bds} for $\l \cA\GG_{\neq},\cA\,\cT_\star(U,U^2)_{\neq} \r$.

Firstly, note that 
\begin{equation}
\cT_\star(U,U^2) =\sum_r \cT^r_\star(U,U^2), \qquad \cT^r_\star(U,U^2):=-|{\nabla_L}|^{\frac32} |\nabla_{x,z}|^{-\frac12}(U^r \de_r^LU^2),
\end{equation}
where $r\in\{1,2,3\}$ with the convention $\de_1^L=\de_x$, $\de_2^L=\de_y^L$, and $\de_3^L=\de_z$.
Moreover, we further split each term based on the interaction which generates it, namely 
\begin{equation}\label{def:cTstar_decomposition}
\cT^r_\star(U,U^2)_{\neq} =\sum_{\kappa_1,\kappa_2\in\{0,\neq\}}\cT^r_\star(U_{\kappa_1},U^2_{\kappa_2})_{\neq}, \qquad \cT^r_\star(U_{\kappa_1},U^2_{\kappa_2})_{\neq}:=-|{\nabla_L}|^{\frac32} |\nabla_{x,z}|^{-\frac12}(U^r_{\kappa_1} \de_r^LU^2_{\kappa_2})_{\neq}.
\end{equation}
Since $(0,0)$ interactions cannot force nonzero modes, we need to bound only the $(\neq,\neq)$, $(\neq,0)$, and $(0,\neq)$ interactions.
We prove the following bounds
\begin{subequations}
\begin{alignat}{7}
    \int_0^\infty|\l \cA \GG_{\neq},\cA \,\cT^r_\star(U_{\neq},U^2_{\neq}) \r| &\lesssim (\nu^{-\frac34}\eps)\eps^2,\quad r=1,3, \label{eq:Tstar-r-neqneq}\\
    \int_0^\infty|\l \cA\GG_{\neq},\cA\,\cT^2_\star(U_{\neq},U^2_{\neq})\r| &\lesssim (\nu^{-\frac12}\eps)\eps^2,\label{eq:Tstar-2-neqneq}\\
    \int_0^\infty|\l\cA\GG_{\neq}, \cA\,\cT^2_\star(U_{\neq},U^2_0) \r| 
    &\lesssim (\nu^{-\frac12}\eps)\eps^2,\label{eq:Tstar-2-neq0}\\
    \int_0^\infty|\l \cA\GG_{\neq},\cA\,\cT^3_\star(U_{\neq},U^2_0) \r| 
    &\lesssim (\nu^{-\frac56}\eps)\eps^2,\label{eq:Tstar-3-neq0}\\
    \int_0^\infty|\l \cA\GG_{\neq},\cA\,\cT^1_\star(U_0,U^2_{\neq}) \r| 
    &\lesssim (\nu^{-\frac34}\eps)\eps^2,\label{eq:Tstar-1-0neq}\\
    \int_0^\infty|\l\cA\GG_{\neq}, \cA\,\cT^2_\star(U_0,U^2_{\neq}) \r| 
    &\lesssim (\nu^{-\frac23}\eps)\eps^2,\label{eq:Tstar-2-0neq}\\
    \int_0^\infty|\l\cA\GG_{\neq}, \cA\,\cT^3_\star(U_0,U^2_{\neq}) \r| 
    &\lesssim (\nu^{-\frac34}\eps)\eps^2.\label{eq:Tstar-3-0neq}
\end{alignat}
\end{subequations}
Note that among the $(\neq, 0)$ interactions, we have
$\cT^1_\star(U_{\neq},U^2_0)_{\neq}=0$ since $\de_x U^2_0=0$.

Beginning with $\cT^1_\star(U_{\neq},U^2_{\neq})_{\neq}$ and $\cT^3_\star(U_{\neq},U^2_{\neq})_{\neq}$, which can be treated in a similar way, we compute (for $r=1$ or $r=3$) that 
\begin{align}
    \l\cA \,\cT^r_\star(U_{\neq},U^2_{\neq}), \cA \GG_{\neq} \r
    &= \langle \cA   |\nabla_L||{\nabla_L}|^{\frac12} |\nabla_{x,z}|^{-\frac12}(U^r_{\neq} \de_r^L |\nabla_{x,z}|^{\frac12}|\nabla_L|^{-\frac32} \GG_{\neq}), \cA \GG_{\neq}\rangle\notag\\
    &= \langle  \cA  |\nabla_L|(U^r_{\neq} \de_r^L |\nabla_{x,z}|^{\frac12}|\nabla_L|^{-\frac32} \GG_{\neq}) ,|{\nabla_L}|^{\frac12} |\nabla_{x,z}|^{-\frac12}\cA \GG_{\neq}\rangle\notag\\
    &=\sum_{j} \langle \cA  \partial_j^L(U^r_{\neq} \de_r^L |\nabla_{x,z}|^{\frac12}|\nabla_L|^{-\frac32} \GG_{\neq}),\partial_j^L|\nabla_L|^{-1}|{\nabla_L}|^{\frac12} |\nabla_{x,z}|^{-\frac12}\cA \GG_{\neq}\rangle\notag\\
    &=\sum_{j} \l \cA  (U^r_{\neq} \de_r^L |\nabla_{x,z}|^{\frac12}\partial_j^L|\nabla_L|^{-\frac32} \GG_{\neq}),\partial_j^L|\nabla_L|^{-1}|{\nabla_L}|^{\frac12} |\nabla_{x,z}|^{-\frac12}\cA \GG_{\neq}\rangle\notag\\
    &\quad +\langle \cA  (\partial_j^LU^r_{\neq} \de_r^L |\nabla_{x,z}|^{\frac12}|\nabla_L|^{-\frac32} \GG_{\neq}),\partial_j^L|\nabla_L|^{-1}|{\nabla_L}|^{\frac12} |\nabla_{x,z}|^{-\frac12}\cA \GG_{\neq}\rangle. \qquad\label{eq:distributing derivatives}
\end{align}    
We carefully bound these terms in such a way that we avoid that the $\de_r^L$ derivative and all the derivatives $|\nabla|^{2m}$ in $\cA$ fall on $\GG_{\neq}$. For the first term on the right-hand side of \eqref{eq:distributing derivatives}, we write
\begin{align}\label{eq:I_1234}
    &|\nabla|^{2m}  (U^r_{\neq} \de_r^L |\nabla_{x,z}|^{\frac12}\partial_j^L|\nabla_L|^{-\frac32} \GG_{\neq})=I_1+I_2-I_3+\de_r^LI_4,
\end{align}
where 
\begin{align}
    &\quad I_1:=\sum_{|\alpha|=2m}\de^\alpha U^r_{\neq}\cdot \de_r^L |\nabla_{x,z}|^{\frac12}\partial_j^L|\nabla_L|^{-\frac32} \GG_{\neq},\\
    &\quad I_2:=\sum_{|\alpha|+|\beta|= 2m-1}\de^\alpha U^r_{\neq}\cdot \de^\beta \de_r^L |\nabla_{x,z}|^{\frac12}\partial_j^L|\nabla_L|^{-\frac32} \GG_{\neq},\\
    &\quad I_3:=\sum_{|\beta|=2m}(\de_r^L U^r_{\neq}\cdot \de^\beta  |\nabla_{x,z}|^{\frac12}\partial_j^L|\nabla_L|^{-\frac32} \GG_{\neq}),\\
    &\quad I_4:=\sum_{|\beta|=2m}(U^r_{\neq}\cdot \de^\beta  |\nabla_{x,z}|^{\frac12} \partial_j^L|\nabla_L|^{-\frac32} \GG_{\neq}),
\end{align}
and thus
\begin{align}
 &|\l \cA  (U^r_{\neq} \de_r^L |\nabla_{x,z}|^{\frac12}\partial_j^L|\nabla_L|^{-\frac32} \GG_{\neq}),\partial_j^L|\nabla_L|^{-1}|{\nabla_L}|^{\frac12} |\nabla_{x,z}|^{-\frac12}\cA \GG_{\neq}\rangle|\notag\\
 &\quad \lesssim \e^{\lambda\nu^{\frac13}t} \left( \|I_1\| + \norm{I_2} +\norm{I_3} \right) \| |{\nabla_L}|^{\frac12} |\nabla_{x,z}|^{-\frac12}\cA \GG_{\neq}\| +\e^{\lambda\nu^{\frac13}t} \norm{I_4} \| |{\nabla_L}|^{\frac12} |\nabla_{x,z}|^{\frac12}\cA \GG_{\neq}\|.
\end{align}
From this, using that
\begin{align}
    &\e^{\lambda\nu^{\frac13}t}\norm{I_1} \lesssim \|\cA U^r_{\neq}\| \| \de_r^L|\nabla_{x,z}|^{\frac12}|\nabla_L|^{-\frac12}G_{\neq}\|_{L^\infty},\\
    &\e^{\lambda\nu^{\frac13}t}\norm{I_2} \lesssim \|U^r_{\neq}\|_{L^\infty}\| |\nabla_{x,z}|^{\frac12}|\nabla_L|^{-\frac12} \cA G_{\neq}\|+\|\cA U^r_{\neq}\| \| \de_r^L|\nabla_{x,z}|^{\frac12}|\nabla_L|^{-\frac12}G_{\neq}\|_{L^\infty},\\
    &\e^{\lambda\nu^{\frac13}t}\norm{I_3} \lesssim \| \nabla_{x,z} U^r_{\neq}\|_{L^\infty}\||\nabla_{x,z}|^{\frac12}|\nabla_L|^{-\frac12}\cA G_{\neq}\|,\\
    &\e^{\lambda\nu^{\frac13}t}\norm{I_4} \lesssim\| U^r_{\neq}\|_{L^\infty}\| |\nabla_{x,z}|^{\frac12}|\nabla_L|^{-\frac12}\cA G_{\neq}\|,
\end{align}
the property \eqref{eq:property_A} of $\cA$, \eqref{eq:time_extraction}, and the simple bounds
\begin{align}\label{eq:other norm bounds}
\| |\nabla_{x,z}|^{\frac12}|\nabla_L|^{-\frac12}\cA G_{\neq}\|& \lesssim \norm{\cA\GG_{\neq}},\\
 \| |{\nabla_L}|^{\frac12} |\nabla_{x,z}|^{-\frac12} \cA \GG_{\neq}\|& \lesssim \| \cA \GG_{\neq}\|^{\frac12}\| {\nabla_L} \cA \GG_{\neq}\|^{\frac12},\\
 \| |{\nabla_L}|^{\frac12} |\nabla_{x,z}|^{\frac12} \cA \GG_{\neq}\|&\lesssim \|{\nabla_L} \cA \GG_{\neq}\|,
\end{align}
we obtain 
\begin{align}
    &|\l \cA  (U^r_{\neq} \de_r^L |\nabla_{x,z}|^{\frac12}\partial_j^L|\nabla_L|^{-\frac32} \GG_{\neq}),\partial_j^L|\nabla_L|^{-1}|{\nabla_L}|^{\frac12} |\nabla_{x,z}|^{-\frac12}\cA \GG_{\neq}\rangle|\notag\\
    &\quad \lesssim \|\cA U^r_{\neq}\| \l t \r^{-\frac12}\| \cA \GG_{\neq}\|^{\frac32}\| {\nabla_L} \cA \GG_{\neq}\|^{\frac12} +\|\cA U^r_{\neq}\|\| \cA \GG_{\neq}\|^{\frac32}\| {\nabla_L} \cA \GG_{\neq}\|^{\frac12}\notag\\
    &\qquad +\| \cA U^r_{\neq}\|\|\cA \GG_{\neq}\| \|{\nabla_L} \cA \GG_{\neq}\|.
\end{align}
With the bootstrap assumptions in Theorem \ref{thm:bootstrap_step} it follows that
\begin{equation}
    \int_0^\infty|\l \cA  (U^r_{\neq} \de_r^L |\nabla_{x,z}|^{\frac12}\partial_j^L|\nabla_L|^{-\frac32} \GG_{\neq}),\partial_j^L|\nabla_L|^{-1}|{\nabla_L}|^{\frac12} |\nabla_{x,z}|^{-\frac12}\cA \GG_{\neq}\rangle|\lesssim (\nu^{-\frac23}\eps)\eps^2. 
\end{equation}
Consider now the second term appearing in the sum in \eqref{eq:distributing derivatives}
\begin{equation}\label{eq:J1234}
    \l \cA  (\de_j^L U^r_{\neq} \de_r^L |\nabla_{x,z}|^{\frac12}|\nabla_L|^{-\frac32} \GG_{\neq}),\partial_j^L|\nabla_L|^{-1}|{\nabla_L}|^{\frac12} |\nabla_{x,z}|^{-\frac12}\cA \GG_{\neq}\r.
\end{equation}
In analogy with \eqref{eq:I_1234}, we split this into $J_1,J_2,J_3,J_4$ with the following bounds
\begin{align}
    &\e^{\lambda\nu^{\frac13}t}\norm{J_1} \lesssim \|\nabla_L\cA  U^r_{\neq}\| \| \de_r^L|\nabla_{x,z}|^{\frac12}|\nabla_L|^{-\frac32}G_{\neq}\|_{L^\infty},\\
    &\e^{\lambda\nu^{\frac13}t}\norm{J_2} \lesssim \| \nabla_L U^r_{\neq}\|_{L^\infty}\| |\nabla_{x,z}|^{\frac12}|\nabla_L|^{-\frac32} \cA G_{\neq}\|+\nu^{-\frac13}\|\cA U^r_{\neq}\| \| \de_r^L|\nabla_{x,z}|^{\frac12}|\nabla_L|^{-\frac32}G_{\neq}\|_{L^\infty},\\
    &\e^{\lambda\nu^{\frac13}t}\norm{J_3} \lesssim \| \nabla_L|\nabla_{x,z}|^{\frac12} U^r_{\neq}\|_{L^\infty}\||\nabla_{x,z}|^{\frac12}|\nabla_L|^{-\frac32}\cA G_{\neq}\|,\\
    &\e^{\lambda\nu^{\frac13}t}\norm{J_4} \lesssim\| \nabla_L U^r_{\neq}\|_{L^\infty}\| \de_r^L|\nabla_L|^{-\frac32}\cA G_{\neq}\|.
\end{align}
Note here that $J_3$ and $J_4$ arise from distributing $|\nabla_{x,z}|^{\frac12}$ and not $\de_r^L$.
Now, using \eqref{eq:time_extraction}, \eqref{eq:property_A}, the multiplier $\cM_3$, and the bounds \eqref{eq:other norm bounds} we have
\begin{align}
    &|\l \cA  (\de_j^L U^r_{\neq} \de_r^L |\nabla_{x,z}|^{\frac12}|\nabla_L|^{-\frac32} \GG_{\neq}),\partial_j^L|\nabla_L|^{-1}|{\nabla_L}|^{\frac12} |\nabla_{x,z}|^{-\frac12}\cA \GG_{\neq}\rangle|\notag\\
    &\quad \lesssim \norm{\nabla_L\cA U^r_{\neq}} \l t \r^{-\frac32}\| \cA \GG_{\neq}\|^{\frac32}\| {\nabla_L} \cA \GG_{\neq}\|^{\frac12} \notag\\
    &\qquad +\nu^{-\frac13}\|\cA U^r_{\neq}\|\norm{\sqrt{-\frac{\dot\cM_3}{\cM_3}}\cA G_{\neq}}\| \cA \GG_{\neq}\|^{\frac12}\| {\nabla_L} \cA \GG_{\neq}\|^{\frac12}\notag\\
    &\qquad +\nu^{-\frac13}\|\cA U^r_{\neq}\| \l t \r^{-\frac32}\|\cA \GG_{\neq}\|^{\frac32}\| {\nabla_L} \cA \GG_{\neq}\|^{\frac12}.
\end{align}
This allows us to conclude from the bootstrap assumptions that
\begin{equation}
    \int_0^\infty |\l \cA  (\de_j^L U^r_{\neq} \de_r^L |\nabla_{x,z}|^{\frac12}|\nabla_L|^{-\frac32} \GG_{\neq}),\partial_j^L|\nabla_L|^{-1}|{\nabla_L}|^{\frac12} |\nabla_{x,z}|^{-\frac12}\cA \GG_{\neq}\rangle|\lesssim (\nu^{-\frac34}\eps)\eps^2,
\end{equation}
which is again consistent with Proposition \ref{prop:GGamma}. The largest contribution hereby comes from the term 
\begin{align}
    &\int_0^\infty \l s \r^{-\frac32}\|\nabla_L \cA    U^1_{\neq}\| \norm{\nabla_L\cA\GG_{\neq}}^{\frac12}\norm{\cA \GG_{\neq}}^{\frac32}\dd s\notag \\
    &\quad \lesssim \left(\int _0^\infty\l s \r^{-6}\dd s\right)^{\frac14}\left( \int _0^\infty \|\nabla_L\cA   U^1_{\neq}\|^2 \right)^{\frac12}\left(\int _0^\infty\norm{{\nabla_L}\cA \GG_{\neq}}^{2}\right)^{\frac14} \norm{\cA \GG_{\neq}}^{\frac32}_{L^\infty_tL^2}\notag\\
    &\quad \lesssim (\nu^{-\frac12}\eps)(\nu^{-\frac14}\eps^{\frac12})\eps^{\frac32}.
\end{align}
Altogether we recover \eqref{eq:Tstar-r-neqneq}.

We proceed with $\cT^2_\star(U_{\neq},U^2_{\neq})$ in \eqref{eq:Tstar-2-neqneq}. Using \eqref{eq:SimVar}, the corresponding term $\l \cA\;\cT^2_\star(U_{\neq},U^2_{\neq}) , \cA\GG_{\neq}\r$ reads
\begin{align}
\langle \cA |\nabla_L|^{\frac32}|\nabla_{x,z}|^{-\frac12}(|\nabla_{x,z}|^{\frac12}|\nabla_L|^{-\frac32} \GG_{\neq}\de_y^L |\nabla_{x,z}|^{\frac12}|\nabla_L|^{-\frac32} \GG_{\neq}),\cA\GG_{\neq}\rangle.
\end{align}
As done for the previous term in \eqref{eq:distributing derivatives} and \eqref{eq:I_1234}, we distribute derivatives to obtain
\begin{align}
    |\l \cA\;\cT^2_\star(U_{\neq},U^2_{\neq}) , \cA\GG_{\neq}\r|\lesssim \cI_1+\cI_2,
\end{align}
where
\begin{align}
\cI_1&:=  \| |\nabla_{x,z}|^{\frac12}|\nabla_L|^{-\frac32} \cA \GG_{\neq}\|\||\nabla_{x,z}|^{\frac12}\de_y^L \GG_{\neq}\|_{L^\infty}\||\nabla_{x,z}|^{-\frac12}\cA \GG_{\neq}\|\notag\\
& \quad + \| \nabla_{x,z}|\nabla_L|^{-\frac32} \GG_{\neq}\|_{L^\infty}\|\de_y^L \cA \GG_{\neq}\|\||\nabla_{x,z}|^{-\frac12}\cA \GG_{\neq}\|\notag\\
&\quad + \| |\nabla_{x,z}|^{\frac12}|\nabla_L|^{-\frac32} \GG_{\neq}\|_{L^\infty}\|\de_y^L \cA \GG_{\neq}\|\|\cA\GG_{\neq}\|,
\end{align}
and
\begin{align}
    \cI_2&:=  \|\cA \GG_{\neq}\|\|\nabla_{x,z} |\nabla_L|^{-\frac32}\de_y^L \GG_{\neq}\|_{L^\infty}\||\nabla_{x,z}|^{-\frac12}\cA \GG_{\neq}\|\notag\\
    &\quad +\|\cA \GG_{\neq}\|\||\nabla_{x,z}|^{\frac12}|\nabla_L|^{-\frac32}\de_y^L \GG_{\neq}\|_{L^\infty}\|\cA\GG_{\neq}\|\notag\\
    & \quad + \| |\nabla_{x,z}|^{\frac12} \GG_{\neq}\|_{L^\infty}\||\nabla_{x,z}|^{\frac12}\de_y^L |\nabla_L|^{-\frac32}\cA\GG_{\neq}\|\||\nabla_{x,z}|^{-\frac12}\cA \GG_{\neq}\|.
\end{align}
Using the multiplier $\cM_3$, \eqref{eq:time_extraction}, \eqref{eq:property_A}, and \eqref{eq:other norm bounds}, $\cI_1$ can be bounded as
\begin{align}
  \cI_1&\lesssim \nu^{-\frac13} \norm{\sqrt{-\frac{\dot \cM}{\cM}}\cA\GG_{\neq}}\|\cA \GG_{\neq}\|^2+ \l t\r ^{-\frac32} \| \cA\GG_{\neq}\|^2\|\nabla_L\cA  \GG_{\neq}\|,
\end{align}
while $\cI_2$ satisfies 
\begin{align}
    \cI_2&\lesssim \l t \r ^{-\frac12} \|\cA \GG_{\neq}\|^3 + \|\cA\GG_{\neq}\|^3.
\end{align}
With the bootstrap assumption \eqref{eq:bootstrap_nonzero_G} and the enhanced dissipation estimate \eqref{eq:enhanced_diss} we conclude that 
\begin{equation}
    \int_0^\infty|\l\cA\GG_{\neq}, \cA\;\cT^2_\star(U_{\neq},U^2_{\neq}) \r| \lesssim (\nu^{-\frac12}\eps)\eps^2,
\end{equation}
which is \eqref{eq:Tstar-2-neqneq}. 

We now turn our attention to the $(\neq, 0)$ interactions.
Regarding \eqref{eq:Tstar-2-neq0},
%$\cT^2_\star(U_{\neq},U^2_0)_{\neq}$, 
after using \eqref{eq:SimVar} it reads 
$$
\l \cA\GG_{\neq},\cA\cT^2_\star(U_{\neq},U^2_0)\r =\langle \cA\GG_{\neq}, \cA|\nabla_{x,z}|^{-\frac12} |\nabla_L|^{-\frac32}(|\nabla_L|^{\frac12}|\nabla_L|^{-\frac32}\GG_{\neq} \de_y |\de_z|^{\frac12}|\nabla|^{-\frac32}\GG_0)\rangle.
$$
In this case, when distributing derivatives to mimic the analysis for \eqref{eq:distributing derivatives} and \eqref{eq:I_1234}, we use that $\de_x$ derivatives vanish on zero modes. 
We have 
\begin{align}
    |\l \cA\GG_{\neq},\cA\cT^2_\star(U_{\neq},U^2_0)\r |&\lesssim \cI_1+\cI_2,
\end{align}
where
\begin{align}
    \cI_1 &:= \||\nabla_{x,z}|^{-\frac12} \cA\GG_{\neq}\|\|\cA\GG_{\neq}\|\| \de_y |\de_z|^{\frac32}|\nabla|^{-\frac32}\GG_0\|_{L^\infty}\notag\\
    &\quad +\| \cA\GG_{\neq}\|\| \cA\GG_{\neq}\|\| \de_y|\de_z|^{\frac12}|\nabla|^{-\frac32}\GG_0\|_{L^\infty}\notag\\
    &\quad +\||\nabla_{x,z}|^{-\frac12} \cA\GG_{\neq}\|\| \e^{\lambda \nu^{\frac13} t}|\nabla_{x,z}|^{\frac12}\GG_{\neq}\|_{L^\infty} \| \de_y |\de_z|^{\frac12}|\nabla|^{-\frac32}\GG_0\|_{H^{2m}}
\end{align}
and
\begin{align}
    \cI_2 &:= \||\nabla_{x,z}|^{-\frac12} \cA\GG_{\neq}\|\|\cA |\nabla_{x,z}|^{\frac12}|\nabla_L|^{-\frac32}\GG_{\neq}\|\| \de_y |\de_z|^{\frac12}\GG_0\|_{L^\infty}\notag\\
    &\quad + \| |\nabla_{x,z}|^{-\frac12} \cA\GG_{\neq}\|\| \e^{\lambda\nu^{\frac13}t}|\de_z|^{\frac12}|\nabla_{x,z}|^{\frac12}|\nabla_L|^{-\frac32}\GG_{\neq}\|_{L^\infty}\| \de_y \GG_0\|_{H^{2m}}\notag\\
    &\quad + \| |\de_z|^{\frac12}|\nabla_{x,z}|^{-\frac12} \cA\GG_{\neq}\|\| \e^{\lambda\nu^{\frac13}t}|\nabla_{x,z}|^{\frac12}|\nabla_L|^{-\frac32}\GG_{\neq}\|_{L^\infty}\| \de_y \GG_0\|_{H^{2m}}.
\end{align}
Each term of $\cI_1$ can be bounded by $\| \cA\GG_{\neq}\|^2\| \GG_0\|_{H^{2m}}$, hence by the bootstrap assumptions 
\begin{equation}
    \int_0^\infty \cI_1\lesssim (\nu^{-\frac13}\eps)\eps^2.
\end{equation}
For $\cI_2$, with \eqref{eq:time_extraction} and by construction of $\cM_3$, we have 
\begin{equation}
    \cI_2 \lesssim \|\cA\GG_{\neq}\|\norm{\sqrt{-\frac{\dot \cM}{\cM}}\cA\GG_{\neq}} \|\GG_0\|_{H^{2m}}+\l t \r^{-\frac32} \| \cA\GG_{\neq}\|^2\| \nabla \GG_0\|_{H^{2m}},
\end{equation}
and again by bootstrap assumptions 
\eqref{eq:bootstrap_nonzero_G} and \eqref{eq:bootstrap_zero_g}
\begin{equation}
    \int_0^\infty\cI_2 \lesssim (\nu^{-\frac16}\eps)\eps^2+(\nu^{-\frac12}\eps)\eps^2,
\end{equation}
which is \eqref{eq:Tstar-2-neq0}.
Moving on to $\cT^3_\star(U_{\neq},U^2_0)_{\neq}$ in \eqref{eq:Tstar-3-neq0} we have 
\begin{align}
\l \cA\GG_{\neq},\cA\cT^3_\star(U_{\neq},U^2_0)\r
&=\langle|\nabla_{x,z}|^{-\frac12} \cA\GG_{\neq}, \cA |\nabla_L|^{\frac32}(U_{\neq}^3\de_z|\de_z|^{\frac12}|\nabla|^{-\frac32}\GG_0)\rangle,
\end{align}
and analogously to $\cT^3_\star(U_{\neq},U^2_{\neq})_{\neq}$ we have
\begin{align}
    |\l \cA\GG_{\neq}, \cA\cT^3_\star(U_{\neq},U^2_0)\r |&\lesssim \cI_1+\cI_2,
\end{align}
where 
\begin{align}
    \cI_1&:=\| |\nabla_{x,z}|^{-\frac12}|\nabla_L|^{\frac12} \cA\GG_{\neq}\|\| \nabla_L\cA U_{\neq}^3\|\|\de_z|\de_z|^{\frac12}|\nabla|^{-\frac32}\GG_0\|_{L^\infty}\notag\\
    &\quad +\| |\nabla_{x,z}|^{-\frac12}|\nabla_L|^{\frac12} \cA\GG_{\neq}\|\| \e^{\lambda\nu^{\frac13}t}\nabla_L U_{\neq}^3\|_{L^\infty}\| \de_z|\de_z|^{\frac12}|\nabla|^{-\frac32}\GG_0\|_{H^{2m}}
\end{align}
and
\begin{align}
    \cI_2 &:=\| |\nabla_{x,z}|^{-\frac12}|\nabla_L|^{\frac12} \cA\GG_{\neq}\|\| \cA U_{\neq}^3\|\|\de_z|\de_z|^{\frac12}|\nabla|^{-\frac12}\GG_0\|_{L^\infty}\notag\\
    &\quad +\| |\de_z|^{\frac12}|\nabla_{x,z}|^{-\frac12}|\nabla_L|^{\frac12} \cA\GG_{\neq}\|\| \e^{\lambda\nu^{\frac13}t}U_{\neq}^3\|_{L^\infty}\|\de_z|\nabla|^{-\frac12}\GG_0\|_{H^{2m}}\notag\\
     &\quad +\| |\nabla_{x,z}|^{-\frac12}|\nabla_L|^{\frac12} \cA\GG_{\neq}\|\|\e^{\lambda\nu^{\frac13}t}|\de_z|^{\frac12}U_{\neq}^3\|_{L^\infty}\|\de_z|\nabla|^{-\frac12}\GG_0\|_{H^{2m}}.
\end{align}
By interpolation we have
\begin{align}
    \cI_1\lesssim \| \nabla_L\cA\GG_{\neq}\|^{\frac12}\|\cA\GG_{\neq}\|^{\frac12}\| \nabla_L\cA U_{\neq}^3\|\|\GG_0\|_{H^{2m}},
\end{align}
and after integration in time, the bootstrap assumptions in Theorem \ref{thm:bootstrap_step} lead to
\begin{align}
    \int_0^\infty\cI_1\lesssim (\nu^{-\frac56}\eps)\eps^2.
\end{align}
Again, by interpolation there holds that
\begin{align}
    \cI_2&\lesssim \| \nabla_L\cA\GG_{\neq}\|^{\frac12}\|\cA\GG_{\neq}\|^{\frac12}\| \cA U_{\neq}^3\|\|\GG_0\|_{H^{2m}}\notag\\
    &\quad+\| \nabla_L\cA\GG_{\neq}\|^{\frac12}\|\cA\GG_{\neq}\|^{\frac12}\| \cA U_{\neq}^3\|\|\nabla\GG_0\|_{H^{2m}}^{\frac12}\|\GG_0\|_{H^{2m}}^{\frac12}.
\end{align}
Consequently using the bootstrap assumptions yields
\begin{align}
    \int_0^\infty \cI_2\lesssim (\nu^{-\frac12}\eps)\eps^2+(\nu^{-\frac23}\eps)\eps^2,
\end{align}
as needed for \eqref{eq:Tstar-3-neq0}.
The next terms to be analysed is $\cT^1_\star (U_0,U^2_{\neq})_{\neq}$, which reads
\begin{align}
   \l \cA\GG_{\neq},\cA\,\cT^1_\star (U_0,U^2_{\neq})\r  = \langle \cA\GG_{\neq}, \cA|\nabla_{x,z}|^{-\frac12} |\nabla_L|^{\frac32}(U_0^1 \de_x |\nabla_{x,z}|^{\frac12}|\nabla_L|^{-\frac32} \GG_{\neq})\rangle.
\end{align}
As in \eqref{eq:distributing derivatives} and \eqref{eq:I_1234}, we distribute $\nabla_L$ on the product and use that $\de_xU^1_0=0$. We have
\begin{align}
    |\l \cA\GG_{\neq},\cA\,\cT^1_\star (U_0,U^2_{\neq})\r |\lesssim \cI_1+\cI_2,
\end{align}
where
\begin{align}
    \cI_1 &=\|\nabla_{x,z}|^{-\frac12} |\nabla_L|^{\frac12}\cA\GG_{\neq}\|\| U_0^1 \|_{H^{2m}}\|\e^{\lambda\nu^{\frac13}t}\de_x|\nabla_{x,z}|^{\frac12}|\nabla_L|^{-\frac12} \GG_{\neq}\|_{L^\infty}\notag\\
    &\quad + \|\de_x|\nabla_{x,z}|^{-\frac12} |\nabla_L|^{\frac12}\cA\GG_{\neq}\|\| U_0^1\|_{L^\infty} \||\nabla_{x,z}|^{\frac12}|\nabla_L|^{-\frac12} \cA\GG_{\neq}\|
\end{align}
and
\begin{align}
    \cI_2&=\| |\nabla_{x,z}|^{-\frac12} |\nabla_L|^{\frac12}\cA\GG_{\neq}\|\| \nabla U_0^1 \|_{H^{2m}}\|\e^{\lambda\nu^{\frac13}t}\de_x |\nabla_{x,z}|^{\frac12}|\nabla_L|^{-\frac32} \GG_{\neq}\|_{L^\infty}\notag\\
    &\quad +\| |\nabla_{x,z}|^{-\frac12} |\nabla_L|^{\frac12}\cA\GG_{\neq}\|\| \nabla U_0^1\|_{L^\infty}\| \de_x |\nabla_{x,z}|^{\frac12}|\nabla_L|^{-\frac32} \cA\GG_{\neq}\|.
\end{align}
By \eqref{eq:other norm bounds},
\begin{align}
    \cI_1\lesssim \|\nabla_L\cA\GG_{\neq}\|^{\frac12}\|\cA\GG_{\neq}\|^{\frac12}\| U_0^1 \|_{H^{2m}}\|\cA \GG_{\neq}\|+\|\nabla_L\cA\GG_{\neq}\|\| U_0^1 \|_{H^{2m}}\|\cA \GG_{\neq}\|, 
\end{align}
which, after time integration gives
\begin{align}
    \int_0^\infty \cI_1 \lesssim (\nu^{-\frac12}\eps)\eps^2+(\nu^{-\frac23}\eps)\eps^2.
\end{align}
Regarding $\cI_2$, interpolation, the multiplier $M_3$ defined in \eqref{def:M3} and \eqref{eq:time_extraction} give
\begin{align}
    \cI_2&\lesssim  \|\nabla_L\cA\GG_{\neq}\|^{\frac12}\|\cA\GG_{\neq}\|^{\frac12}\left[\l t\r^{-\frac32}\|\cA\GG_{\neq}\|\| \nabla U_0^1 \|_{H^{2m}}+\| U_0^1 \|_{H^{2m}}\norm{\sqrt{-\frac{\dot \cM}{\cM}}\cA\GG_{\neq}}\right],
\end{align}
hence via bootstrap 
\begin{align}
    \int_0^\infty \cI_2&\lesssim (\nu^{-\frac34}\eps)\eps^2+(\nu^{-\frac13}\eps)\eps^2,
\end{align}
which is \eqref{eq:Tstar-1-0neq}.
For $\cT^2_\star(U_0,U^2_{\neq})_{\neq}$ we have, recalling the change of variables \eqref{eq:SimVar},
$$
   \l \cA\GG_{\neq},\cA\,\cT^2_\star(U_0,U^2_{\neq})\r  = \l \cA\GG_{\neq}, \cA|\nabla_{x,z}|^{-\frac12} |\nabla_L|^{\frac32}(|\de_z|^{\frac12}|\nabla|^{-\frac32}\GG_0 \de_y^L |\nabla_{x,z}|^{\frac12}|\nabla_L|^{-\frac32} \GG_{\neq})\r.
$$
As in $\cT^2_\star(U_{\neq},U^2_{0})_{\neq}$, we distribute $|\nabla_L|^{\frac32}$ inside the product and use that $\de_x\GG_0=0$. This gives
\begin{align}
    |\l \cA\GG_{\neq},\cA\,\cT^2_\star(U_0,U^2_{\neq})\r |\lesssim \cI_1+\cI_2,
\end{align}
where
\begin{align}
\cI_1&= \||\nabla_{x,z}|^{-\frac12}\cA\GG_{\neq}\|\|  \de_z|\nabla|^{-\frac32} \GG_0\|_{H^{2m}}\|\e^{\lambda\nu^{\frac13}t}|\nabla_{x,z}|^{\frac12}\de_y^L \GG_{\neq}\|_{L^\infty}\notag\\
& \quad +\||\nabla_{x,z}|^{-\frac12}\cA\GG_{\neq}\| \| \de_z|\nabla|^{-\frac32} \GG_0\|_{L^\infty}\|\de_y^L \cA\GG_{\neq}\|\notag\\
&\quad +\|\cA\GG_{\neq}\|\| |\de_z|^{\frac12}|\nabla|^{-\frac32} \GG_0\|_{L^\infty}\|\de_y^L \cA\GG_{\neq}\|
\end{align}
and
\begin{align}
    \cI_2&=  \||\nabla_{x,z}|^{-\frac12}\cA\GG_{\neq}\|\| \GG_0\|_{H^{2m}}\|\e^{\lambda\nu^{\frac13}t}|\de_z|^{\frac12}|\nabla_{x,z}|^{\frac12}|\nabla_L|^{-\frac32}\de_y^L \GG_{\neq}\|_{L^\infty}\notag\\
    &\quad +\||\de_z|^{\frac12}|\nabla_{x,z}|^{-\frac12}\cA\GG_{\neq}\|\|\GG_0\|_{H^{2m}}\|\e^{\lambda\nu^{\frac13}t}|\nabla_{x,z}|^{\frac12}|\nabla_L|^{-\frac32}\de_y^L \GG_{\neq}\|_{L^\infty}\notag\\
    & \quad + \||\nabla_{x,z}|^{-\frac12}\cA\GG_{\neq}\|\| |\de_z|^{\frac12} \GG_0\|_{L^\infty}\||\nabla_{x,z}|^{\frac12}\de_y^L |\nabla_L|^{-\frac32}\cA\GG_{\neq}\|.
\end{align}
All three terms in $\cI_1$ can be bounded by
\begin{align}
    \cI_1 \lesssim \|\nabla_L\cA\GG_{\neq}\|\|\cA\GG_{\neq}\|\| \GG_0 \|_{H^{2m}}, 
\end{align}
which gives
\begin{align}
    \int_0^\infty \cI_1 \lesssim (\nu^{-\frac23}\eps)\eps^2.
\end{align}
Regarding $\cI_2$, after extracting the time decay using \eqref{eq:time_extraction} and \eqref{eq:other norm bounds}, we have 
\begin{equation}
    \cI_{2}\lesssim \l t\r^{-\frac12} \|\cA\GG_{\neq}\|^2\| \GG_0 \|_{H^{2m}}+\|\cA\GG_{\neq}\|^2\| \GG_0 \|_{H^{2m}},
\end{equation}
hence using the bootstrap assumptions in Theorem \ref{thm:bootstrap_step} we obtain
\begin{equation}
     \int_0^\infty \cI_{2}\lesssim (\nu^{-\frac29}\eps)\eps^2+(\nu^{-\frac13}\eps)\eps^2,
\end{equation}
and \eqref{eq:Tstar-2-0neq} is proven.
For the last term in \eqref{eq:Tstar-3-0neq}, $\cT^3_\star(U_0,U^2_{\neq})_{\neq}$  has a similar structure as $\cT^1_\star (U_0,U^2_{\neq})$, since 
\begin{align}
   \l \cA\GG_{\neq},\cA\,\cT^3_\star(U_0,U^2_{\neq})\r  = \langle \cA\GG_{\neq}, \cA|\nabla_{x,z}|^{-\frac12} |\nabla_L|^{\frac32}(U_0^3 \de_z |\nabla_{x,z}|^{\frac12}|\nabla_L|^{-\frac32} \GG_{\neq})\rangle, 
\end{align}
with the exception that $\de_z$ can fall onto $U^3_0$. Hence, the analysis is similar to the one of $\cT^r_\star(U_{\neq},U^2_{\neq})_{\neq}$, from which we deduce
\eqref{eq:Tstar-3-0neq}. This concludes the proof of \eqref{eq:T_star_bds}.

\subsubsection{Nonlinear terms analysis: $\cT_\circ(U,\Theta)_{\neq}$}\label{sssec:cTcircGGamma}
We shift our analysis to the nonlinear term  $\cT_\circ(U,\Theta)$ and the estimate \eqref{eq:T_circ_bds}. As done for $\cT_\star(U,U^2)$, we split this term into different parts according to the modes interacting and the operator $U\cdot\nabla_L$, more precisely 
\begin{equation}\label{def:cTcirc_decomposition}
    \cT_\circ(U,\Theta)_{\neq}= \sum_{r,\kappa_1,\kappa_2}\cT^r_\circ(U_{\kappa_1},\Theta_{\kappa_2})_{\neq},\qquad \cT^r_\circ(U_{\kappa_1},\Theta_{\kappa_2})_{\neq}:=-|\nabla_{x,z}|^{\frac12}|\nabla_L|^{\frac12}(U^r_{\kappa_1}\de_r^L\Theta_{\kappa_2})_{\neq},
\end{equation}
for $r\in \{1,2,3\}$ and $\kappa_1,\kappa_2 \in\{0,\neq\}$, where the convention is again $\de_1^L=\de_x$, $\de_2^L=\de_y^L$ and $\de_3^L=\de_z$.
We prove the following bounds
\begin{subequations}
\begin{alignat}{6}
    \int_0^\infty|\l \cA \GG_{\neq},\cA \,\cT^r_\circ(U_{\neq},\Theta_{\neq}) \r| &\lesssim (\nu^{-\frac23}\eps)\eps^2,\quad r=1,3,\label{eq:Tcirc-r-neqneq}\\
    \int_0^\infty|\l \cA\GG_{\neq},\cA\,\cT^2_\circ(U_{\neq},\Theta_{\neq})\r| &\lesssim (\nu^{-\frac12}\eps)\eps^2,\label{eq:Tcirc-2-neqneq}\\
    \int_0^\infty|\l\cA\GG_{\neq}, \cA\,\cT^2_\circ(U_{\neq},\Theta_0) \r| 
    &\lesssim (\nu^{-\frac12}\eps)\eps^2,\label{eq:Tcirc-2-neq0}\\
    \int_0^\infty|\l \cA\GG_{\neq},\cA\,\cT^3_\circ(U_{\neq},\Theta_0) \r| 
    &\lesssim (\nu^{-\frac23}\eps)\eps^2,\label{eq:Tcirc-3-neq0}\\
    \int_0^\infty|\l \cA\GG_{\neq},\cA\,\cT^r_\circ(U_0,\Theta_{\neq}) \r| 
    &\lesssim (\nu^{-\frac23}\eps)\eps^2,\quad r=1,3,\label{eq:Tcirc-r-0neq}\\
    \int_0^\infty|\l\cA\GG_{\neq}, \cA\,\cT^2_\circ(U_0,\Theta_{\neq}) \r| 
    &\lesssim (\nu^{-\frac23}\eps)\eps^2.\label{eq:Tcirc-2-0neq}
\end{alignat}
\end{subequations}
Note that $\l \cA\GG_{\neq}, \cA\, \cT^1_\circ(U_{\neq},\Theta_0)\r=0$ because $\Theta_0$ is $x$-independent, and hence this term does not appear above.
Let us consider $\cT^r_\circ(U_{\neq},\Theta_{\neq})_{\neq}$ for $r=1$ or $r=3$. 
From the symmetric change of variables \eqref{eq:SimVar}, we get
\begin{equation}
    \l \cA\GG_{\neq}, \cA\, \cT^r_\circ(U_{\neq},\Theta_{\neq})\r = \l \cA\GG_{\neq}, \cA|\nabla_{x,z}|^{\frac12}|\nabla_L|^{\frac12}(U^r_{\neq}\de_r^L|\nabla_{x,z}|^{-\frac12}|\nabla_L|^{-\frac12}\Gamma_{\neq})\r. 
\end{equation}
Analogously to \eqref{eq:distributing derivatives}, using \eqref{eq:other norm bounds}, \eqref{eq:prod_ineq_A} and the Sobolev embedding, we obtain
$$
    |\l \cA\GG_{\neq}, \cA\, \cT^r_\circ(U_{\neq},\Theta_{\neq})\r|\lesssim \||\nabla_{x,z}|^{\frac12}|\nabla_L|^{\frac12}\cA\GG_{\neq}\|\|\cA\left(U^r_{\neq}\Gamma_{\neq}\right)\|
    \lesssim \norm{\nabla_L \cA\GG_{\neq}}\norm{\cA U^r_{\neq}}\norm{\cA\Gamma_{\neq}}.
$$
By the bootstrap assumptions in Theorem \ref{thm:bootstrap_step}, we end up with \eqref{eq:Tcirc-r-neqneq}.
The term involving $\cT^2_\circ(U_{\neq},\Theta_{\neq})_{\neq}$ in
\eqref{eq:Tcirc-2-neqneq} can be written as
$$
    \l \cA\GG_{\neq}, \cA\, \cT^2_\circ(U_{\neq},\Theta_{\neq})\r = \l \cA\GG_{\neq}, \cA|\nabla_{x,z}|^{\frac12}|\nabla_L|^{\frac12}(|\nabla_{x,z}|^{\frac12}|\nabla_L|^{-\frac32}\GG_{\neq}\de_y^L|\nabla_{x,z}|^{-\frac12}|\nabla_L|^{-\frac12}\Gamma_{\neq})\r. 
$$
Distributing $|\nabla_{x,z}|^{\frac12}|\nabla_L|^{\frac12}$ inside the product, mimicking \eqref{eq:distributing derivatives} and \eqref{eq:I_1234}, we arrive  at
\begin{equation}
    |\l \cA\GG_{\neq}, \cA\, \cT^2_\circ(U_{\neq},\Theta_{\neq})\r |\lesssim \cI_1+\cI_2,
\end{equation}
where
\begin{align}
    \cI_1&=\norm{\cA\GG_{\neq}}\left[\norm{\nabla_{x,z}|\nabla_L|^{-1}\cA\GG_{\neq}}\norm{\de_y^L|\nabla_{x,z}|^{-\frac12}|\nabla_L|^{-\frac12}\Gamma_{\neq}}_{L^\infty}\right.\notag\\
    &\qquad\qquad\qquad \left.+\norm{\nabla_{x,z}|\nabla_L|^{-1}\GG_{\neq}}_{L^\infty}\norm{\de_y^L|\nabla_{x,z}|^{-\frac12}|\nabla_L|^{-\frac12}\cA\Gamma_{\neq}}\right]\notag\\
    &\quad +\norm{\cA\GG_{\neq}}\left[\norm{|\nabla_{x,z}|^{\frac12}|\nabla_L|^{-1}\cA\GG_{\neq}}\norm{\de_y^L|\nabla_L|^{-\frac12}\Gamma_{\neq}}_{L^\infty}\right.\notag\\
    &\qquad\qquad\qquad \left.+\norm{|\nabla_{x,z}|^{\frac12}|\nabla_L|^{-1}\GG_{\neq}}_{L^\infty}\norm{\de_y^L|\nabla_L|^{-\frac12}\cA\Gamma_{\neq}}\right]
\end{align}
and
\begin{align}
    \cI_2&=\norm{\cA\GG_{\neq}}\left[\norm{\nabla_{x,z}|\nabla_L|^{-\frac32}\cA\GG_{\neq}}\norm{\de_y^L|\nabla_{x,z}|^{-\frac12}\Gamma_{\neq}}_{L^\infty}\right.\notag\\
    &\qquad\qquad\qquad \left.+\norm{\nabla_{x,z}|\nabla_L|^{-\frac32}\GG_{\neq}}_{L^\infty}\norm{\de_y^L|\nabla_{x,z}|^{-\frac12}\cA\Gamma_{\neq}}\right]\notag\\
    &\quad  + \norm{\cA\GG_{\neq}}\left[\norm{|\nabla_{x,z}|^{\frac12}|\nabla_L|^{-\frac32}\cA\GG_{\neq}}\norm{\de_y^L\Gamma_{\neq}}_{L^\infty}+\norm{|\nabla_{x,z}|^{\frac12}|\nabla_L|^{-\frac32}\GG_{\neq}}_{L^\infty}\norm{\de_y^L\cA\Gamma_{\neq}}\right].
\end{align}
Using the definitions of $\cA$ and $\cM_3$, \eqref{eq:time_extraction}, and \eqref{eq:other norm bounds} we bound $\cI_1$ as
\begin{align}
\cI_1&\lesssim \nu^{-\frac16}\norm{\cA\GG_{\neq}}^2\norm{\cA\Gamma_{\neq}}+\l t\r^{-1}\norm{\cA\GG_{\neq}}^2\norm{\nabla_L\cA\Gamma_{\neq}}^{\frac12}\norm{\cA\Gamma_{\neq}}^{\frac12}\\
&\qquad +\nu^{-\frac16}\norm{\cA\GG_{\neq}}\norm{\sqrt{-\frac{\dot\cM}{\cM}}\cA\GG_{\neq}}\norm{\cA\Gamma_{\neq}}
\end{align}
and $\cI_2$ as  
\begin{align}
\cI_2&\lesssim \nu^{-\frac13}\norm{\cA\GG_{\neq}}\norm{\sqrt{-\frac{\dot\cM}{\cM}}\cA\GG_{\neq}}\norm{\cA\Gamma_{\neq}}+\l t\r^{-\frac32}\norm{\cA\GG_{\neq}}^2\norm{\nabla_L\cA\Gamma_{\neq}}.
\end{align}
Using the bootstrap assumptions in Theorem \ref{thm:bootstrap_step}, we obtain \eqref{eq:Tcirc-2-neqneq}.

Turning to $(\neq,0)$ interactions, the term involving  $\cT^2_\circ(U_{\neq},\Theta_{0})_{\neq}$ is treated similarly as $\cT^2_\circ(U_{\neq},\Theta_{\neq})_{\neq}$. By \eqref{eq:U2Theta_recover} we have
\begin{align}
     \l \cA\GG_{\neq}, \cA\, \cT^2_\circ(U_{\neq},\Theta_{0})\r &= \l \cA\GG_{\neq}, \cA|\nabla_{x,z}|^{\frac12}|\nabla_L|^{\frac12}(|\nabla_{x,z}|^{\frac12}|\nabla_L|^{-\frac32}\GG_{\neq}\de_y|\de_z|^{-\frac12}|\nabla|^{-\frac12}\Gamma_{0})\r\\
     &\quad +\l \cA\GG_{\neq}, \cA|\nabla_{x,z}|^{\frac12}|\nabla_L|^{\frac12}(|\nabla_{x,z}|^{\frac12}|\nabla_L|^{-\frac32}\GG_{\neq}\de_y\overline\Theta_{0})\r,
\end{align}
and distribute
 $|\nabla_{x,z}|^{\frac12}|\nabla_L|^{\frac12}$ inside the product to obtain
\begin{equation}
    |\l \cA\GG_{\neq}, \cA\, \cT^2_\circ(U_{\neq},\Theta_{0})\r|\lesssim \cI_1+\cI_2 +\cI_3.
\end{equation}
Here
\begin{align}
    \cI_1&=\norm{\cA\GG_{\neq}}\left[\|\nabla_{x,z}|\nabla_L|^{-1}\cA\GG_{\neq}\|\|\de_y|\de_z|^{-\frac12}|\nabla|^{-\frac12}\Gamma_0\|_{L^\infty}\right.\notag\\
    &\qquad\qquad\qquad \left.+\|\e^{\lambda\nu^\frac13 t}\nabla_{x,z}|\nabla_L|^{-1}\GG_{\neq}\|_{L^\infty}\|\de_y|\de_z|^{-\frac12}|\nabla|^{-\frac12}\Gamma_0\|_{H^{2m}}\right]\notag\\
    &\quad +\norm{\cA\GG_{\neq}}\left[\||\nabla_{x,z}|^{\frac12}|\nabla_L|^{-1}\cA\GG_{\neq}\|\|\de_y|\nabla|^{-\frac12}\Gamma_0\|_{L^\infty}\right.\notag\\
    &\qquad\qquad\qquad \left.+\|\e^{\lambda\nu^\frac13 t}|\nabla_{x,z}|^{\frac12}|\nabla_L|^{-1}\GG_{\neq}\|_{L^\infty}\|\de_y|\nabla|^{-\frac12}\Gamma_0\|_{H^{2m}}\right],
\end{align}
while
\begin{align}
    \cI_2&=\norm{\cA\GG_{\neq}}\left[\|\nabla_{x,z}|\nabla_L|^{-\frac32}\cA\GG_{\neq}\|\|\de_y|\de_z|^{-\frac12}\Gamma_0\|_{L^\infty}+\|\e^{\lambda\nu^\frac13 t}\nabla_{x,z}|\nabla_L|^{-\frac32}\GG_{\neq}\|_{L^\infty}\|\de_y|\de_z|^{-\frac12}\Gamma_0\|_{H^{2m}}\right]\notag\\
    &\quad  + \|\cA\GG_{\neq}\|\left[\||\nabla_{x,z}|^{\frac12}|\nabla_L|^{-\frac32}\cA\GG_{\neq}\|\|\de_y\Gamma_0\|_{L^\infty}+\|\e^{\lambda\nu^\frac13 t}|\nabla_{x,z}|^{\frac12}|\nabla_L|^{-\frac32}\GG_{\neq}\|_{L^\infty}\|\de_y\Gamma_0\|_{H^{2m}}\right]
\end{align}
and
\begin{equation}
\cI_3
=\norm{|\nabla_L|^\frac12\cA \GG_{\neq}}\left[\|\nabla_{x,z}|\nabla_L|^{-\frac32}\cA\GG_{\neq}\|\|\de_y\overline{\Theta}_0\|_{L^\infty}+\|\e^{\lambda\nu^\frac13 t}\nabla_{x,z}|\nabla_L|^{-\frac32}\GG_{\neq}\|_{L^\infty}\|\de_y\overline{\Theta}_0\|_{H^{2m}}\right].
\end{equation}
As done before for $\cT^2_\circ(U_{\neq},\Theta_{\neq})_{\neq}$, we further bound $\cI_1$ with
$$
\cI_1\lesssim \norm{\cA\GG_{\neq}}^2\norm{\Gamma_0}_{H^{2m}}+\l t\r^{-1}\norm{\cA\GG_{\neq}}^2\norm{\nabla\Gamma_0}^{\frac12}_{H^{2m}}\norm{\Gamma_0}^{\frac12}_{H^{2m}}+\norm{\cA\GG_{\neq}}\norm{\sqrt{-\frac{\dot\cM}{\cM}}\cA\GG_{\neq}}\norm{\Gamma_0}_{H^{2m}},
$$
the term $\cI_2$ with  
\begin{align}
\cI_2&\lesssim \norm{\cA\GG_{\neq}}\norm{\sqrt{-\frac{\dot\cM}{\cM}}\cA\GG_{\neq}}\norm{\Gamma_0}_{H^{2m}}+\l t\r^{-\frac32}\norm{\cA\GG_{\neq}}^2\norm{\nabla\Gamma_0}_{H^{2m}},
\end{align}
and $\cI_3$ with
\begin{align}
\cI_3
\lesssim\norm{\cA\GG_{\neq}}^\frac12\norm{\nabla_L\cA\GG_{\neq}}^\frac12\norm{\sqrt{-\frac{\dot\cM}{\cM}}\cA\GG_{\neq}}\norm{\overline{\Theta}_0}_{H^{2m+1}}.
\end{align}
By the bootstrap assumptions 
\begin{equation}
    \int_0^\infty |\l \cA\GG_{\neq}, \cA\, \cT^2_\circ(U_{\neq},\Theta_{0})\r| \lesssim (\nu^{-\frac13}\eps)\eps^2+(\nu^{-\frac12}\eps)\eps^2 +(\nu^{-\frac13}\eps)\eps^2,
\end{equation}
which is consistent with \eqref{eq:Tcirc-2-neq0}.
The last $(\neq,0)$ term $\cT^3_\circ(U_{\neq},\Theta_{0})_{\neq}$ reads
 \begin{equation}
     \l \cA\GG_{\neq}, \cA\, \cT^3_\circ(U_{\neq},\Theta_{0})\r = \l \cA\GG_{\neq}, \cA|\nabla_{x,z}|^{\frac12}|\nabla_L|^{\frac12}(U^3_{\neq}\de_z|\de_z|^{-\frac12}|\nabla|^{-\frac12}\Gamma_{0})\r,
 \end{equation}
 and can be treated as $\cT^r_\circ(U_{\neq},\Theta_{\neq})_{\neq}$, hence \eqref{eq:Tcirc-3-neq0} follows similarly.

The remaining $(0,\neq)$ interactions are treated similarly to the $(\neq,\neq)$ ones. $\cT^r_\circ(U_0,\Theta_{\neq})_{\neq}$, where $r=1$ or $r=3$, is
\begin{equation}
    \l \cA\GG_{\neq}, \cA\, \cT^r_\circ(U_0,\Theta_{\neq})\r = \l \cA\GG_{\neq}, \cA|\nabla_{x,z}|^{\frac12}|\nabla_L|^{\frac12}(U^r_{0}\de_r^L|\nabla_{x,z}|^{-\frac12}|\nabla_L|^{-\frac12}\Gamma_{\neq})\r.
\end{equation}
The analysis is as for $\cT^r_\circ(U_{\neq},\Theta_{\neq})_{\neq}$, hence we deduce 
\begin{align}
    |\l \cA\GG_{\neq}, \cA\, \cT^r_\circ(U_0,\Theta_{\neq})\r|\lesssim \norm{\nabla_L\cA\GG_{\neq}}\norm{ U^r_{0}}_{H^{2m}}\norm{\cA\Gamma_{\neq}},
\end{align}
which by bootstrap leads to \eqref{eq:Tcirc-r-0neq}.
% \begin{equation}
%     \int_0^t|\l \cA\GG_{\neq}, \cA\, \cT^r_\circ(U_0,\Theta_{\neq})\r|\dd s \lesssim (\nu^{-\frac23}\eps)\eps^2.
% \end{equation}
The last term involving $\cT^2_\circ(U_0,\Theta_{\neq})_{\neq}$ can be treated again similarly to $\cT^2_\circ(U_{\neq},\Theta_{\neq})_{\neq}$. Therefore we deduce that 
\begin{equation}
    \l \cA\GG_{\neq}, \cA\, \cT^2_\circ(U_0,\Theta_{\neq})\r = \l \cA\GG_{\neq}, \cA|\nabla_{x,z}|^{\frac12}|\nabla_L|^{\frac12}(|\de_z|^{\frac12}|\nabla|^{-\frac32}\GG_0\de_y^L|\nabla_{x,z}|^{-\frac12}|\nabla_L|^{-\frac12}\Gamma_{\neq})\r,
\end{equation}
which can be bounded by 
\begin{equation}
    | \l \cA\GG_{\neq}, \cA\, \cT^2_\circ(U_0,\Theta_{\neq})\r|\lesssim \norm{\cA\GG_{\neq}}\norm{\GG_0}_{H^{2m}}\left[\norm{\nabla_L\cA\Gamma_{\neq}}^{\frac12}\norm{\cA\Gamma_{\neq}}^{\frac12}+\norm{\nabla_L\cA\Gamma_{\neq}}\right].
\end{equation}
Thanks to the bootstrap assumptions of Theorem \ref{thm:bootstrap_step} we obtain \eqref{eq:Tcirc-2-0neq} and conclude the proof of \eqref{eq:T_circ_bds}.

\subsubsection{Nonlinear terms analysis: $\de_y^L\cP_\star(U,U)_{\neq}$}\label{sssec:cPGGamma}
To conclude the proof of Lemma \ref{lem:GGammanonlin}, it remains to prove \eqref{eq:P_star_bds}.
As done for the other two nonlinear terms above, we further divide $\cP_\star(U,U)_{\neq}$ into different components following the structure of $\cP(U,U)$ given by $(\nabla_L\otimes\nabla_L)(U\otimes U)$. Namely 
\begin{equation}\label{def:cPstar_decomposition}
    \cP_\star(U,U)_{\neq}= \sum_{i,j,\kappa_1,\kappa_2}\cP_\star^{i,j}(U_{\kappa_1},U_{\kappa_2})_{\neq},\quad \cP_\star^{i,j}(U_{\kappa_1},U_{\kappa_2})_{\neq}:=-|\nabla_{x,z}|^{-\frac12}|\nabla_L|^{-\frac12}(\de_i^LU^j_{\kappa_1}\de_j^LU^i_{\kappa_2})_{\neq},
\end{equation}
for $i,j\in \{1,2,3\}$ and $\kappa_1,\kappa_2 \in\{0,\neq\}$ with the convention $\de_1^L=\de_x$, $\de_2^L=\de_y^L$, and $\de_3^L=\de_z$. 
In view of the symmetry 
\begin{equation}\label{eq:pressure_symmetry}
    \cP_\star^{i,j}(U_{\kappa_1},U_{\kappa_2})=\cP_\star^{j,i}(U_{\kappa_2},U_{\kappa_1}),
\end{equation}
the number of terms appearing in the decomposition reduces significantly.
We prove the following bounds
\begin{subequations}
\begin{alignat}{9}
     \int_0^\infty |\l \cA\GG_{\neq},\cA\, \de_y^L \cP_\star^{r,s}(U_{\neq},U_{\neq})\r| &\lesssim (\nu^{-\frac23}\eps)\eps^2, \qquad r,s=1,3, \label{eq:Pstar-rs-neqneq}\\
     \int_0^\infty |\l \cA\GG_{\neq},\cA\,  \de_y^L\cP_\star^{r,2}(U_{\neq},U_{\neq})\r| &\lesssim (\nu^{-\frac34}\eps)\eps^2,\qquad r=1,3,\label{eq:Pstar-r2-neqneq}\\
     \int_0^\infty |\l \cA\GG_{\neq},\cA \, \de_y^L\cP_\star^{2,2}(U_{\neq},U_{\neq})\r| &\lesssim (\nu^{-\frac13}\eps)\eps^2,\label{eq:Pstar-22-neqneq}\\
     \int_0^\infty |\l \cA\GG_{\neq},\cA \, \de_y^L\cP_\star^{r,2}(U_{\neq},U_0)\r| &\lesssim (\nu^{-\frac34}\eps)\eps^2,\qquad r=1,3,\label{eq:Pstar-12-neq0}\\
     \int_0^\infty |\l \cA\GG_{\neq},\cA \,\de_y^L \cP_\star^{r,3}(U_{\neq},U_0)\r| &\lesssim (\nu^{-\frac23}\eps)\eps^2,\qquad r=1,3,\label{eq:Pstar-r3-neq0}\\
     \int_0^\infty |\l \cA\GG_{\neq},\cA \,\de_y^L \cP_\star^{2,2}(U_{\neq},U_0)\r| &\lesssim (\nu^{-\frac13}\eps)\eps^2,\label{eq:Pstar-22-neq0}\\
     \int_0^\infty |\l \cA\GG_{\neq},\cA \,\de_y^L \cP_\star^{2,3}(U_{\neq},U_0)\r| &\lesssim (\nu^{-\frac56}\eps)\eps^2.\label{eq:Pstar-23-neq0}
\end{alignat}
\end{subequations}
In the $(\neq,0)$ interactions, we note that
\begin{equation}
\l \cA\GG_{\neq},\cA \de_y^L \cP_\star^{r,1}(U_{\neq},U_0)\r
=\l \cA\GG_{\neq},|\nabla_{x,z}|^{-\frac12}|\nabla_L|^{-\frac12}\de_y^L(\de_r^LU^1_{\neq}\de_xU^r_{0})\r=0,
\end{equation}
for $r\in\{1,2,3\}$, so these terms do not appear in the list above. 

Starting with $\cP_\star^{1,1}(U_{\neq},U_{\neq})_{\neq}$ in \eqref{eq:Pstar-rs-neqneq}, we write 
\begin{equation}
    \l \cA\GG_{\neq},\cA \de_y^L \cP_\star^{1,1}(U_{\neq},U_{\neq})\r =\l \cA\GG_{\neq}, \cA |\nabla_{x,z}|^{-\frac12}|\nabla_L|^{-\frac12}\de_y^L(\de_xU^1_{\neq}\de_xU^1_{\neq})\r,
\end{equation}
which, following \eqref{eq:distributing derivatives} and \eqref{eq:I_1234}, can be bounded as
\begin{align}
    |\l \cA\GG_{\neq},\cA  \de_y^L\cP_\star^{1,1}(U_{\neq},U_{\neq})\r|&\lesssim \norm{|\nabla_{x,z}|^{-\frac12}|\nabla_L|^{-\frac12}\de_y^L \cA\GG_{\neq}} \norm{\cA U^1_{\neq}}\norm{\de_x^2U^1_{\neq}}_{L^\infty} \notag\\
    &\quad +\norm{\de_x|\nabla_{x,z}|^{-\frac12}|\nabla_L|^{-\frac12}\de_y^L \cA \GG_{\neq}} \norm{\cA U^1_{\neq}}\norm{\de_xU^1_{\neq}}_{L^\infty}.
\end{align}
Interpolation and Sobolev embedding leads to 
\begin{equation}
    |\l \cA\GG_{\neq},\cA \de_y^L \cP_\star^{1,1}(U_{\neq},U_{\neq})\r| \lesssim \norm{\nabla_L \cA  \GG_{\neq}}^{\frac12}\norm{\cA  \GG_{\neq}}^{\frac12} \norm{\cA  U^1_{\neq}}^2+ \norm{\nabla_L \cA \GG_{\neq}} \norm{\cA U^1_{\neq}}^2,
\end{equation}
and, by the bootstrap assumptions in Theorem \ref{thm:bootstrap_step}, we have \eqref{eq:Pstar-rs-neqneq} for $r=s=1$.
% \begin{equation}\label{eq:finalbound_P11star}
%     \int_0^t |\l \cA\GG_{\neq},\cA  \cP_\star^{1,1}(U_{\neq},U_{\neq})\r|\dd s \lesssim (\nu^{-\frac12}\eps)\eps^2 +(\nu^{-\frac23}\eps)\eps^2.
% \end{equation}
The same structure and hence analysis is shared by $\cP_\star^{1,3}(U_{\neq},U_{\neq})_{\neq}$ and $\cP_\star^{3,3}(U_{\neq},U_{\neq})_{\neq}$, for which the same bound holds.

Turning to \eqref{eq:Pstar-r2-neqneq}, we write $\cP_\star^{1,2}(U_{\neq},U_{\neq})_{\neq}$ using \eqref{eq:SimVar} as
\begin{equation}
    \l \cA\GG_{\neq},\cA \de_y^L\cP_\star^{1,2}(U_{\neq},U_{\neq})\r =\l \cA\GG_{\neq}, \cA\, |\nabla_{x,z}|^{-\frac12}|\nabla_L|^{-\frac12}\de_y^L(\de_x|\nabla_{x,z}|^{\frac12}|\nabla_L|^{-\frac32}\GG_{\neq}\de_y^LU^1_{\neq})\r.
\end{equation}
We can bound it following \eqref{eq:distributing derivatives} and distributing $|\de_x|^{\frac12}$ similarly to \eqref{eq:J1234} we have
\begin{align}
     |\l \cA\GG_{\neq},\cA\de_y^L \cP_\star^{1,2}(U_{\neq},U_{\neq})\r |&\lesssim \norm{|\de_x|^{\frac12}|\nabla_{x,z}|^{-\frac12}|\nabla_L|^{\frac12}\cA\GG_{\neq}}\norm{ |\de_x|^{\frac12}|\nabla_{x,z}|^{\frac12}|\nabla_L|^{-\frac32}\cA\GG_{\neq}}\norm{\de_y^LU^1_{\neq}}_{L^\infty}\notag\\
     &\quad + \norm{|\nabla_L|^{\frac12}\cA\GG_{\neq}}\norm{ |\de_x|^{\frac12}|\nabla_{x,z}|^{\frac12}|\nabla_L|^{-\frac32}\cA\GG_{\neq}}\norm{\de_y^L|\de_x|^{\frac12}U^1_{\neq}}_{L^\infty}\notag\\
    &\quad +\norm{|\nabla_L|^{\frac12}\cA\GG_{\neq}}\norm{ \de_x|\nabla_{x,z}|^{\frac12}|\nabla_L|^{-\frac32}\GG_{\neq}}_{L^\infty}\norm{\de_y^L\cA U^1_{\neq}}.
\end{align}
Using \eqref{eq:property_A}, \eqref{eq:other norm bounds}, \eqref{eq:time_extraction}, and the multiplier $\cM_3$ \eqref{def:M3}, we arrive at
\begin{align}
     |\l \cA\GG_{\neq},\cA\de_y^L \cP_\star^{1,2}(U_{\neq},U_{\neq})\r |&\lesssim \nu^{-\frac13}\norm{\nabla_L\cA\GG_{\neq}}^{\frac12}\norm{\cA\GG_{\neq}}^{\frac12}\norm{\sqrt{-\frac{\dot \cM}{\cM}}\cA \GG_{\neq}}\norm{\cA U^1_{\neq}}\notag\\
     &\quad +\l t \r^{-\frac32}\norm{\nabla_L\cA\GG_{\neq}}^{\frac12}\norm{\cA\GG_{\neq}}^{\frac32}\norm{\nabla_L\cA U^1_{\neq}},
\end{align}
concluding \eqref{eq:Pstar-r2-neqneq} with $r=1$ by the bootstrap hypotheses. Again, the term $\cP_\star^{3,2}(U_{\neq},U_{\neq})_{\neq}$ has the same structure provided we exchange $\de_z$ with $\de_x$ and $U^3$ with $U^1$, hence it has the same bound.

The last $(\neq, \neq)$ interaction is $\cP_\star^{2,2}(U_{\neq},U_{\neq})_{\neq}$ which, using again \eqref{eq:SimVar}, we rewrite as
\begin{align}
   \l \cA\GG_{\neq}, \cA\de_y^L |\nabla_{x,z}|^{-\frac12}|\nabla_L|^{-\frac12}\de_y^L(\de_y^L|\nabla_{x,z}|^{\frac12}|\nabla_L|^{-\frac32}\GG_{\neq}\de_y^L|\nabla_{x,z}|^{\frac12}|\nabla_L|^{-\frac32}\GG_{\neq})\r.
\end{align}
We distribute $|\nabla_L|^{\frac12}$ inside the product (as \eqref{eq:distributing derivatives}) and consequently $|\nabla_{x,z}|^{\frac12}$ when it appears on the high frequency (as done in \eqref{eq:I_1234}), to obtain 
\begin{align}
    |\l \cA\GG_{\neq},\cA\de_y^L \cP_\star^{2,2}(U_{\neq},U_{\neq})\r|&\lesssim \norm{\cA\GG_{\neq}} \norm{\cA\GG_{\neq}}\norm{|\nabla_{x,z}|^{\frac12}|\nabla_L|^{-\frac12}\GG_{\neq}}_{L^\infty}\notag\\
    &+\norm{|\nabla_{x,z}|^{-\frac12}\cA\GG_{\neq}} \norm{\cA\GG_{\neq}}\norm{\nabla_{x,z}|\nabla_L|^{-\frac12}\GG_{\neq}}_{L^\infty}\notag\\
    &+\norm{|\nabla_{x,z}|^{-\frac12}\cA\GG_{\neq}} \norm{|\nabla_{x,z}|^{\frac12}|\nabla_L|^{-\frac12}\cA\GG_{\neq}}\norm{|\nabla_{x,z}|^{\frac12}\GG_{\neq}}_{L^\infty}.
\end{align}
Using \eqref{eq:other norm bounds} and the Sobolev embedding, all terms can be bounded by $\norm{\cA\GG_{\neq}}^3$. Therefore, by bootstrap assumptions (see Theorem \ref{thm:bootstrap_step}), estimate \eqref{eq:Pstar-22-neqneq} follows.

Moving on to the $(\neq,0)$, we begin from the term containing $\cP_\star^{1,2}(U_{\neq},U_0)_{\neq}$. From \eqref{eq:SimVar}, the term
\begin{equation}
    \l \cA\GG_{\neq},\cA\de_y^L \cP_\star^{1,2}(U_{\neq},U_0)\r =\l \cA\GG_{\neq}, \cA\, |\nabla_{x,z}|^{-\frac12}|\nabla_L|^{-\frac12}\de_y^L(\de_x|\nabla_{x,z}|^{\frac12}|\nabla_L|^{-\frac32}\GG_{\neq}\de_yU^1_{0})\r
\end{equation}
can be bounded following \eqref{eq:distributing derivatives} as 
\begin{align}
    |\l \cA\GG_{\neq},\cA\de_y^L \cP_\star^{1,2}(U_{\neq},U_0)\r|&\lesssim \norm{|\nabla_{x,z}|^{-\frac12}|\nabla_L|^{\frac12}\cA\GG_{\neq}} \norm{\de_x|\nabla_{x,z}|^{\frac12}|\nabla_L|^{-\frac32}\cA\GG_{\neq}}\norm{\de_yU^1_{0}}_{L^\infty}\notag\\
    &\quad + \norm{|\nabla_{x,z}|^{-\frac12}|\nabla_L|^{\frac12}\cA\GG_{\neq}} \norm{\e^{\lambda\nu^{\frac13}t}\de_x|\nabla_{x,z}|^{\frac12}|\nabla_L|^{-\frac32}\GG_{\neq}}_{L^\infty}\norm{\de_yU^1_{0}}_{H^{2m}}.
\end{align}
Using \eqref{eq:other norm bounds}, \eqref{eq:time_extraction}, the multiplier $\cM_3$ \eqref{def:M3} and Sobolev embedding
\begin{align}
    |\l \cA\GG_{\neq},\cA\de_y^L \cP_\star^{1,2}(U_{\neq},U_0)\r|&\lesssim \norm{\cA\GG_{\neq}}^{\frac12}\norm{\nabla_L\cA\GG_{\neq}}^{\frac12} \norm{\sqrt{-\frac{\dot \cM}{\cM}}\cA\GG_{\neq}}\norm{U^1_{0}}_{H^{2m}}\notag\\
    &\quad +\norm{\cA\GG_{\neq}}^{\frac12}\norm{\nabla_L\cA\GG_{\neq}}^{\frac12} \l t\r^{-\frac32}\norm{\cA\GG_{\neq}}\norm{\nabla U^1_{0}}_{H^{2m}}.
\end{align}
Hence, by the bootstrap assumptions, we end up with \eqref{eq:Pstar-12-neq0}. The same strategy applies to the term involving $\cP_\star^{3,2}(U_{\neq},U_0)$.
% \begin{equation}
%     \int_0^t|\l \cA\GG_{\neq},\cA\, \cP_\star^{1,2}(U_{\neq},U_0)\r|\lesssim (\nu^{-\frac13}\eps)\eps^2+(\nu^{-\frac34}\eps)\eps^2.
% \end{equation}
For $\cP_\star^{1,3}(U_{\neq},U_0)$ we have
\begin{equation}
    \l \cA\GG_{\neq},\cA\de_y^L \cP_\star^{1,3}(U_{\neq},U_0)\r =\l \cA\GG_{\neq}, \cA\, |\nabla_{x,z}|^{-\frac12}|\nabla_L|^{-\frac12}\de_y^L(\de_xU^3_{\neq}\de_zU^1_{0})\r,
\end{equation}
which can be bounded similarly to \eqref{eq:I_1234}, after commuting $\de_z$ when it appears on the high frequencies of $U^1_0$, as
\begin{align}
    |\l \cA\GG_{\neq},\cA\de_y^L \cP_\star^{1,3}(U_{\neq},U_0)\r|&\lesssim \norm{\de_x|\nabla_{x,z}|^{-\frac12}|\nabla_L|^{\frac12}\cA\GG_{\neq}}\norm{\cA  U^3_{\neq}}\norm{\de_zU^1_{0}}_{L^\infty}\notag\\
    &\quad +\norm{|\nabla_{x,z}|^{-\frac12}|\nabla_L|^{\frac12}\cA\GG_{\neq}}\norm{\e^{\lambda\nu^{\frac13}t} \de_z\de_xU^3_{\neq}}_{L^\infty}\norm{U^1_{0}}_{H^{2m}}\notag\\
    &\quad +\norm{\de_z|\nabla_{x,z}|^{-\frac12}|\nabla_L|^{\frac12}\cA\GG_{\neq}}\norm{\e^{\lambda\nu^{\frac13}t} \de_xU^3_{\neq}}_{L^\infty}\norm{U^1_{0}}_{H^{2m}}.
\end{align}
Using \eqref{eq:other norm bounds} then implies
\begin{equation}
    |\l \cA\GG_{\neq},\cA\de_y^L \cP_\star^{1,3}(U_{\neq},U_0)\r|\lesssim \norm{\nabla_L\cA\GG_{\neq}}\norm{\cA U^3_{\neq}}\norm{U^1_{0}}_{H^{2m}}+\norm{\cA\GG_{\neq}}^{\frac12}\norm{\nabla_L\cA\GG_{\neq}}^{\frac12}\norm{\cA U^3_{\neq}}\norm{U^1_{0}}_{H^{2m}},
\end{equation}
and hence, via the bootstrap assumptions, we conclude \eqref{eq:Pstar-r3-neq0} for $r=1$. The case $r=3$ follows similarly.
% \begin{equation}
%     \int_0^t |\l \cA\GG_{\neq},\cA\, \cP_\star^{1,3}(U_{\neq},U_0)\r| \lesssim (\nu^{-\frac23}\eps)\eps^2+(\nu^{-\frac12}\eps)\eps^2.
% \end{equation}
The term $\cP_\star^{2,2}(U_{\neq},U_0)_{\neq}$ has the same structure as $\cP_\star^{2,2}(U_{\neq},U_{\neq})_{\neq}$: indeed
\begin{align}
    &\l \cA\GG_{\neq},\cA\de_y^L \cP_\star^{2,2}(U_{\neq},U_0)\r \notag\\
    &\quad =\l \cA\GG_{\neq}, \cA\, |\nabla_{x,z}|^{-\frac12}|\nabla_L|^{-\frac12}\de_y^L(\de_y^L|\nabla_{x,z}|^{\frac12}|\nabla_L|^{-\frac32}\GG_{\neq}\de_y|\de_z|^{\frac12}|\nabla|^{-\frac32}\GG_0)\r,
\end{align}
hence \eqref{eq:Pstar-22-neq0} follows in the same way as \eqref{eq:Pstar-22-neqneq} above.
Finally, using \eqref{eq:SimVar}, the term with $\cP_\star^{2,3}(U_{\neq},U_0)_{\neq}$ reads 
\begin{equation}
    \l \cA\GG_{\neq},\cA\de_y^L \cP_\star^{2,3}(U_{\neq},U_0)\r =\l \cA\GG_{\neq}, \cA\, |\nabla_{x,z}|^{-\frac12}|\nabla_L|^{-\frac12}\de_y^L(\de_y^LU^3_{\neq}\de_z|\de_z|^{\frac12}|\nabla|^{-\frac32}\GG_0)\r,
\end{equation}
and is bounded following \eqref{eq:distributing derivatives} as
\begin{align}
    |\l \cA\GG_{\neq},\cA\de_y^L \cP_\star^{2,3}(U_{\neq},U_0)\r|&\lesssim \norm{\cA|\nabla_{x,z}|^{-\frac12}|\nabla_L|^{\frac12}\GG_{\neq}}\norm{\cA \de_y^LU^3_{\neq}}\norm{|\de_z|^{\frac32}|\nabla|^{-\frac32}\GG_0}_{L^\infty}\notag\\
    &\quad + \norm{\cA|\nabla_{x,z}|^{-\frac12}|\nabla_L|^{\frac12}\GG_{\neq}}\norm{\e^{\lambda\nu^{\frac13}t} \de_y^LU^3_{\neq}}_{L^\infty}\norm{|\de_z|^{\frac32}|\nabla|^{-\frac32}\GG_0}_{H^{2m}}.
\end{align}
The Sobolev embedding and \eqref{eq:other norm bounds} bound this term by 
$\norm{\cA\GG_{\neq}}^{\frac12}\norm{\nabla_L\cA\GG_{\neq}}^{\frac12}\norm{\nabla_L\cA U^3_{\neq}}\norm{\GG_0}_{H^{2m}}$ and we can conclude \eqref{eq:Pstar-23-neq0} from the bootstrap assumptions.
Thus \eqref{eq:P_star_bds} is proven, and Lemma \ref{lem:GGammanonlin} follows. This was the last step missing in the proof of Proposition \ref{prop:GGamma}.

\subsection{Control of $U^1_{\neq},U^3_{\neq}$ -- proof of Proposition \ref{prop:U13}}\label{ssec:U13neq}

For $r=1,3$, the equation \eqref{eq:full_shorthand} satisfied by $U^r_{\neq}$ 
can be written as
\begin{equation}\label{eq:Ur}
\de_tU^r_{\neq}=\nu\Delta_LU^r_{\neq}-\frac{3-r}{2}U^2_{\neq} +\de_rP_{\neq}  + \cT(U,U^r)_{\neq} +\de_r\cP(U,U)_{\neq}.
\end{equation}
Here, $\de_1=\de_x$ and $\de_3=\de_z$ respectively. Using the multiplier $\cA$, we compute the time derivative of the $L^2$ norm of $U^r_{\neq}$,
which, after integrating in time, using \eqref{eq:dotA} to treat the term arising from the multiplier $\cA$ and Lemma  \ref{lem:enhanced_dissip_estim}, becomes
\begin{align}\label{eq:energy_est_Ur_pre}
    &\norm{\cA U^r_{\neq}(t)}^2+\nu \norm{\nabla_L\cA U^r_{\neq}}^2_{L^2_tL^2} +\norm{\sqrt{-\frac{\dot \cM }{\cM }}\cA U^r_{\neq}}^2_{L^2_tL^2} \notag\\
    &\quad \leq \norm{\cA U^r_{\neq}(0)}^2 +\frac{3-r}{2}\int_0^\infty\langle \cA U^1_{\neq},\cA U^2_{\neq}\rangle  +2\int_0^\infty\langle \cA U^r_{\neq},\cA \de_r\de_x|\nabla_L|^{-2}U^2_{\neq}\rangle\notag\\
  &\qquad  +\beta \int_0^\infty\langle \cA U^r_{\neq},\cA\de_r\de_y^L|\nabla_L|^{-2}\Theta_{\neq}\rangle- \int_0^\infty\langle \cA U^r_{\neq}, \cA\cT(U,U^r)_{\neq}\rangle-\int_0^\infty\langle \cA U^r_{\neq},  \cA\de_r\cP(U,U)_{\neq}\rangle.\qquad
\end{align}
To prove Proposition \ref{prop:U13} we are left to bound the terms appearing on the right-hand side. The terms arising from linear interactions can be treated as follows. 
Starting from the lift-up term, that only appears when $r=1$, we have using \eqref{eq:SimVar}
\begin{equation}
    \langle \cA U^1_{\neq},\cA U^2_{\neq}\rangle =\langle \cA U^1_{\neq},\cA|\nabla_{x,z}|^{\frac12}|\nabla_L|^{-\frac32}\GG_{\neq}\rangle.
\end{equation}
By exploiting the time decay of the $|\nabla_L|^{-\frac32}$ and using the multiplier $\cM_3$ (see \eqref{def:M3}) we can bound it by
\begin{equation}
    |\langle \cA U^1_{\neq},\cA U^2_{\neq}\rangle|\leq \norm{\sqrt{-\frac{\dot \cM_3}{\cM_3}}\cA U^1_{\neq}}\norm{\sqrt{-\frac{\dot \cM_3}{\cM_3}}\cA\GG_{\neq}},
\end{equation}
hence via bootstrap hypothesis in Theorem \ref{thm:bootstrap_step}, and recalling that $C_0\geq 10^4$, we have
\begin{equation}
    \int_0^\infty |\langle \cA U^1_{\neq},\cA U^2_{\neq}\rangle|\leq C_0^2\eps^2.
\end{equation}
The next term
\begin{equation}
    \langle \cA U^r_{\neq},A\de_r\de_x|\nabla_L|^{-2}U^2_{\neq}\rangle =\langle \cA \de_r\de_x|\nabla_L|^{-2}U^r_{\neq},\cA|\nabla_{x,z}|^{\frac12}|\nabla_L|^{-\frac32}\GG_{\neq}\rangle
\end{equation}
is bounded analogously to the previous one, since we can exploit the time decay coming from $|\nabla_L|^{-2}$ and $|\nabla_L|^{-\frac32}$. Hence, we deduce
\begin{equation}
    \int_0^\infty|\langle \cA U^r_{\neq},\cA\de_r\de_x|\nabla_L|^{-2}U^2_{\neq}\rangle|\leq C_0^2\eps^2.
\end{equation}
Finally, the same applies to the term 
\begin{equation}
    \langle \cA U^r_{\neq},\cA\de_r\de_y^L|\nabla_L|^{-2}\Theta_{\neq}\rangle =\langle \cA U^r_{\neq},\cA\de_r\de_y^L|\nabla_{x,z}|^{-\frac12}|\nabla_L|^{-\frac52}\Gamma_{\neq}\rangle,
\end{equation}
which can be still bounded by means of the multiplier $\cM_3$ defined in \eqref{def:M3}. We deduce that
\begin{equation}
    \int_0^\infty |\langle \cA U^r_{\neq},\cA\de_r\de_y^L|\nabla_L|^{-2}\Theta_{\neq}\rangle|\leq C_0^2\eps^2.
\end{equation}
In particular, it follows from \eqref{eq:energy_est_Ur_pre} that
\begin{align}\label{eq:energy_est_Ur}
    &\norm{\cA U^r_{\neq}(t)}^2+\nu \norm{\nabla_L\cA U^r_{\neq}}^2_{L^2_tL^2} +\norm{\sqrt{-\frac{\dot \cM }{\cM }}\cA U^r_{\neq}}^2_{L^2_tL^2} \notag\\
    &\qquad \leq 3C_0^2\eps^2+ \int_0^\infty|\langle \cA U^r_{\neq}, \cA\cT(U,U^r)_{\neq}\rangle|+\int_0^\infty|\langle \cA U^r_{\neq},  \cA\de_r\cP(U,U)_{\neq}\rangle|.
\end{align}
For the nonlinear terms, we collect our main estimates in the following lemma, whose proof is postponed to the next sections. 
\begin{lemma}\label{lem:Urnonlin}
 Under the assumptions of Theorem \ref{thm:bootstrap_step}, for $r=1,3$ there holds that
 \begin{align}
  \int_0^\infty |\langle \cA U^r_{\neq}, \cA\cT(U,U^r)_{\neq}\rangle|  
    &\lesssim (\nu^{-\frac23}\eps)\eps^2, \label{eq:T_r_bds}\\
  \int_0^\infty|\langle \cA U^r_{\neq},  \cA\de_r\cP(U,U)_{\neq}\rangle|   &\lesssim (\nu^{-\frac12}\eps)\eps^2. \label{eq:P_r_bds}
 \end{align} 
\end{lemma}

\begin{proof}
The proofs of \eqref{eq:T_r_bds} resp.\ \eqref{eq:P_r_bds} are given in Sections \ref{ssec:T_r_bds} resp.\ \ref{ssec:P_r_bds}.
\end{proof}

With Lemma \ref{lem:Urnonlin} at hand, the bound in Proposition \ref{prop:U13} follows from \eqref{eq:energy_est_Ur}.

\subsubsection{Nonlinear terms analysis: $\cT(U,U^r)_{\neq}$}\label{ssec:T_r_bds}
We start the analysis of the nonlinear terms appearing in \eqref{eq:Ur} from $\cT(U,U^r)_{\neq}$. As done for the nonlinear terms in Section \ref{ssec:GGamma_neq}, we split $\cT(U,U^r)_{\neq}$, for $r=1,3$, as
\begin{align}
   \cT(U,U^r)_{\neq}=\sum_{j,\kappa_1,\kappa_2} \cT^j(U_{\kappa_1},U^r_{\kappa_2})_{\neq}, \qquad \cT^j(U_{\kappa_1},U^r_{\kappa_2})_{\neq}:=(U^j_{\kappa_1}\de_j^LU^r_{\kappa_2})_{\neq}.
\end{align}
where $j\in \{1,2,3\}$ and $\kappa_1,\kappa_2 \in\{0,\neq\}$. 
We prove that
\begin{subequations}
    \begin{alignat}{1}
        \int_0^\infty|\l \cA U^r_{\neq},\cA \,\cT^j(U_{\neq},U^r_{\neq})\r|&\lesssim (\nu^{-\frac23}\eps)\eps^2, \qquad j=1,3 \label{eq:T-j-neqneq}\\
        \int_0^\infty|\l \cA U^r_{\neq},\cA \,\cT^2(U_{\neq},U^r_{\neq})\r|&\lesssim (\nu^{-\frac12}\eps)\eps^2,\label{eq:T-2-neqneq}\\
        \int_0^\infty|\l \cA U^r_{\neq},\cA \,\cT^j(U_0,U^r_{\neq})\r|&\lesssim(\nu^{-\frac23}\eps)\eps^2, \qquad j=1,2,3, \label{eq:T-j-0neq}\\
        \int_0^\infty|\l \cA U^r_{\neq},\cA \,\cT^2(U_{\neq},U^r_0)\r|&\lesssim (\nu^{-\frac12}\eps)\eps^2,\label{eq:T-2-neq0}\\
        \int_0^\infty|\l \cA U^r_{\neq},\cA \,\cT^3(U_{\neq},U^r_0)\r|&\lesssim (\nu^{-\frac23}\eps)\eps^2.\label{eq:T-3-neq0}
    \end{alignat}
\end{subequations}
Note that
\begin{equation}
    \l \cA U^r_{\neq},\cA \,\cT^1(U_{\neq},U^r_0)\r= \l \cA U^r_{\neq}, \cA (U^1_{\neq}\de_x U^r_{0}) \r = 0,
\end{equation}
so this term does not appear in the list above.
Starting with $\cT^1(U_{\neq},U^r_{\neq})_{\neq}$ appearing in 
\begin{equation}
    \l \cA U^r_{\neq},\cA \,\cT^1(U_{\neq},U^r_{\neq})\r= \l \cA U^r_{\neq}, \cA (U^1_{\neq}\de_xU^r_{\neq}) \r,
\end{equation}
the structure is simpler compared to the previous analysis done for $\cT_\star^1(U_{\neq},U^2_{\neq})_{\neq}$, so by using \eqref{eq:prod_ineq_A} we can deduce directly that 
\begin{equation}
   |\l \cA U^r_{\neq},\cA  \,\cT^1(U_{\neq},U^r_{\neq})\r|\lesssim \norm{\cA U^r_{\neq}}^2\norm{\cA U^1_{\neq}} + \norm{\cA U^r_{\neq}}\norm{\cA U^1_{\neq}}\norm{\nabla_L\cA U^r_{\neq}}.
\end{equation}
Here we used the 
inequality $\|\de_x\cA U^r_{\neq}\|\leq\|\nabla_L\cA U^r_{\neq}\|$ in the second term.
Hence, by the bootstrap assumptions we can conclude that
\begin{equation}
    \int_0^\infty|\l \cA U^r_{\neq},\cA \,\cT^1(U_{\neq},U^r_{\neq})\r|\lesssim (\nu^{-\frac13}\eps)\eps^2+(\nu^{-\frac23}\eps)\eps^2.
\end{equation}
The same structure appears in 
\begin{equation}
    \l \cA U^r_{\neq},\cA \,\cT^3(U_{\neq},U^r_{\neq})\r= \l \cA U^r_{\neq}, \cA (U^3_{\neq}\de_zU^r_{\neq}) \r,
\end{equation}
therefore \eqref{eq:T-j-neqneq} follows.

For $\cT^2(U_{\neq},U^r_{\neq})_{\neq}$ we can exploit the time decay to obtain a better estimate compared to the previous terms. Starting from 
\begin{equation}
    \l \cA U^r_{\neq},\cA \,\cT^2(U_{\neq},U^r_{\neq})\r= \l \cA U^r_{\neq}, \cA (|\nabla_{x,z}|^{\frac12}|\nabla_L|^{-\frac32}\GG_{\neq}\de_y^LU^r_{\neq}) \r,
\end{equation}
and using \eqref{eq:prod_ineq_A} we have 
\begin{align}
    |\l \cA U^r_{\neq},\cA \,\cT^2(U_{\neq},U^r_{\neq})\r|&\lesssim \norm{\cA U^r_{\neq}}\norm{|\nabla_{x,z}|^{\frac12}|\nabla_L|^{-\frac32}\cA\GG_{\neq}}\norm{\de_y^LU^r_{\neq}}_{L^\infty} \notag\\
    &\qquad + \norm{\cA U^r_{\neq}}\norm{|\nabla_{x,z}|^{\frac12}|\nabla_L|^{-\frac32}\GG_{\neq}}_{L^\infty}\norm{\de_y^L\cA U^r_{\neq}}\notag\\
    &\lesssim \nu^{-\frac13}\norm{\cA U^r_{\neq}}^2\norm{\sqrt{-\frac{\dot \cM}{\cM}}\cA\GG_{\neq}}+ \l t\r^{-\frac32}\norm{\cA U^r_{\neq}}\norm{\cA\GG_{\neq}}\norm{\nabla_L\cA U^r_{\neq}},
\end{align}
where in the last line we used the multiplier $\cM_3$ defined in \eqref{def:M3}, \eqref{eq:property_A}, and \eqref{eq:time_extraction}.
We conclude, via bootstrap assumptions in Theorem \ref{thm:bootstrap_step}, that \eqref{eq:T-2-neqneq} holds.
% \begin{equation}
%     \int_0^t|\l \cA ^r_{\neq},\cA \,\cT^2(U_{\neq},U^r_{\neq})\r|\lesssim (\nu^{-\frac12}\eps)\eps^2.
% \end{equation}
Regarding $\cT^1(U_0,U^r_{\neq})_{\neq}$, the nonlinear term 
\begin{equation}
    \l \cA U^r_{\neq},\cA \,\cT^1(U_0,U^r_{\neq})\r= \l \cA U^r_{\neq}, \cA (U^1_{0}\de_x U^r_{\neq}) \r
\end{equation}
can be bounded using \eqref{eq:prod_ineq_A} as
\begin{align}
    |\l \cA U^r_{\neq},\cA \,\cT^1(U_0,U^r_{\neq})\r|&\lesssim \norm{\cA  U^r_{\neq}}\norm{ U^1_{0}}_{H^{2m}}\norm{\e^{\lambda\nu^{\frac13}t}\de_x U^r_{\neq}}_{L^\infty} + \norm{\cA U^r_{\neq}} \norm{ U^1_{0}}_{L^\infty}\norm{\de_x \cA U^r_{\neq}}\notag\\
    &\lesssim \norm{\cA U^r_{\neq}}^2\norm{ U^1_{0}}_{H^{2m}} + \norm{\cA U^r_{\neq}} \norm{ U^1_{0}}_{H^{2m}}\norm{\nabla_L \cA U^r_{\neq}}.
\end{align}
This is analogous to $\cT^1(U_{\neq},U^r_{\neq})_{\neq}$, and the same argument applies to
% hence 
% \begin{equation}
%     \int_0^t|\l \cA U^r_{\neq},\cA \,\cT^1(U_0,U^r_{\neq})\r|\lesssim (\nu^{-\frac13}\eps)\eps^2+(\nu^{-\frac23}\eps)\eps^2.
% \end{equation}
$\cT^2(U_0,U^r_{\neq})_{\neq}$ and $\cT^3(U_0,U^r_{\neq})_{\neq}$, so \eqref{eq:T-j-0neq} holds.

Next, $\cT^2(U_{\neq},U^r_0)_{\neq}$ can be bounded similarly to $\cT^2(U_{\neq},U^r_{\neq})_{\neq}$, using \eqref{eq:prod_ineq_A}, the multiplier $\cM_3$ \eqref{def:M3}, and time decay \eqref{eq:time_extraction}. Hence
\begin{equation}
    \l \cA U^r_{\neq},\cA \,\cT^2(U_{\neq},U^r_0)\r= \l \cA U^r_{\neq}, \cA (|\nabla_{x,z}|^{\frac12}|\nabla_L|^{-\frac32}\GG_{\neq}\de_y^L U^r_{0}) \r.
\end{equation}
gives \eqref{eq:T-2-neq0}, using the bootstrap assumptions in Theorem \ref{thm:bootstrap_step}. 
% \begin{equation}
%     \int_0^t|\l \cA U^r_{\neq},\cA \,\cT^2(U_{\neq},U^r_0)\r|\lesssim (\nu^{-\frac12}\eps)\eps^2.
% \end{equation}
Finally, $\cT^3(U_{\neq},U^r_0)_{\neq}$ concludes this part as we note that this term is analogous to $\cT^1(U_{\neq}, U^r_{\neq})_{\neq}$, hence getting \eqref{eq:T-3-neq0}.
Thus, \eqref{eq:T_r_bds} is proved.
% \begin{equation}
%     \int_0^t|\l \cA U^r_{\neq},\cA \,\cT^3(U_{\neq},U^r_0)\r|\lesssim (\nu^{-\frac23}\eps)\eps^2.
% \end{equation}

\subsubsection{Nonlinear terms analysis: $\de_r\cP(U,U)_{\neq}$}\label{ssec:P_r_bds}
The second and last nonlinear term appearing in \eqref{eq:Ur} is $\de_r^L\cP(U,U)_{\neq}$, where $\de_1=\de_x$ and $\de_3=\de_z$. Analogously to $\cT(U,U^r)_{\neq}$, we divide $\de_r\cP(U,U)_{\neq}$ into different components following the structure of $\cP(U,U)$, as done for  $\de_y^L\cP_\star(U,U)_{\neq}$ in Section \ref{sssec:cPGGamma}. Namely 
\begin{equation}
    \cP(U,U)_{\neq}= \sum_{i,j,\kappa_1,\kappa_2}\cP^{i,j}(U_{\kappa_1},U_{\kappa_2})_{\neq},\quad \cP^{i,j}(U_{\kappa_1},U_{\kappa_2})_{\neq}:=|\nabla_L|^{-2}
    (\de_i^LU^j_{\kappa_1}\de_j^LU^i_{\kappa_2})_{\neq},
\end{equation}
for $i,j\in \{1,2,3\}$ and $\kappa_1,\kappa_2 \in\{0,\neq\}$ with the convention $\de_1^L=\de_x$, $\de_2^L=\de_y^L$, and $\de_3^L=\de_z$. 
Using again that the symmetry $\cP^{i,j}(U_{\kappa_1},U_{\kappa_2})=\cP^{j,i}(U_{\kappa_2},U_{\kappa_1})$, we reduce ourselves to proving
\begin{subequations}
\begin{align}
\int_0^\infty|\l \cA U^r_{\neq},\cA \,\de_r\cP^{i,j}(U_{\neq},U_{\neq})\r|&\lesssim (\nu^{-\frac16}\eps)\eps^2, \qquad i,j=1,3, \label{eq:P-rs-neqneq} \\
\int_0^\infty |\l \cA U^r_{\neq},\cA \, \de_r\cP^{2,j}(U_{\neq},U_{\neq})\r|&\lesssim(\nu^{-\frac13}\eps)\eps^2,\qquad j=1,3, \label{eq:P-2j-neqneq}\\
\int_0^\infty |\langle \cA U^r_{\neq},\cA\,\de_r\cP^{2,2}(U_{\neq},U_{\neq})\rangle|
&\leq (\nu^{-\frac16}\eps)\eps^2,\label{eq:P-22-neqneq}\\
\int_0^\infty|\langle \cA U^r_{\neq},\cA \,\de_r\cP^{i,2}(U_{\neq},U_0)\rangle| &\lesssim (\nu^{-\frac16}\eps)\eps^2,\qquad i=1,3,\label{eq:P-12-neq0}\\
\int_0^\infty|\langle \cA U^r_{\neq},\cA \,\de_r\cP^{i,3}(U_{\neq},U_0)\rangle| &\lesssim (\nu^{-\frac13}\eps)\eps^2, \qquad i=1,3, \label{eq:P-i3-neq0}\\
\int_0^\infty \langle \cA U^r_{\neq},\cA \,\de_r\cP^{2,2}(U_{\neq},U_0)\rangle| &\lesssim (\nu^{-\frac13}\eps)\eps^2,\label{eq:P-22-neq0}\\
\int_0^\infty|\langle \cA U^r_{\neq}, \cA \,\de_r\cP^{2,3}(U_{\neq},U_0)\rangle| &\lesssim (\nu^{-\frac12}\eps)\eps^2.\label{eq:P-23-neq0}
\end{align}
\end{subequations}
Note that for $j\in\{1,2,3\}$,  $\cP^{j,1}(U_{\neq},U_0)_{\neq}=0$, so these terms do not appear in the list above.
Starting from $\cP^{1,1}(U_{\neq},U_{\neq})_{\neq}$ in \eqref{eq:P-rs-neqneq}, we have
\begin{equation}
    \l \cA U^r_{\neq},\cA \,\de_r\cP^{1,1}(U_{\neq},U_{\neq})\r= \l \cA  U^r_{\neq}, \cA \de_r|\nabla_L|^{-2}(\de_xU^1_{\neq}\de_xU^1_{\neq}) \r,
\end{equation}
and, after using \eqref{eq:prod_ineq_A}, can be bounded as 
\begin{align}
    |\l \cA U^r_{\neq},\cA \,\de_r\cP^{1,1}(U_{\neq},U_{\neq})\r|&\lesssim \norm{ \de_r|\nabla_L|^{-2} \cA U^r_{\neq}} \norm{\cA U^1_{\neq}}\norm{\de_x^2U^1_{\neq}}_{L^\infty}\notag\\
    &\quad + \norm{\de_x \de_r|\nabla_L|^{-2} \cA U^r_{\neq}} \norm{\de_x U^1_{\neq}}_{L^\infty}\norm{\cA U^1_{\neq}},
\end{align}
where, following \eqref{eq:I_1234}, $\de_x$ has been distributed to the other two terms when necessary. Using the multiplier $\cM_3$ \eqref{def:M3} and the Sobolev embedding we obtain 
\begin{equation}
    |\l \cA U^r_{\neq},\cA \,\de_r\cP^{1,1}(U_{\neq},U_{\neq})\r|\lesssim \norm{\sqrt{-\frac{\dot \cM}{\cM}}\cA U^r_{\neq}}\norm{\cA U^1_{\neq}}^2,
\end{equation}
which implies \eqref{eq:P-rs-neqneq} via the bootstrap assumptions.
% \begin{equation}
%     \int_0^t|\l \cA U^r_{\neq},\cA \,\de_r\cP^{1,1}(U_{\neq},U_{\neq})\r|\lesssim (\nu^{-\frac16}\eps)\eps^2.
% \end{equation}
The same bound is obtained for the nonlinear terms $\cP^{1,3}(U_{\neq},U_{\neq})_{\neq}$ and $\cP^{3,3}(U_{\neq},U_{\neq})_{\neq}$, since upon exchanging $\de_x$ with $\de_z$ it is possible to recover the same structure.

For $\cP^{2,1}(U_{\neq},U_{\neq})_{\neq}$ we have, via the symmetric change of variables \eqref{eq:SimVar},
\begin{equation}
    \l \cA U^r_{\neq},\cA \,\de_r\cP^{2,1}(U_{\neq},U_{\neq})\r= \l \cA U^r_{\neq}, \cA \de_r|\nabla_L|^{-2}(\de_x|\nabla_{x,z}|^{\frac12}|\nabla_L|^{-\frac32}\GG_{\neq}\de_y^LU^1_{\neq}) \r,
\end{equation}
which can be bounded following \eqref{eq:distributing derivatives} and \eqref{eq:I_1234} as
\begin{equation}
  |\l \cA U^r_{\neq},\cA \,\de_r\cP^{2,1}(U_{\neq},U_{\neq})\r |\lesssim \cI_1+\cI_2,
\end{equation}
where 
\begin{equation}
    \cI_1= \norm{ \de_r|\nabla_L|^{-2}\cA U^r_{\neq}} \norm{\de_x|\nabla_{x,z}|^{\frac12}|\nabla_L|^{-\frac32}\cA\GG_{\neq}}\norm{\de_y^LU^1_{\neq}}_{L^\infty}
\end{equation}
and
\begin{align}
    \cI_2&= \norm{\de_y^L\de_r|\nabla_L|^{-2} \cA U^r_{\neq}} \norm{\de_x|\nabla_{x,z}|^{\frac12}|\nabla_L|^{-\frac32}\GG_{\neq}}_{L^\infty}\norm{\cA U^1_{\neq}}\notag\\
    &\quad + \norm{ \de_r|\nabla_L|^{-2}\cA U^r_{\neq}} \norm{\de_y^L\de_x|\nabla_{x,z}|^{\frac12}|\nabla_L|^{-\frac32}\GG_{\neq}}_{L^\infty}\norm{\cA U^1_{\neq}}.
\end{align}
To bound $\cI_1$ we use the multiplier $\cM_3$ \eqref{def:M3} and the property \eqref{eq:property_A} of $\cA$ and we get 
\begin{equation}
   \cI_1\lesssim \norm{\sqrt{-\frac{\dot \cM}{\cM}}\cA U^r_{\neq}} \norm{\sqrt{-\frac{\dot \cM}{\cM}}\cA \GG_{\neq}}\nu^{-\frac13}\norm{\cA U^1_{\neq}}, 
\end{equation}
from which we deduce via the bootstrap assumptions that
\begin{equation}
    \int_0^t \cI_1\lesssim (\nu^{-\frac13}\eps)\eps^2.
\end{equation}
For $\cI_2$, \eqref{eq:other norm bounds}, the extra time decay obtained via \eqref{eq:time_extraction}, the multiplier $\cM_3$ and Sobolev embedding lead to 
\begin{equation}
    \cI_2 \lesssim \l t \r^{-\frac32} \norm{\cA U^r_{\neq}} \norm{\cA\GG_{\neq}}\norm{\cA U^1_{\neq}}+\l t \r^{-\frac12} \norm{\sqrt{-\frac{\dot \cM}{\cM}}\cA U^r_{\neq}}\norm{\cA \GG_{\neq}}\norm{\cA U^1_{\neq}}.
\end{equation}
We conclude that
\begin{equation}
    \int_0^\infty \cI_2\,\lesssim \eps^3+(\nu^{-1/12}\eps)\eps^2.
\end{equation}
In a similar way the term involving $\cP^{2,3}(U_{\neq},U_{\neq})_{\neq}$ can be bounded by 
\begin{equation}
    \int_0^\infty |\l \cA U^r_{\neq},\cA \,\de_r\cP^{2,3}(U_{\neq},U_{\neq})\r|\lesssim(\nu^{-\frac13}\eps)\eps^2,
\end{equation}
and \eqref{eq:P-2j-neqneq} is proved.
For $\cP^{2,2}(U_{\neq},U_{\neq})_{\neq}$ we have, using \eqref{eq:SimVar}, 
\begin{equation}
\langle \cA U^r_{\neq},\cA\,\de_r\cP^{2,2}(U_{\neq},U_{\neq})\rangle=\langle \cA U^r_{\neq},\cA\de_r|\nabla_L|^{-2}(\de_y^L|\nabla_{x,z}|^{\frac12}|\nabla_L|^{-\frac32}\GG_{\neq}\de_y^L|\nabla_{x,z}|^{\frac12}|\nabla_L|^{-\frac32}\GG_{\neq})\rangle.
\end{equation}
Then \eqref{eq:prod_ineq_A} leads to
\begin{align}
|\langle \cA U^r_{\neq},\cA\,\de_r\cP^{2,2}(U_{\neq},U_{\neq})\rangle|&\leq \norm{\de_r|\nabla_L|^{-2}\cA U^r_{\neq}}\norm{\de_y^L|\nabla_{x,z}|^{\frac12}|\nabla_L|^{-\frac32}\cA\GG_{\neq}}\norm{\de_y^L|\nabla_{x,z}|^{\frac12}|\nabla_L|^{-\frac32}\GG_{\neq}}\notag\\
&\leq \norm{\sqrt{-\frac{\dot \cM}{\cM}}\cA U^r_{\neq}}\norm{\cA\GG_{\neq}}^2\norm{\cA\GG_{\neq}},
\end{align}
where we used in the last line \eqref{eq:other norm bounds} and the multiplier $\cM_3$ \eqref{def:M3}. The bootstrap assumptions in Theorem \ref{thm:bootstrap_step} give \eqref{eq:P-22-neqneq}.
% \begin{align}
% \int_0^t |\langle \cA U^r_{\neq},\cA\,\de_r\cP^{2,2}(U_{\neq},U_{\neq})\rangle|
% &\leq (\nu^{-\frac16}\eps)\eps^2.
% \end{align}
We proceed now with the $(\neq,0)$ interactions.
% For $j\in\{1,2,3\}$,  $\cP^{j,1}(U_{\neq},U_0)_{\neq}$ is zero since $\de_xU^j_0=0$. Hence
% \begin{equation}
%     \l \cA U^r_{\neq},\cA \,\de_r\cP^{j,1}(U_{\neq},U_0)\r=  0
% \end{equation}
Moving on to $\cP^{1,2}(U_{\neq},U_0)_{\neq}$, we have 
\begin{align}
\langle \cA U^r_{\neq},\cA \de_r\cP^{1,2}(U_{\neq},U_0)\rangle=\langle \cA U^r_{\neq},\cA \de_r|\nabla_L|^{-2}(\de_x|\nabla_{x,z}|^{\frac12}|\nabla_L|^{-\frac32}\GG_{\neq}\de_yU_0^1)\rangle.
\end{align}
Then, mimicking \eqref{eq:distributing derivatives} and \eqref{eq:I_1234}, we bound 
\begin{align}
|\langle \cA U^r_{\neq},\cA \de_r\cP^{1,2}(U_{\neq},U_0)\rangle|&\lesssim \norm{\de_r|\nabla_L|^{-2}\cA U^r_{\neq}}\norm{\de_x|\nabla_{x,z}|^{\frac12}|\nabla_L|^{-\frac32}\cA\GG_{\neq}}\norm{\de_yU_0^1}_{L^\infty} \notag\\
&\quad +\norm{ \de_y\de_r|\nabla_L|^{-2}\cA U^r_{\neq}}\norm{\e^{\lambda\nu^{\frac13}t}\de_x|\nabla_{x,z}|^{\frac12}|\nabla_L|^{-\frac32}\GG_{\neq}}_{L^\infty}\norm{U_0^1}_{H^{2m}}\notag\\
&\quad +\norm{\de_r|\nabla_L|^{-2}\cA U^r_{\neq}}\norm{\e^{\lambda\nu^{\frac13}t}\de_y\de_x|\nabla_{x,z}|^{\frac12}|\nabla_L|^{-\frac32}\GG_{\neq}}\norm{U_0^1}_{H^{2m}}.
\end{align}
Using again \eqref{eq:other norm bounds} and the multiplier $\cM_3$ we have
\begin{align}
|\langle \cA U^r_{\neq},\cA \,\de_r\cP^{1,2}(U_{\neq},U_0)\rangle|
&\lesssim \norm{\sqrt{-\frac{\dot \cM}{\cM}}\cA U^r_{\neq}}\norm{\sqrt{-\frac{\dot \cM}{\cM}}\cA \GG_{\neq}}\norm{U_0^1}_{H^{2m}}\notag \\
&\quad +\norm{\cA U^r_{\neq}}\norm{\sqrt{-\frac{\dot \cM}{\cM}}\cA \GG_{\neq}}\norm{U_0^1}_{H^{2m}},
\end{align}
which is bounded via the bootstrap assumptions in Theorem \ref{thm:bootstrap_step} as in \eqref{eq:P-12-neq0}. The term involving $\cP^{3,2}(U_{\neq},U_0)$ is similar.
% \begin{align}
% \int_0^t|\langle \cA U^r_{\neq},\cA \,\de_r\cP^{1,2}(U_{\neq},U_0)\rangle| &\lesssim \eps^3+(\nu^{-\frac16}\eps)\eps^2.
% \end{align}
For $\cP^{1,3}(U_{\neq},U_0)_{\neq}$ we bound
\begin{align}
\langle \cA U^r_{\neq},\cA \de_r\cP^{1,3}(U_{\neq},U_0)\rangle=\langle AU^r_{\neq},A \de_r|\nabla_L|^{-2}(\de_xU_{\neq}^3\de_zU_0^1)\rangle
% &=\langle A\de_r|\nabla_L|^{-2}U^r_{\neq},A\de_xU_{\neq}^3\de_zU_0^1\rangle+\langle A\de_r|\nabla_L|^{-2}U^r_{\neq},\e^{\lambda\nu^{\frac13}t} \de_xU_{\neq}^3B\de_zU_0^1\rangle\\
% &=\langle A\de_r\de_x|\nabla_L|^{-2}U^r_{\neq},AU_{\neq}^3\de_zU_0^1\rangle+\langle A\de_r|\nabla_L|^{-2}U^r_{\neq},\e^{\lambda\nu^{\frac13}t} \de_x\de_zU_{\neq}^3BU_0^1\rangle\\
% &\quad +\langle A\de_z\de_r|\nabla_L|^{-2}U^r_{\neq},\e^{\lambda\nu^{\frac13}t} \de_xU_{\neq}^3BU_0^1\rangle
\end{align}
as 
\begin{align}
|\langle \cA U^r_{\neq},\cA \de_r\,\cP^{1,3}(U_{\neq},U_0)\rangle|&\lesssim \norm{\de_x\de_r|\nabla_L|^{-2}\cA U^r_{\neq}}\norm{\cA U_{\neq}^3}\norm{\de_zU_0^1}_{L^\infty}\notag \\
&\quad +\norm{ \de_z\de_r|\nabla_L|^{-2}\cA U^r_{\neq}}\norm{\e^{\lambda\nu^{\frac13}t}\de_xU_{\neq}^3}_{L^\infty}\norm{U_0^1}_{H^{2m}}\notag\\
&\quad +\norm{\de_r|\nabla_L|^{-2}\cA U^r_{\neq}}\norm{\e^{\lambda\nu^{\frac13}t}\de_z\de_xU_{\neq}^3}_{L^\infty}\norm{U_0^1}_{H^{2m}},
\end{align}
where we treated the terms as in \eqref{eq:distributing derivatives} and \eqref{eq:I_1234}. Then, by definition \eqref{def:M3} of $\cM_3$ and \eqref{eq:other norm bounds}
\begin{align}
|\langle \cA U^r_{\neq},\cA \de_r\,\cP^{1,3}(U_{\neq},U_0)\rangle|&\lesssim\norm{\sqrt{-\frac{\dot \cM}{\cM}}\cA U^r_{\neq}}\norm{\cA U_{\neq}^3}\norm{U_0^1}_{H^{2m}} +\norm{\cA U^r_{\neq}}\norm{\cA U_{\neq}^3}\norm{U_0^1}_{H^{2m}}.
\end{align}
Finally, the bootstrap assumptions lead to
\begin{align}
\int_0^\infty|\langle \cA U^r_{\neq},\cA \de_r\,\cP^{1,3}(U_{\neq},U_0)\rangle| &\lesssim (\nu^{-\frac16}\eps)\eps^2 + (\nu^{-\frac13}\eps)\eps^2.
\end{align}
A similar argument works for $\cP^{3,3}(U_{\neq},U_0)$, hence proving \eqref{eq:P-i3-neq0}.
Moving to 
$\cP^{2,2}(U_{\neq},U_0)_{\neq}$ we have
\begin{align}
\langle \cA U^r_{\neq},\cA\, \de_r\cP^{2,2}(U_{\neq},U_0)\rangle=\langle \cA U^r_{\neq},A \de_r|\nabla_L|^{-2}(\de_y^L|\nabla_{x,z}|^{\frac12}|\nabla_L|^{-\frac32}\GG_{\neq}\de_y^2|\de_z|^{\frac12}|\nabla|^{-\frac32}\GG_0)\rangle,
% \\
% &=\langle A\de_r|\nabla_L|^{-2}U^r_{\neq},A\de_y^L|\nabla_{x,z}|^{\frac12}|\nabla_L|^{-\frac32}\GG_{\neq}\de_yU_0^2\rangle+\langle A\de_r|\nabla_L|^{-2}U^r_{\neq},\e^{\lambda\nu^{\frac13}t}\de_y^L|\nabla_{x,z}|^{\frac12}|\nabla_L|^{-\frac32}\GG_{\neq}B\de_yU_0^2\rangle\\
% &=I^{HL}+I^{LH}
\end{align}
after using \eqref{eq:SimVar}, which can be bounded by 
\begin{align}
    |\langle \cA U^r_{\neq},\cA\, \de_r\cP^{2,2}(U_{\neq},U_0)\rangle|\lesssim \cI_1+\cI_2.
\end{align}
For $\cI_1$ we have, mimicking \eqref{eq:distributing derivatives} and \eqref{eq:I_1234},
\begin{align}
\cI_1&\lesssim\norm{\de_r|\nabla_L|^{-2}\cA U^r_{\neq}}\norm{|\nabla_{x,z}|^{\frac12}|\nabla_L|^{-\frac32}\cA\GG_{\neq}}\norm{\de_y^3|\de_z|^{\frac12}|\nabla|^{-\frac32}\GG_0}_{L^\infty}\notag\\
&\quad +\norm{ \de_r\de_y^L|\nabla_L|^{-2}\cA U^r_{\neq}}\norm{|\nabla_{x,z}|^{\frac12}|\nabla_L|^{-\frac32}\cA\GG_{\neq}}\norm{\de_y^2|\de_z|^{\frac12}|\nabla|^{-\frac32}\GG_0}_{L^\infty}.
\end{align}
Hence, using \eqref{eq:other norm bounds} and $\cM_3$, we obtain
\begin{align}
\cI_1&\lesssim \norm{\sqrt{-\frac{\dot \cM}{\cM}}\cA U^r_{\neq}}\norm{\sqrt{-\frac{\dot \cM}{\cM}}\cA\GG_{\neq}}\norm{\GG_0}_{H^{2m}}+\norm{\cA U^r_{\neq}}\norm{\sqrt{-\frac{\dot \cM}{\cM}}\cA\GG_{\neq}}\norm{\GG_0}_{H^{2m}},
\end{align}
which implies via the bootstrap assumptions that 
\begin{align}
\int_0^\infty\cI_1 &\lesssim \eps^3 +(\nu^{-\frac16}\eps)\eps^2.
\end{align}
On the other hand, we can treat $\cI_2$ similarly and obtain via \eqref{eq:other norm bounds} and the multiplier $\cM_3$ that
\begin{align}
\cI_2&\lesssim \norm{\cA U^r_{\neq}}\norm{\cA\GG_{\neq}}\norm{\GG_0}_{H^{2m}},
\end{align}
which, by the bootstrap assumptions leads to 
\begin{align}
\int_0^\infty \cI_2  & \lesssim (\nu^{-\frac13}\eps)\eps^2.
\end{align}
This shows \eqref{eq:P-22-neq0}.
Finally, we analyse the term containing $\cP^{2,3}(U_{\neq},U_0)_{\neq}$, which is 
\begin{align}
\langle \cA U^r_{\neq}, \cA \,\de_r\cP^{2,3}(U_{\neq},U_0)\rangle
&=\langle \cA U^r_{\neq},\cA \de_r|\nabla_L|^{-2}(\de_y^LU_{\neq}^3\de_z|\de_z|^{\frac12}|\nabla|^{-\frac32}\GG_0)\rangle,
\end{align}
and can be bounded similarly to \eqref{eq:distributing derivatives} and \eqref{eq:I_1234} with 
\begin{align}
|\langle \cA U^r_{\neq}, \cA \,\de_r\cP^{2,3}(U_{\neq},U_0)\rangle|&\lesssim \norm{\de_y^L\de_r|\nabla_L|^{-2}\cA U^r_{\neq}}\norm{\cA  U_{\neq}^3}\norm{\de_z|\de_z|^{\frac12}|\nabla|^{-\frac32}\GG_0}_{L^\infty}\notag\\
&\quad +\norm{\de_r|\nabla_L|^{-2}\cA U^r_{\neq}}\norm{\cA U_{\neq}^3}\norm{\de_y\de_z|\de_z|^{\frac12}|\nabla|^{-\frac32}\GG_0}_{L^\infty}\notag\\
&\quad +\norm{\de_r|\nabla_L|^{-2}\cA U^r_{\neq}}\norm{\e^{\lambda\nu^{\frac13}t}\de_y^LU_{\neq}^3}_{L^\infty}\norm{\de_z|\de_z|^{\frac12}|\nabla|^{-\frac32}\GG_0}_{H^{2m}}.
\end{align}
Using \eqref{eq:other norm bounds} and the multiplier $\cM_3$, this term can be further bounded as
\begin{align}
|\langle \cA U^r_{\neq}, \cA \,\de_r\cP^{2,3}(U_{\neq},U_0)\rangle|&\lesssim \norm{\cA U^r_{\neq}}\norm{\cA U_{\neq}^3}\norm{\GG_0}_{H^{2m}}+\norm{\sqrt{-\frac{\dot \cM}{\cM}}\cA U^r_{\neq}}\norm{\cA U_{\neq}^3}\norm{\GG_0}_{H^{2m}}\\
&\quad +\norm{\sqrt{-\frac{\dot \cM}{\cM}}\cA U^r_{\neq}}\norm{\nabla_L\cA U_{\neq}^3}\norm{\GG_0}_{H^{2m}}
\end{align}
and hence, via the bootstrap assumptions in Theorem \ref{thm:bootstrap_step} we conclude the validity \eqref{eq:P-23-neq0}.
% \begin{align}
% \int_0^t|\langle \cA U^r_{\neq}, \cA \,\de_r\cP^{2,3}(U_{\neq},U_0)\rangle| &\lesssim (\nu^{-\frac13}\eps)\eps^2 + (\nu^{-\frac16}\eps)\eps^2 +(\nu^{-\frac12}\eps)\eps^2.
% \end{align}
Thus \eqref{eq:P_r_bds} is proven, and Lemma \ref{lem:Urnonlin} follows. This was the last step missing in the proof of Proposition \ref{prop:U13}.

\section{Analysis of the zero and double zero modes}\label{sec:zeros}
This section is devoted to the dynamics of the zero modes in the $x$ frequency, which satisfy
\begin{subequations}
\begin{alignat}{4}
\de_tU^1_0&=\nu\Delta U^1_0-U^2_0 + \cT(U,U^1)_0, \\
\de_tU^3_0&=\nu\Delta U^3_0 +\de_zP_0 + \cT(U,U^3)_0 +\de_z\cP(U,U)_0,   \\
\de_t\GG_0 &=\nu\Delta \GG_0 +\beta |\de_z||\nabla_{y,z}|^{-1}\Gamma_0 + \cT_\star(U,U^2)_0+\de_y\cP_\star(U,U)_0,\\
\de_t\Gamma_0 & = \nu\Delta\Gamma_0 -\beta |\de_z||\nabla_{y,z}|^{-1}\GG_0+\cT_\circ (U,\Theta)_0,
\end{alignat}
\label{eq:zero_shorthand}
\end{subequations}
as well as the double zero modes in the $x$ and $z$ frequencies.

To not overburden the notation, we have refrained from explicitly insisting on the absence of $k=0$ modes in the notation. Instead, we remind the reader that the unknowns here are mean-free functions on $(y,z)\in\RR\times\TT$ (so e.g.\ $P_0=-\beta\de_y|\nabla_{y,z}|^{-2}\Theta_0$.) We highlight moreover that by construction, $G_0$ and $\Gamma_0$ satisfy
\begin{equation}\label{eq:GG0Gamma0zmean}
    \int_{\TT}G_0(y,z)\dd z=\int_{\TT}\Gamma_0(y,z)\dd z=0.
\end{equation}
To control the zero modes, the oscillatory nature of the coupling between $G_0$ and $\Gamma_0$ will play a crucial role: In Section \ref{ssec:lin_disp} we will study this at the linearized level and derive a (dispersive) amplitude decay estimate for the associated semigroup. Under the bootstrap assumptions, this will give improved $L^\infty$ bounds for $U_0^2$, $\widetilde{\Theta}_0$ and $\widetilde{U}_0^3$ (see Section \ref{ssec:disp_Duhamel}), allowing us to conclude the bootstrap arguments for both the double zero modes in Section \ref{ssec:btstrap_doublezero}, and for the zero modes in Section \ref{ssec:btstrap_zero}. The (passive) dynamics of $U^1_0$ are treated in Section \ref{ssec:U10}.

\subsection{Linear dispersive estimate}\label{ssec:lin_disp}
We investigate here the linearized dynamics of $G_0$ and $\Gamma_0$ in \eqref{eq:zero_shorthand}. That is, for two functions
\begin{equation}
 g_0,\gamma_0:\mathbb{R}\times\mathbb{T}\to\mathbb{R},\qquad \textnormal{with 
}\quad \int_{\RR\times\TT}g_0(y,z)\dd z=\int_{\RR\times\TT}\gamma_0(y,z)\dd z=0,
\end{equation}
we consider
\begin{equation}
 \de_t g_0 =\nu\Delta g_0 +\beta |\de_z||\nabla_{y,z}|^{-1}\gamma_0 ,\qquad
\de_t\gamma_0  = \nu\Delta\gamma_0 -\beta |\de_z||\nabla_{y,z}|^{-1}g_0,  
\end{equation}
or equivalently 
\begin{equation}
 \de_t( g_0+i\gamma_0)=\cL (g_0+i\gamma_0),\qquad \cL:=\nu\Delta -\beta\cR,\qquad \cR:=i\abs{\de_z}|\nabla_{y,z}|^{-1}.
\end{equation}
As we will show next, the operator $\cR$ is of dispersive nature. To see this, we employ the Fourier transform: For $f:\RR\times\TT\to\CC$, we have with $f_l(y):=\frac{1}{2\pi}\int_{\TT}\e^{-ilz}f(y,z)\dd z$ that
\begin{equation}
 \cR f(y,z)= \sum_{l\neq 0}\e^{il z}\cR^l f_l(y),
\end{equation}
where for $h\in\cS(\RR)$ we let
\begin{equation}
 \cR^l h(y):=\mathcal{F}^{-1}_{\eta}\left(i\frac{\abs{l}}{\abs{\eta,l}}\hat{h}(\eta)\right)(y).
\end{equation}
The semigroup generated by $\cR^l$ engenders the following amplitude decay:
\begin{proposition}\label{prop:disp_est}
    For any $\mu>0$, there exists a constant $C>0$ such that for any $h\in\cS(\RR)$ and $l\neq 0$, $a\geq 0$ there holds that
    \begin{equation}\label{eq:disp_est}
        \abs{l}^a\norm{(l^2-\de_y^2)^{-\frac{a}{2}}\e^{t\beta \cR^l}h}_{L^\infty(\RR)}  \leq C |l| (t\beta )^{-\frac13}\norm{h}_{W^{\frac32+\mu,1}(\RR)}.
    \end{equation}
\end{proposition}
The proof makes use of Littlewood-Paley theory to decompose $h$, then establishes decay of the localized pieces, from which summation gives the claim.
\begin{proof}
 Let $l\neq 0$ be fixed, and consider first the case $a=0$. Let $P_jh(y)$, $j\in\ZZ$, be the projections associated to a standard Littlewood-Paley decomposition,
\begin{equation}
    h=\sum_{j\in\ZZ}P_jh,\qquad \cF (P_jh) (\eta)=\varphi_j(\eta)\hat h(\eta),\qquad \varphi_j(\eta):=\varphi(2^{-j}\eta),\quad \varphi(\eta):=\phi(\eta)-\phi(2\eta),
\end{equation}
where $\phi:\RR\to [0,1]$ is a smooth, even bump function with $\supp(\phi)\subset [-2,2]$ and $\phi|_{[-3/2,3/2]}=1$.

Then by Young's convolution inequality we have
\begin{align}
   \norm{\e^{t\beta \cR^l} h}_{L^{\infty}} &\leq  \sum_{j\in\ZZ}\norm{\e^{t\beta \cR^l}  P_jh }_{L^{\infty}}\notag\\
   &\leq \sum_{j\in\ZZ}\norm{\cF^{-1}\left(\e^{-it\beta\frac{\abs{l}}{|\eta,l|}}\varphi(2^{-j}\eta)\right)*\cF^{-1}\left({\varphi}^{\dag}(2^{-j}\eta)\hat h(\eta)\right)}_{L^\infty}\notag\\
   &\leq \sum_{j\in\ZZ} \norm{\e^{t\beta \cR^l}\check{\varphi}_j}_{L^\infty}\norm{{P}^{\dag}_jh}_{L^1}.\label{eq:summation localization}
\end{align}
Here $\varphi^\dag$ has similar support properties as $\varphi$ and satisfies $\varphi=\varphi\varphi^{\dag}$, and ${P}^\dag_j$ is the associated Littlewood-Paley decomposition.

We first study the behaviour of the semigroup on the localised terms. By the change of variables $\eta=l\xi$ it follows that
\begin{align}
    \e^{t\beta \cR^l}\check{\varphi}_j(y)&= \int_\RR \e^{iy\eta - i t \beta \frac{\abs{l}}{|\eta,l|}}\varphi_j(\eta)\dd \eta\notag \\
    & = |l|\int_\RR \e^{iyl\xi - i t \beta \frac{1}{|\xi,1|}}\varphi(2^{-j}l\xi)\dd \xi \notag\\
    & = |l|\int_\RR \e^{i t\beta \Phi(\xi)}\varphi(2^{-j}l\xi)\dd \xi\notag\\
    &=:|l|I(t\beta, l,j),
\end{align}
where the phase function $\Phi(\xi)$ is
\begin{equation}
    \Phi(\xi)=\frac{yl}{t\beta}\xi -\frac{1}{(1+\xi^2)^{1/2}}.
\end{equation}
We will bound $I(t\beta,l,j)$ using the method of stationary phase. To this end, we compute
\begin{equation}
    \Phi''(\xi)=\frac{1-2\xi^2}{(1+\xi^2)^{5/2}}=0\quad\Leftrightarrow\quad \xi=\pm\xi_0,\quad \xi_0=1/\sqrt{2}.
\end{equation}
Our analysis can be divided into two essential cases, depending on whether $j$ is such that $\Phi''$ can vanish on the support of $\varphi(2^{-j}l\xi)$ or not: Assuming without loss of generality that $l>0$ we decompose
\begin{equation}
 I(t\beta, l,j)=I_-(t\beta, l,j)+I_+(t\beta, l,j),\qquad I_+(t\beta, l,j):=\int_{\RR_+}\e^{i t\beta \Phi(\xi)}\varphi(2^{-j}l\xi)\dd \xi, 
\end{equation}
and restrict our attention to the case of $I^+$. We note that there holds that
\begin{align}
    \xi_0\in [2^{j-2}l^{-1},2^{j+2}l^{-1}]&\;\Leftrightarrow\; j\in[j_0,j_0+4],\qquad j_0:=\log_2 l-\frac{5}{2}.
    % &\Rightarrow 2^{-j}l\in \left[\frac{\sqrt{2}}{4},4\sqrt{2}\right]\\
    % &\Leftrightarrow 2^j \in \left[\frac{l}{4\sqrt{2}},\frac{4l}{\sqrt{2}}\right] \\
    % &\Leftrightarrow j\in \left[\log_2 l-\frac{5}{2},\log_2 l+\frac{3}{2}\right].
\end{align}

\emph{Case 1: $j\in[j_0,j_0+4]$.} 
Letting $\delta>0$ a small parameter to be chosen later, we split the integral into two parts
\begin{equation}
\begin{aligned}
    I_+(t\beta,l,j)&=I_+^1(t\beta,l,j)+I_+^2(t\beta,l,j),\\
    I_+^1(t\beta,l,j)&=\int_{[\xi_0-\delta,\xi_0+\delta]}\e^{i t\beta \Phi(\xi)}\varphi(2^{-j}l\xi)\dd \xi,\\
    I_+^2(t\beta,l,j)&=\int_{\RR_+\setminus[\xi_0-\delta,\xi_0+\delta]}\e^{i t\beta \Phi(\xi)}\varphi(2^{-j}l\xi)\dd \xi.
\end{aligned}    
\end{equation}
A simple set size estimate gives that
\begin{equation}\label{eq:bd_I1_disp}
 \abs{I_+^1(t\beta,l,j)}\leq 2\delta.
\end{equation}
To bound $I_2$, notice that there exists a small constant $c>0$, independent of $l,j$, such that
\begin{equation}
    |\Phi''(\xi)|=|1-\sqrt{2}\xi|\frac{1+\sqrt{2}\xi}{(1+\xi^2)^{5/2}}\geq c \delta.
\end{equation}
Applying van der Corput's Lemma (see e.g.\ \cite{SteinHA}*{Chapter 8}), we have 
\begin{align}
    |I_+^2(t\beta, j,l)|&\leq C(\delta t\beta)^{-\frac12}\left[\norm{\varphi}_{L^{\infty}}+\int_{\RR_+\setminus [\xi_0-\delta,\xi_0+\delta]}\left|\frac{\dd}{\dd \xi}\varphi(2^{-j}l\xi)\right|\dd \xi \right]\leq 2C(\delta t \beta)^{-\frac12}.
\end{align}
Combining this with \eqref{eq:bd_I1_disp} and choosing $\delta=(t\beta)^{-\frac13}$ yields
\begin{equation}
    |I_+(t\beta, j,l)|\leq C (t\beta)^{-\frac13}.
\end{equation}

\emph{Case 2: $j\neq[j_0,j_0+4]$.} Then $|\Phi''(\xi)|\geq c>0$ is bounded away from zero. If $j< j_0$, a rough estimate gives
\begin{equation}
    |I_+(t\beta,j,l)|\leq C2^{j}|l|^{-1}.
\end{equation}
Alternatively, since in this case $|\xi|\lesssim 1/8$, which implies that there exists a constant $C>0$ independent of $j,l$, such that $\Phi''(\xi)>C,$ another application of van der Corput yields
\begin{equation}
    |I_+(t\beta,j,l)|\leq C (t\beta)^{-\frac12}.
\end{equation}
On the other hand, when $j>j_0+4$ we have that $|\xi|\leq C 2^j|l|^{-1}$, and hence
\begin{equation}
    |\Phi''(\xi)|\geq C|\xi|^{-3} \geq C 2^{-3j}|l|^3.
\end{equation}
By van der Corput this implies that
\begin{equation}
    |I_+(t\beta,j,l)|\leq C (t \beta)^{-\frac12} 2^{3/2 j}|l|^{-3/2}.
\end{equation}

\smallskip
We summarize the estimates in those cases as
\begin{equation}
    |I(t\beta,j,l)|\lesssim 
    \begin{cases}
        \min\{2^{j}|l|^{-1},(\beta t)^{-\frac12}\} &\qquad j<j_0,\\
         (t\beta)^{-\frac13} &\qquad j_0\leq j\leq j_0+4,\\ 
         (t\beta)^{-\frac12} 2^{\frac32 j}|l|^{-\frac32} &\qquad j_0+4<j.
    \end{cases}
\end{equation}
Inserting this in \eqref{eq:summation localization} we obtain 
\begin{align}
    \norm{\e^{t\beta \cR^l} h}_{L^\infty}&\lesssim \sum_{j<j_0}\min\{2^{j}|l|^{-1},(t\beta)^{-\frac12}\}\norm{{P}^{\dag}_j h}_{L^1} +|l|(t\beta)^{-\frac13}\sum_{j_0}^{j_0+4}\norm{{P}^{\dag}_j h}_{L^1} \\
    &\qquad +|l|^{-\frac12}(t\beta)^{-\frac12}\sum_{j>j_0+4} 2^{\frac32 j}\norm{{P}^{\dag}_j h}_{L^1} \\
    & \lesssim |l|(t\beta)^{-\frac13} \norm{h}_{L^1} +|l|^{-\frac12}(t\beta)^{-\frac12} \norm{h}_{\dot W^{3/2+\mu,1}}\\
    &\lesssim |l|(t\beta)^{-\frac13}\norm{h}_{W^{\frac32+\mu,1}},
\end{align}
and the proposition is proved in case $a=0$.
For $a>0$ it suffices to observe that
\begin{equation}
 \norm{(l^2-\de_y^2)^{-\frac{a}{2}}P_jh}_{L^1}\lesssim \abs{l}^{-a}\norm{{P}^{\dag}_jh}_{L^1}.
\end{equation}
\end{proof}

In particular, it follows that for functions with mean zero in $z$, amplitudes decay as follows:
\begin{corollary}\label{coroll:disp_est}
 There exists a constant $C>0$ such that for all $f\in\cS(\RR\times\TT)$ with $\int_{\TT}f(y,z) \dd z=0$ there holds for any $a\geq 0$ that
\begin{equation}
    \norm{\abs{\de_z}^a\abs{\nabla_{y,z}}^{-a}\e^{\cL t}f}_{L^\infty(\RR\times\TT)} \leq C \e^{-\nu t}(t \beta)^{-\frac13}\norm{f}_{W^{4,1}(\RR\times\TT)}.
\end{equation}
\end{corollary}
\begin{proof}
   By Proposition \ref{prop:disp_est} we have
    \begin{align}
        |\abs{\de_z}^a\abs{\nabla_{y,z}}^{-a}\e^{\cL t}f(y,z)| &= \left| \sum_{l\neq 0} \e^{izl-\nu l^2 t}\int_{\eta}\e^{iy\eta -\nu t \eta^2 + it\beta \frac{\abs{l}}{|\eta,l|}}\abs{l}^a (l^2+\eta^2)^{-\frac{a}{2}}\hat{f}_l (\eta)\dd \eta \right|\\
        &\leq \sum_{l\neq 0}\e^{-\nu t}\abs{l}^{a}\left|\int_\eta \e^{iy\eta + it\beta \frac{\abs{l}}{|\eta,l|}} \e^{-\nu t\eta^2}(l^2+\eta^2)^{-\frac{a}{2}}\hat{f}_l (\eta)\dd \eta \right| \\
        &\leq \e^{-\nu t} \sum_{l\neq 0} \abs{l}^{a}\norm{\e^{\cR t}\left(\e^{\nu t \de_y^2}(l^2-\de_y^2)^{-\frac{a}{2}}f_l\right)}_{L^\infty(\RR)}\\
        &\lesssim \e^{-\nu t}(t\beta)^{-\frac13} \sum_{l\neq 0} |l| \norm{\e^{\nu t \de_y^2}f_l}_{W^{3/2+\mu,1}(\RR)}\\
        &\lesssim \e^{-\nu t}(t\beta)^{-\frac13} \sum_{l\neq 0} |l|^{-1-\mu} |l|^{2+\mu}\norm{f_l}_{W^{3/2+\mu,1}(\RR)}\\
        &\lesssim \e^{-\nu t}(t\beta)^{-\frac13} \sum_{l\neq 0} |l|^{-1-\mu} \norm{f}_{W^{3/2+\mu,1}_y W^{2+\mu,1}_z(\RR\times\TT)}.
    \end{align}
\end{proof}

\subsection{$L^\infty$ bounds via Duhamel}\label{ssec:disp_Duhamel}
As a direct consequence of the amplitude decay result in Corollary \ref{coroll:disp_est}, under the bootstrap assumptions we obtain improved $L^\infty$ bounds on $U_0^2$ and $\widetilde{\Theta}_0$ and their derivatives.
\begin{proposition}\label{prop:LinftyGGamma}
    Under the hypothesis of Theorem \ref{thm:bootstrap_step}, there holds that 
    \begin{equation}
        \sum_{\alpha\in\NN_0^2,\abs{\alpha}\leq 2}\norm{\nabla_{y,z}^\alpha U^2_0(t)}_{L^\infty}+\norm{\nabla_{y,z}^\alpha\widetilde{\Theta}_0(t)}_{L^\infty}+\norm{\nabla_{y,z}^\alpha\widetilde{U}^3_0(t)}_{L^\infty}\lesssim (t\beta)^{-\frac13}\e^{-\nu t}\varepsilon + \beta^{-\frac13}\nu^{-\frac23}\eps^2.
    \end{equation}
\end{proposition}
\begin{proof}
 For simplicity of notation we will work with the complex unknown
 \begin{equation}\label{eq:defUpsilon}
  \Upsilon:=\abs{\de_z}^{\frac12}\abs{\nabla_{y,z}}^{-\frac32}(\GG_0+i\Gamma_0)=U_0^2+i\abs{\de_z}\abs{\nabla_{y,z}}^{-1}\widetilde{\Theta}_0.
 \end{equation}
 We will show that under the above assumptions there holds that
 \begin{equation}
    \sum_{a=1}^3\norm{\de_y^a\Upsilon(t)}_{L^\infty}\lesssim (t\beta)^{-\frac13}\e^{-\nu t}\varepsilon + \beta^{-\frac13}\nu^{-\frac23}\eps^2,
 \end{equation}
 which yields the claim for $U_0^2$. The bounds for $\widetilde{\Theta}_0$ follow analogously, and for $\widetilde{U}_0^3$ we use that by incompressibility there holds that $\widetilde{U}_0^3=-\de_z^{-1}\de_yU_0^2$.
 
 By \eqref{eq:zero_shorthand}, $\Upsilon$ satisfies the nonlinear equation
 \begin{equation}\label{eq:Upsnonlin}
   \de_t\Upsilon=\cL\Upsilon+\cN(U,\Upsilon),
 \end{equation}
 and thus
 \begin{equation}\label{eq:DHUpsilon}
  \Upsilon(t)=e^{t\cL}\Upsilon(0)+\int_{0}^te^{(t-\tau)\cL}\cN(U,\Upsilon)(\tau)\dd\tau,
 \end{equation}
 where
 \begin{equation}\label{eq:Upsnonlin-defN}
 \begin{aligned}
  \cN(U,\Upsilon)&=\cN_1(U,\Upsilon)+i\abs{\de_z}\abs{\nabla_{y,z}}^{-1}\cN_2(U,\Upsilon),\\
  \cN_1(U,\Upsilon)&=\cT(U,U^2)_0+\de_y\cP(U,U)_0,\quad \cN_2(U,\Upsilon)=\cT(U,\Theta)_0.
 \end{aligned}
 \end{equation}
 From \eqref{eq:GG0Gamma0zmean} we see that
 \begin{equation}
  \int_{\TT}\Upsilon(\tau,y,z)\dd z=\int_{\TT}\cN_j(U,\Upsilon)(\tau,y,z)\dd z=0, \qquad \tau\geq 0, \quad j=1,2.
 \end{equation}
 We can then apply Corollary \ref{coroll:disp_est} to deduce that
 \begin{equation}
 \begin{aligned}
  \norm{\nabla_{y,z}^\alpha\Upsilon(t)}_{L^\infty}&\lesssim \e^{-\nu t}(t \beta)^{-\frac13}\left(\norm{\nabla_{y,z}^\alpha U^2(0)}_{W^{4,1}(\RR\times\TT)}+\norm{\nabla_{y,z}^\alpha \widetilde{\Theta}(0)}_{W^{4,1}(\RR\times\TT)}\right)\\
  &\qquad+\int_0^{t}\e^{-\nu (t-\tau)}((t-\tau) \beta)^{-\frac13}  \norm{\nabla_{y,z}^\alpha\cN_1(U,\Upsilon)(\tau)}_{W^{4,1}(\RR\times\TT)}\dd\tau\\
  &\qquad+\int_0^{t}\e^{-\nu (t-\tau)}((t-\tau) \beta)^{-\frac13}  \norm{\nabla_{y,z}^\alpha\cN_2(U,\Upsilon)(\tau)}_{W^{4,1}(\RR\times\TT)}\dd\tau.
  % &+\int_{t-1}^t \norm{\cN(U,\Upsilon)(\tau)}_{H^{2}(\RR\times\TT)}\dd\tau.
 \end{aligned}
 \end{equation}
Noting that
\begin{equation}\label{eq:decay_int}
 \norm{t^{-\frac13}\e^{-\nu t}}_{L^2_t}\lesssim \nu^{-\frac16},   
\end{equation}
by the bootstrap assumptions it follows that for $\abs{\alpha}\leq 3$ and $j=1,2$ we have
\begin{equation}
\begin{aligned}
 \norm{\nabla_{y,z}^\alpha\cN_j(U,\Upsilon)(\tau)}_{L^2_tW^{4,1}(\RR\times\TT)}&\lesssim \!\left(\sum_{F\in \{\GG,\Gamma,U^1,U^3\}}\!\!\!\!\left(\norm{\nabla_L \cA F_{\neq}}_{L^2_tL^2}+\norm{\nabla F_0}_{L^2_tH^{2m}}\right)\!+\!\norm{\de_y \overline{\Theta}_0}_{L^2_tH^{2m+1}}\right)\\
 &\qquad \cdot \left(\sum_{F\in \{\GG,\Gamma,U^1,U^3\}}\!\!\!\!\left(\norm{\cA F_{\neq}}_{L^\infty_tL^2}+\norm{F_0}_{L^\infty_tH^{2m}}\right)\!+\!\norm{ \overline{\Theta}_0}_{L^\infty_tH^{2m+1}}\right)\\
 &\lesssim \nu^{-\frac12}\eps^2,
\end{aligned} 
\end{equation}
so that
\begin{equation}
 \int_0^{t}\e^{-\nu (t-\tau)}((t-\tau) \beta)^{-\frac13}  \norm{\nabla_{y,z}^\alpha\cN_j(U,\Upsilon)(\tau)}_{W^{4,1}(\RR\times\TT)}\dd\tau\lesssim \beta^{-\frac13}\nu^{-\frac23}\eps^2.   
\end{equation}
Together with \eqref{eq:decay_int} his proves the claim.
\end{proof}

We give next a decomposition of $U^2_0$ that will be crucial for obtaining a large threshold for simple and double zero modes.
\begin{lemma}\label{lemma:innl_decomp}
 We can decompose
 \begin{equation}
    U_0^2=U_0^{2,in}+U_0^{2,nl}, % \Upsilon(t)=\Upsilon^{in}(t)+\Upsilon^{nl}(t)
 \end{equation}
 where
 \begin{equation}\label{eq:Upsdecomp_bds}
 \begin{aligned}
  \norm{\l\nabla\r^{2m}U_0^{2,in}(t)}_{L^\infty}&\lesssim \eps (\beta t)^{-\frac13}\e^{-\nu t},\\
  \norm{U_0^{2,nl}(t)}_{H^{2m}}&\lesssim \eps\left(\nu^{-\frac23}\eps+\nu^{-\frac53}\eps^2\right)+\beta^{-\frac13}\eps\left(\nu^{-\frac23}\eps+\nu^{-\frac53}\eps^2+\nu^{-\frac83}\eps^3\right).
  % \norm{\abs{\de_z}^{1/2}\abs{\nabla_{y,z}}^{-3/2}\Upsilon^{nl}(t)}_{H^{2m}}&\lesssim \eps\left(\nu^{-\frac23}\eps+\nu^{-\frac53}\eps^2\right)+\beta^{-\frac13}\eps\left(\nu^{-\frac23}\eps+\nu^{-\frac53}\eps^2+\nu^{-\frac83}\eps^3\right).
 \end{aligned} 
 \end{equation}
\end{lemma}

\begin{proof}
We will give the corresponding result for $\Upsilon$ from \eqref{eq:defUpsilon}, which gives the claim since $U^2_0=\Re\Upsilon$. We decompose according to Duhamel's formula \eqref{eq:DHUpsilon} and the notation in \eqref{eq:Upsnonlin-defN}, letting
\begin{equation}
 \Upsilon^{in}(t):=e^{t\cL}\Upsilon(0),\quad \Upsilon^{nl}(t):=\int_{0}^te^{(t-\tau)\cL}\cN(U,\Upsilon)(\tau)\dd\tau.   
\end{equation}
The estimate for $\Upsilon^{in}$ follows directly. For $\Upsilon^{nl}$ more care is needed and we treat separately the transport and pressure terms constituting $\cN$ -- recall \eqref{eq:Upsnonlin-defN}.

\emph{Transport terms.} We have that
\begin{equation}
\begin{aligned}
  \cT(U,U^2)_0&=-(U_{\neq}\cdot\nabla_L U^2_{\neq})_0-U_0\cdot\nabla U_0^2,\\
  \abs{\de_z}\abs{\nabla_{y,z}}^{-1}\cT(U,\Theta)_0&=\abs{\de_z}\abs{\nabla_{y,z}}^{-1}(U_{\neq}\cdot\nabla_L \Theta_{\neq})_0-\abs{\de_z}\abs{\nabla_{y,z}}^{-1}U_0\cdot\nabla \Theta_0,
\end{aligned}  
\end{equation}
and observe that by the divergence condition and $\de_zU_0^2=0$ we have for $F\in\{U^2,\Theta\}$ that\footnote{Recall the notation $\widetilde{F}_0=F_0-\overline{F}_0$ of \eqref{eq:zmean-notation}.}
\begin{equation}
\begin{aligned}
 U_0\cdot\nabla F_0&=U_0^2\de_yF_0+U_0^3\de_z F_0=U_0^2\de_yF_0+(U_0^3-\overline{U}_0^3)\de_z F_0+\overline{U}_0^3\de_z F_0\\
 &=U_0^2\de_y\widetilde{F}_0-\de_z^{-1}\de_yU_0^2\de_z F_0+\overline{U}_0^3\de_z F_0+U_0^2\de_y\overline{F}_0.
\end{aligned} 
\end{equation}
We next bound the contributions from these terms individually. First, we have that 
$$
\begin{aligned}
 \norm{(U_0^2\de_y\widetilde{F}_0)(\tau)}_{H^{2m}}&\lesssim \norm{U_0^2(\tau)}_{L^\infty}\norm{\de_y \widetilde{F}_0(\tau)}_{H^{2m}}+\norm{U_0^2(\tau)}_{H^{2m}}\norm{\de_y \widetilde{F}_0(\tau)}_{L^\infty}\\
% &\lesssim \beta^{-\frac13}\eps \left(\tau^{-\frac13}\e^{-\nu\tau}+\eps\nu^{-\frac23}\right)\blue{\left(\norm{\GG_0(\tau)}_{H^{2m}}+\norm{\Gamma_0(\tau)}_{H^{2m}}^{\frac12}\norm{\nabla \Gamma_0(\tau)}_{H^{2m}}^{\frac12}\right)},\\
 &\lesssim \beta^{-\frac13}\eps \left(\tau^{-\frac13}\e^{-\nu\tau}+\eps\nu^{-\frac23}\right)\left(\norm{\GG_0(\tau)}_{H^{2m}}+\norm{\Gamma_0(\tau)}_{H^{2m}}+\norm{\nabla \Gamma_0(\tau)}_{H^{2m}}\right),
\end{aligned} 
$$
having used Proposition \ref{prop:LinftyGGamma} and that $\widetilde{U}^2_0=U^2_0=\abs{\de_z}^{\frac12}\abs{\nabla_{y,z}}^{-3/2}G_0$ resp.\ $\widetilde{\Theta}_0=\abs{\de_z}^{-\frac12}\abs{\nabla_{y,z}}^{-\frac12}\Gamma_0$. Similarly we obtain that
\begin{equation}
 \norm{(\de_z^{-1}\de_yU_0^2\de_z F_0)(\tau)}_{H^{2m}}\lesssim \beta^{-\frac13}\eps \left(\tau^{-\frac13}\e^{-\nu\tau}+\eps\nu^{-\frac23}\right)\left(\norm{\GG_0(\tau)}_{H^{2m}}+\norm{\Gamma_0(\tau)}_{H^{2m}}\right),  
\end{equation}
and integration then yields that for $I\in\{U_0^2\de_y\widetilde{F}_0,\de_z^{-1}\de_yU_0^2\de_z F_0\}$ we have the bound
\begin{equation}
\begin{aligned}
 \int_0^{t}\norm{ \e^{(t-\tau)\cL}I(\tau)}_{H^{2m}}\dd\tau &\lesssim \beta^{-\frac13}\eps^2 \int_0^t \e^{-(t-\tau)\nu}\left(\tau^{-\frac13}\e^{-\nu\tau}+\eps\nu^{-\frac23}\right)\dd\tau\\
 &\quad +  \beta^{-\frac13}\eps\e^{-\nu t}\int_0^t\tau^{-\frac13}\norm{\nabla\Gamma_0}_{H^{2m}}\dd \tau \\
 &\quad +  \beta^{-\frac13}\nu^{-\frac23}\eps^2\int_0^t\e^{-(t-\tau)\nu}\norm{\nabla\Gamma_0}_{H^{2m}}\dd \tau \\ 
% &\lesssim \beta^{-\frac13}\eps^2 (\nu^{-\frac23}+\eps\nu^{-\frac53})+ \beta^{-\frac13}\nu^{-\frac14}\eps^2 \e^{-t\nu}\left(\int_0^t \tau^{-\frac49}\dd\tau\right)^{\frac34}+\nu^{-\frac{11}{12}}\eps^3\left(\int_0^t\e^{-(t-\tau)\nu}\dd\tau\right)^{\frac34}\\
 &\lesssim \beta^{-\frac13}\eps^2 (\nu^{-\frac23}+\eps\nu^{-\frac53})
\end{aligned} 
\end{equation}

Next, for $\overline{U}_0^3\de_z F_0$, $F\in\{U^2,\Theta\}$, we use Lemma \ref{lemma:zzeroU3theta} to bound $\norm{\overline U^3_0}_{H^{2m}}$, which yields 
\begin{align}
    \norm{\overline{U}_0^3\de_z F_0}_{H^{2m}}&\lesssim\norm{\overline U^3_0}_{H^{2m}}\left(\norm{\GG_0(\tau)}_{H^{2m}}+\norm{\Gamma_0(\tau)}_{H^{2m}}\right)\lesssim (\eps^2\nu^{-\frac23}+\beta^{-\frac13}\nu^{-\frac53}\eps^3)\eps.
\end{align}
Upon integrating in time this gives
\begin{equation}
\begin{aligned}
 \int_0^{t}\norm{\e^{(t-\tau)\cL}(\overline{U}_0^3\de_z F_0)}_{H^{2m}}\dd\tau &\lesssim \int_0^t\e^{-(t-\tau)\nu}\left(\eps^3\nu^{-\frac23}+\beta^{-\frac13}\nu^{-\frac53}\eps^4\right)\dd\tau\\
 &\lesssim \nu^{-\frac53}\eps^3+\beta^{-\frac13}\nu^{-\frac83}\eps^4.
\end{aligned} 
\end{equation}
Finally, since $\overline U^2_0\equiv 0$ (see \eqref{eq:xz-meanU2}), for $U^2_0\de_y\overline F_0$ we need only consider the term $U^2_0\de_y\overline \Theta_0$, which we can simply control as
% \begin{align}
%     \norm{U^2_0\de_y\overline \Theta_0}_{H^{2m}}&\lesssim \norm{U^2_0}_{L^\infty}\norm{\de_y\overline \Theta_0}_{H^{2m}} +\norm{U^2_0}_{H^{2m}}\norm{\de_y\overline \Theta_0}_{L^\infty}\\
%     &\lesssim\norm{\Upsilon_0}_{L^\infty}\norm{\de_y\overline \Theta_0}_{H^{2m}} +\norm{\Upsilon_0}_{H^{2m}}\norm{\overline \Theta_0}_{H^{2m}}\\
%     &\lesssim \left((t\beta)^{-\frac13}\e^{-\nu t}\varepsilon + \beta^{-\frac13}\nu^{-\frac23}\eps^2\right)\norm{\de_y\overline \Theta_0}_{H^{2m}}\\
%     &\quad + \eps\left(\eps^2\nu^{-\frac23}+\beta^{-\frac13}\nu^{-\frac53}\eps^3\right),
% \end{align}
\begin{align}
    \norm{U^2_0\de_y\overline \Theta_0}_{H^{2m}}&\lesssim \norm{U^2_0}_{H^{2m}}\norm{\overline \Theta_0}_{H^{2m+1}}\lesssim \eps\left(\eps^2\nu^{-\frac23}+\beta^{-\frac13}\nu^{-\frac53}\eps^3\right),
\end{align}
having invoked Lemma \ref{lemma:zzeroU3theta}.
It follows that
\begin{equation}
 \int_0^{t}\norm{ \e^{(t-\tau)\cL}(U_0^2\de_y \overline{\Theta}_0)}_{H^{2m}}\dd\tau \lesssim \int_0^t \e^{-(t-\tau)\nu}\left(\eps^3\nu^{-\frac23}+\beta^{-\frac13}\nu^{-\frac53}\eps^4\right)\dd\tau\lesssim \nu^{-\frac53}\eps^3+\beta^{-\frac13}\nu^{-\frac83}\eps^4.
\end{equation}
% \blue{
% Then again 
% \begin{equation}
% \begin{aligned}
%  \int_0^{t}\norm{ \e^{(t-\tau)\cL}\abs{\nabla}^{2m}(\overline{U}_0^2\de_y \Theta_0)}\dd\tau &\lesssim \int_0^t \e^{-(t-\tau)\nu}\left((\tau\beta)^{-\frac13}\e^{-\nu \tau}\varepsilon + \beta^{-\frac13}\nu^{-\frac23}\eps^2\right)\norm{\de_y\overline \Theta_0}_{H^{2m}}\\
%     &\quad + \left(\eps^3\nu^{-\frac23}+\beta^{-\frac13}\nu^{-\frac53}\eps^4\right)\dd\tau\\
%  &\lesssim \beta^{-\frac13}\nu^{-\frac23}\eps^2+\beta^{-\frac13}\nu^{-\frac53}\eps^3+\nu^{-\frac53}\eps^3+\beta^{-\frac13}\nu^{-\frac83}\eps^4.
% \end{aligned} 
% \end{equation}
% }

On the other hand, we have that for $F\in\{U^2,\Theta\}$
\begin{equation}
\begin{aligned}
 \norm{(U_{\neq}\cdot\nabla_LF_{\neq})(\tau)}_{H^{2m}}&\lesssim \norm{U_{\neq}}_{L^\infty}\norm{\nabla_LF_{\neq}}_{H^{2m}}+\norm{U_{\neq}}_{H^{2m}}\norm{\nabla_LF_{\neq}}_{L^\infty}\\
 &\lesssim \norm{\cA U_{\neq}}\norm{\nabla_L\cA F_{\neq}},
\end{aligned} 
\end{equation}
and thus
\begin{equation}
\begin{aligned}
 \int_0^{t}\norm{\e^{(t-\tau)\cL}(U_{\neq}\cdot\nabla_LF_{\neq})(\tau)}_{H^{2m}}\dd\tau &\lesssim \eps^2\nu^{-\frac23}.
\end{aligned} 
\end{equation}

\medskip
\emph{Pressure terms.}
We have that
\begin{align}
    % \de_y|\de_z|^{\frac12}|\nabla_{y,z}|^{-\frac32}\cP_\star(U,U)_0= \de_y|\nabla|^{-2}(\de_iU^j_0\de_jU^i_0)+\de_y|\nabla|^{-2}(\de_i^LU^j_{\neq}\de_j^LU^i_{\neq})_0.
    \de_y\cP(U,U)_0= \de_y|\nabla_{y,z}|^{-2}(\de_iU^j_0\de_jU^i_0)+\de_y|\nabla_{y,z}|^{-2}(\de_i^LU^j_{\neq}\de_j^LU^i_{\neq})_0.
\end{align}
Here the first term can be further decomposed as
\begin{align}
    \de_iU^j_0\de_jU^i_0&=\de_yU^2_0\de_yU^2_0+2\de_yU^3_0\de_zU^2_0+\de_zU^3_0\de_zU^3_0\\
    &=2\de_y U^2_0\de_y U^2_0+2\de_y\widetilde U^3_0\de_z U^2_0+2\de_y\overline U^3_0\de_z U^2_0,
\end{align}
where we used that $U^2_0=\widetilde U^2_0$ and $\de_z\widetilde U^3_0=-\de_y U^2_0$.
For $\de_y U^2_0\de_y U^2_0$, we easily deduce from \eqref{eq:SimVar} and Proposition \ref{prop:LinftyGGamma} that
\begin{equation}
    \norm{\de_y U^2_0 \de_y  U^2_0}_{H^{2m}}\lesssim\norm{\GG_0}_{H^{2m}}\norm{\de_yU_0^2}_{L^\infty}\lesssim (t\beta)^{-\frac13}\e^{-\nu t}\eps^2 + \beta^{-\frac13}\nu^{-\frac23}\eps^3,
\end{equation}
and integration in time gives
\begin{equation}
\begin{aligned}
 \int_0^{t}\norm{ \e^{(t-\tau)\cL}(\de_y U^2_0 \de_y  U^2_0)(\tau)}_{H^{2m}}\dd\tau &\lesssim \int_0^t\e^{-(t-\tau)\nu}\left((\tau\beta)^{-\frac13}\e^{-\nu \tau}\varepsilon^2 + \beta^{-\frac13}\nu^{-\frac23}\eps^3\right)\dd \tau\\
 &\lesssim  \beta^{-\frac13}\nu^{-\frac23}\eps^2+\beta^{-\frac13}\nu^{-\frac53}\eps^3.
\end{aligned} 
\end{equation}
Similarly, for the term $\de_z U^2_0\de_y\widetilde U^3_0$ we deduce from 
\begin{align}
    \norm{\de_z U^2_0 \de_y \widetilde U^3_0}_{H^{2m}}&\lesssim \norm{\GG_0}_{H^{2m}}\norm{\de_y\widetilde{U}_0^3}_{L^\infty}+\norm{\de_z U^2_0 }_{L^\infty}\norm{\de_y\GG_0}^{\frac12}_{H^{2m}}\norm{\GG_0}^{\frac12}_{H^{2m}}\\
    &\lesssim  \left((t\beta)^{-\frac13}\e^{-\nu t}\varepsilon + \beta^{-\frac13}\nu^{-\frac23}\eps^2 \right)\left(\eps +\eps^\frac12\norm{\de_y\GG_0}^{\frac12}_{H^{2m}}\right)
\end{align}
that
\begin{equation}
\begin{aligned}
  \int_0^{t}\norm{ \e^{(t-\tau)\cL}(\de_z U^2_0\de_y\widetilde U^3_0)(\tau)}_{H^{2m}}\dd\tau 
 &\lesssim \int_0^t\e^{-(t-\tau)\nu}\left((\tau\beta)^{-\frac13}\e^{-\nu \tau}\varepsilon^2 + \beta^{-\frac13}\nu^{-\frac23}\eps^3\right)\dd \tau\\
 &\quad +  \e^{-\nu t}\int_0^t\left((\tau\beta)^{-\frac13}\varepsilon\right)\eps^\frac12\norm{\de_y\GG_0}^{\frac12}_{H^{2m}}\dd \tau \\
 &\quad +  \int_0^t\e^{-(t-\tau)\nu}\left( \beta^{-\frac13}\nu^{-\frac23}\eps^2\right)\eps^\frac12\norm{\de_y\GG_0}^{\frac12}_{H^{2m}}\dd \tau \\
 &\lesssim  \beta^{-\frac13}\nu^{-\frac23}\eps^2+\beta^{-\frac13}\nu^{-\frac53}\eps^3.
\end{aligned} 
\end{equation}
Finally, for $\de_z U^2_0 \de_y \overline U^3_0$ we proceed by bounding the $H^{2m}$ norm as 
\begin{align}
    \norm{\de_z U^2_0 \de_y \overline U^3_0}_{H^{2m}}&\lesssim \norm{G_0}_{H^{2m}}\norm{\de_y\overline U^3_0}_{H^{2m}}\lesssim \eps \norm{\de_y\overline U^3_0}_{H^{2m}}
\end{align}
and integrate to obtain 
\begin{equation}
\begin{aligned}
  \int_0^{t}\norm{ \e^{(t-\tau)\cL}(\de_z U^2_0 \de_y \overline U^3_0)(\tau)}_{H^{2m}}\dd\tau 
 &\lesssim \eps\int_0^t\e^{-(t-\tau)\nu}\norm{\de_y\overline{U}^3_0}_{H^{2m}}\dd \tau\\
 &\lesssim \eps \nu^{-\frac12}\norm{\de_y\overline{U}^3_0}_{L^2_tH^{2m}}\\
 &\lesssim \nu^{-\frac53}\eps^3+\beta^{-\frac13}\nu^{-\frac83}\eps^4.
\end{aligned} 
\end{equation}
This concludes the proof of the lemma.
% then use Proposition \ref{prop:LinftyGGamma} and Lemma \eqref{lemma:zzeroU3theta}.
\end{proof}

\begin{remark}\label{rem:innl_decomp}
  This decomposition exploits a difference in regularity of one order between $G_0$ and $U_0^2$. In particular, we do \emph{not} have the analogous decomposition and bounds for $\widetilde{U}_0^3=-\de_z^{-1}\de_yU_0^2$. To establish our main result it is thus crucial that the only zero mode forcing the double zero modes is $U_0^2$ -- see Section \ref{ssec:btstrap_doublezero} below.
\end{remark}

\subsection{Control of double zero modes -- proof of Proposition \ref{prop:doublezero}}\label{ssec:btstrap_doublezero}
The double zero (in $x,z$) modes (recall the notation \eqref{eq:zmean-notation}) play a distinguished role for the dynamics. By construction $\overline{\Theta}_0$ and $\overline{U}_0^3$ satisfy
\begin{equation}
 \de_t \overline{\Theta}_0+\de_y(\overline{U^2\Theta})_0=\nu\de_{yy}\overline{\Theta}_0,\qquad \de_t \overline{U}^3_0+\de_y(\overline{U^2U^3})_0=\nu\de_{yy}\overline{U}^3_0.
\end{equation}
In contrast, as already observed in \eqref{eq:xz-meanU2} we have that
\begin{equation}\label{eq:ulU20=0}
  \overline{U}^2_0(t)=0,\qquad t\geq 0.  
\end{equation}
\begin{lemma}\label{lemma:zzeroU3theta}
  If
  \begin{equation}\label{eq:double0id}
    \overline{U}^3_0(0)=\overline{\Theta}_0(0)=0,
  \end{equation}
  then for all $t\geq 0$
$$
  \begin{aligned}
    \norm{\overline{U}^3_0(t)}_{H^{2m}}^2+\norm{\overline{\Theta}_0(t)}_{H^{2m+1}}^2+\nu\norm{\de_y\overline{U}^3_0}^2_{L^2_tH^{2m}}+\nu\norm{\de_y\overline{\Theta}_0}_{L^2_tH^{2m+1}}^2 &\lesssim\eps^2 (\beta^{-\frac13}\eps^2\nu^{-\frac53}+\eps\nu^{-\frac23})^2.
  \end{aligned}  
$$
\end{lemma}
The crucial point here is the absence of self-interactions of the double zero modes, see \eqref{eq:noselfdouble0}.
As the proof shows, the claim of Lemma \ref{lemma:zzeroU3theta} still holds if we allow the initial data to be nonzero, but with a bound not exceeding the nonlinear contributions -- compare \eqref{eq:energy_doublezeros}.
\begin{proof}
 We begin by observing that due to \eqref{eq:ulU20=0} there holds that
 \begin{equation}\label{eq:noselfdouble0}
  (\overline{U^2F})_0=(\overline{U^2_0\widetilde{F}_0})+(\overline{U^2_{\neq} F_{\neq}})_0,\qquad F\in\{U^3,\Theta\}.   
 \end{equation}
 By energy estimates it follows that for $F\in\{U^3,\Theta\}$ and $n\in\{2m,2m+1\}$
 \begin{equation}\label{eq:energy_doublezeros}
   \norm{\overline{F}_0(t)}_{H^{n}}^{2}+\nu\int_0^t\norm{\de_y\overline{F}_0}_{H^{n}}^2\dd \tau \leq \norm{\overline{F}_0(0)}_{H^{n}}^{2}+ \nu^{-1}\int_0^t  \norm{(\overline{\widetilde{U}^2_0\widetilde{F}_0})(\tau)}_{H^{n}}^2 +\norm{(\overline{U^2_{\neq} F_{\neq}})_0(\tau)}_{H^{n}}^2\dd\tau.
 \end{equation}
 Since
 \begin{equation}
   \norm{(\overline{\widetilde{U}^2_0\widetilde{F}_0})(\tau)}_{H^{n}}\lesssim \norm{U_0^2(\tau)}_{L^\infty}\norm{ \widetilde{F}_0(\tau)}_{H^{n}}+\norm{U_0^2(\tau)}_{H^{n}}\norm{\widetilde{F}_0(\tau)}_{L^\infty},
 \end{equation}
 we obtain from Proposition \ref{prop:LinftyGGamma} that
$$
 \norm{(\overline{\widetilde{U}^2_0\widetilde{U}_0^3})(\tau)}_{H^{2m}}+\norm{(\overline{\widetilde{U}^2_0\widetilde{\Theta}_0})(\tau)}_{H^{2m+1}}\lesssim \beta^{-\frac13}\eps \left(\tau^{-\frac13}\e^{-\nu\tau}+\eps\nu^{-\frac23}\right)(\norm{\GG_0(\tau)}_{H^{2m}}+\norm{\Gamma_0(\tau)}_{H^{2m}}),
$$
having used that $\widetilde{U}_0^3=-\de_z^{-1}\de_yU_0^2$ and the expression of $U^2$, $\Theta$ in terms of $\GG,\Gamma$. Similarly it follows that
\begin{equation}
 \norm{(\overline{U^2_{\neq} U^3_{\neq}})_0(\tau)}_{H^{2m}}+\norm{(\overline{U^2_{\neq} \Theta_{\neq}})_0(\tau)}_{H^{2m+1}}\lesssim \norm{\cA \GG_{\neq}(\tau)}^2+\norm{\cA \Gamma_{\neq}(\tau)}^2+\norm{\cA U^3_{\neq}(\tau)}^2.
\end{equation}
Together with \eqref{eq:double0id}, integration in time and using also \eqref{eq:enhanced_diss} yields that
\begin{equation}
\begin{aligned}
    &\norm{\overline{U}^3_0(t)}_{H^{2m}}^2+\norm{\overline{\Theta}_0(t)}_{H^{2m+1}}^2+\nu\int_0^t\norm{\de_y\overline{U}^3_0(\tau)}_{H^{2m}}^2+\norm{\de_y\overline{\Theta}_0(\tau)}_{H^{2m+1}}^2\dd \tau  \\
    &\qquad\lesssim\eps^2\cdot\beta^{-\frac23}(\eps^2\nu^{-\frac43}+\eps^4\nu^{-\frac{10}{3}})+\eps^2\cdot \eps^2\nu^{-\frac43},
  \end{aligned}  
  \end{equation}
and thus the claim.  
\end{proof}

\subsection{Dynamics of $U^1_0$}\label{ssec:U10}
Unlike the case of $U^3_0$, incompressibility of $U$ does not yield any bounds for $U^1_0$. Rather, as motivated in the introduction it is convenient to consider the variable
\begin{equation}
V(t):=U^1(t)+\beta^{-1}\Theta(t)
\end{equation}
which satisfies
\begin{equation}\label{eq:V_0_dynamics}
 \de_t V_0+(U\cdot\nabla_L V)_0=\nu\Delta V_0.   
\end{equation}
Moreover, by assumptions \eqref{eq:zero-mean-xz-id} we have that
\begin{equation}\label{eq:double0idV}
  \overline{V}_0(0)=0.  
\end{equation}
Since we separately establish the bounds for $\Theta_0$, to obtain the required control of $U_0^1$ it suffices to establish:
\begin{lemma}
  Under the assumptions of Theorem \ref{thm:bootstrap_step}, there holds that
  \begin{equation}\label{eq:energy_V0}
   \norm{\widetilde{V}_0(t)}_{H^{2m}}^2+\nu \norm{\nabla\widetilde{V}_0}_{L^2_tH^{2m}}^2\leq \eps_0^2+C_4\eps^2 \left(\nu^{-\frac{8}{9}}\eps+ \beta^{-\frac13}\nu^{-\frac{11}{12}}\eps\right)
  \end{equation}
  and
  \begin{equation}\label{eq:energy_V00}
    \norm{\overline V_0(t)}_{H^{2m}}^2+\nu\norm{\de_y\overline V_0}_{L^2_tH^{2m}}^2\leq C_4\eps^2 \left(\nu^{-\frac{8}{9}}\eps+ \beta^{-\frac13}\nu^{-\frac{11}{12}}\eps\right)^2.  
  \end{equation}
\end{lemma}
In order to establish these bounds, we will make crucial use of the decomposition of $U^2_0$ in Lemma \ref{lemma:innl_decomp}. Since the double zero modes do not self-interact and are only forced by $U^2_0$ (see \eqref{eq:noselfdouble0V} and \eqref{eq:V00_dynamics}), the proof of \eqref{eq:energy_V00} is a more delicate version of that of Lemma \ref{lemma:zzeroU3theta}. In contrast, the simple zero modes are forced in addition also by $U^3_0$, for which we do not have a decomposition as in Lemma \ref{lemma:innl_decomp} (see also Remark \ref{rem:innl_decomp}). Instead, we need to rely on the precise nature of nonlinear interactions (see e.g.\ \eqref{eq:V0energy_cancel}) and the cancellation (as is typical for energy estimates) of highest order self-interactions of $\widetilde{U}^1_0$ (see \eqref{eq:V0energy_finedecomp}). We present this more difficult argument first.

\begin{proof}[Proof of \eqref{eq:energy_V0}]
From \eqref{eq:V_0_dynamics}, energy estimates yield that
\begin{equation}
    \norm{\widetilde{V}_0(t)}^2_{H^{2m}}  +2\nu\int_0^t \norm{\nabla \widetilde{V}_0}_{H^{2m}}^2 \leq  \norm{\widetilde{V}_0(0)}^2_{H^{2m}} + \int_0^t\l \widetilde V_0, (U\cdot \nabla_L V)_0\r_{H^{2m}}.
\end{equation}
We note that
\begin{equation}
    \l\widetilde V_0, (U\cdot \nabla_L V)_0\r_{H^{2m}}=\l\widetilde V_0, U^2_0\de_yV_0+U^3_0\de_zV_0\r_{H^{2m}}+\l\widetilde V_0, (U_{\neq}\cdot \nabla_L V_{\neq})_0\r_{H^{2m}},
\end{equation}
and from the the bootstrap assumptions \eqref{eq:bootstrap_nonzero_U}, \eqref{eq:bootstrap_nonzero_Gamma} on $\Theta_{\neq}$ and $U^1_{\neq}$ we directly obtain that
\begin{equation}
    \int_0^\infty |\l  \widetilde{V}_0,(U_{\neq}\cdot \nabla_L V_{\neq})_0\r_{H^{2m}}|\lesssim \int_0^\infty\norm{V_0}_{H^{2m}}\norm{U_{\neq}}_{H^{2m}}\norm{\nabla_L V_{\neq}}_{H^{2m}}\lesssim (\nu^{-\frac23}\eps)\eps^2.
\end{equation}
For the zero mode contributions, observe that
\begin{equation}\label{eq:V0energy_split}
  \l\widetilde V_0, U^2_0\de_yV_0+U^3_0\de_zV_0\r_{H^{2m}}=\l\widetilde V_0,U_0^2\de_y\overline{V}_0+\overline{U}^3_0\de_z\widetilde{V}_0\r+\l\widetilde{V}_0, U^2_0\de_y\widetilde{V}_0+\widetilde{U}^3_0\de_z\widetilde{V}_0)\r_{H^{2m}}.
\end{equation}
The first term in \eqref{eq:V0energy_split} can be bounded as
\begin{equation}
  \abs{\l\widetilde V_0,U_0^2\de_y\overline{V}_0+\overline{U}^3_0\de_z\widetilde{V}_0\r}\lesssim  \norm{\widetilde{V}_0}_{H^{2m}}\left(\norm{U^2_0\de_y\overline{V}_0}_{H^{2m}}+\norm{\overline{U}^3_0\de_z\widetilde{V}_0}_{H^{2m}}\right).
\end{equation}
With
\begin{equation}
  \norm{\overline U^3_0\de_z\widetilde{V}_0}_{H^{2m}}\lesssim \eps(\beta^{-\frac13}\nu^{-\frac53}\eps^2+\nu^{-\frac23}\eps)\norm{\de_z\widetilde{V}_0}_{H^{2m}}  
\end{equation}
we have that
\begin{equation}
  \int_0^\infty  \norm{\widetilde{V}_0}_{H^{2m}}\norm{\overline U^3_0\de_z\widetilde{V}_0}_{H^{2m}}\lesssim \eps^3\nu^{-1}\cdot(\beta^{-\frac13}\nu^{-\frac53}\eps^2+\nu^{-\frac23}\eps)\lesssim \eps^2(\beta^{-\frac13}\nu^{-\frac83}\eps^3+\nu^{-\frac53}\eps^2).
\end{equation}
Using the decomposition $U^2_0 = U^{2,in}_0+U^{2,nl}_0$ from Lemma \ref{lemma:innl_decomp}, we obtain from \eqref{eq:Upsdecomp_bds} that
\begin{equation}\label{eq:prod_decomp1}
\begin{aligned}
  \norm{U^2_0\de_y \overline V_0}_{H^{2m}}&\lesssim \norm{\l\nabla\r^{2m}U^{2,in}_0}_{L^\infty}\norm{ \de_y\overline{V}_0}_{H^{2m}}+\norm{U^{2,nl}_0}_{H^{2m}}\norm{\de_y\overline{V}_0}_{H^{2m}}\\
  &\lesssim \eps (\beta t)^{-\frac13}\e^{-\nu t}\norm{ \de_y\overline{V}_0}_{H^{2m}}\\
  &\quad +\left(\eps\left(\nu^{-\frac23}\eps+\nu^{-\frac53}\eps^2\right)+\beta^{-\frac13}\eps\left(\nu^{-\frac23}\eps+\nu^{-\frac53}\eps^2+\nu^{-\frac83}\eps^3\right)\right)\norm{ \de_y\overline{V}_0}_{H^{2m}},
\end{aligned}  
\end{equation}
which upon time integration yields
$$
\begin{aligned}
  \int_0^\infty  \norm{\widetilde{V}_0}_{H^{2m}}\norm{U^2_0\de_y \overline V_0}_{H^{2m}}&\lesssim
  %\beta^{-\frac13}\eps^3\nu^{-\frac16-\frac12}+\eps^3\nu^{-1}\left(\nu^{-\frac23}\eps+\nu^{-\frac53}\eps^2\right)\\
  %&\quad+\beta^{-\frac13}\eps^3\nu^{-1}\left(\nu^{-\frac23}\eps+\nu^{-\frac53}\eps^2+\nu^{-\frac83}\eps^3\right)\\
  \eps^2\left(\nu^{-\frac53}\eps^2+\nu^{-\frac83}\eps^3\right)+\beta^{-\frac13}\eps^2\left(\nu^{-\frac23}\eps+\nu^{-\frac53}\eps^2+\nu^{-\frac83}\eps^3+\nu^{-\frac{11}{3}}\eps^4\right).
\end{aligned}  
$$
Next, we observe that in the second term of \eqref{eq:V0energy_split} the functions $\widetilde{U}^1_0$ and $\widetilde{\Theta}_0$ cannot occur quadratically at highest order: by incompressibility we have
\begin{equation}\label{eq:V0energy_cancel}
  \l F, U^2_0\de_yF+\widetilde{U}^3_0\de_z F\r=0
\end{equation}
and thus
\begin{equation}\label{eq:V0energy_finedecomp}
\begin{aligned}
  \l\widetilde{V}_0, U^2_0\de_y\widetilde{V}_0+\widetilde{U}^3_0\de_z\widetilde{V}_0\r_{H^{2m}} &=\beta^{-1} \l\widetilde{U}^1_0, U^2_0\de_y\widetilde{\Theta}_0+\widetilde{U}^3_0\de_z\widetilde{\Theta}_0\r_{H^{2m}} +\beta^{-1} \l\widetilde{\Theta}_0, U^2_0\de_y\widetilde{U}^1_0+\widetilde{U}^3_0\de_z\widetilde{U}^1_0\r_{H^{2m}} \\ 
  &\quad +\sum_{F\in\widetilde{U}^1_0,\beta^{-1}\widetilde{\Theta}_0}\sum_{\alpha\in\NN_0^2,\abs{\alpha}\leq 2m}\l\de^\alpha\widetilde{V}_0,R_\alpha^2(F)+R_\alpha^3(F)\r
\end{aligned}  
\end{equation}
where the remainder terms $R^j_\alpha$, $j=2,3$, are given by
\begin{equation}
 R_\alpha^2(F)=\sum_{\substack{\gamma_1,\gamma_2\in\NN_0^2,\gamma_1\neq 0,\\ \alpha=\gamma_1+\gamma_2}}\de^{\gamma_1} U^2_0\de_y\de^{\gamma_2} F,\qquad R_\alpha^3(F)=\sum_{\substack{\gamma_1,\gamma_2\in\NN_0^2,\gamma_1\neq 0,\\ \alpha=\gamma_1+\gamma_2}}\de^{\gamma_1} \widetilde{U}^3_0\de_z\de^{\gamma_2} F
\end{equation}
The first term in \eqref{eq:V0energy_finedecomp} can be bounded as
$$
 \abs{\l\widetilde{U}^1_0, U^2_0\de_y\widetilde{\Theta}_0+\widetilde{U}^3_0\de_z\widetilde{\Theta}_0\r_{H^{2m}}}\lesssim \norm{\widetilde{U}^1_0}_{H^{2m}}\left(\sum_{j=2,3}\norm{\widetilde{U}^j_0}_{L^\infty}\norm{\de_j\widetilde{\Theta}_0}_{H^{2m}}+\norm{\widetilde{U}^j_0}_{H^{2m}}\norm{\de_j\widetilde{\Theta}_0}_{L^\infty}\right),   
$$
and it suffices to invoke the $L^\infty$ decay of $\widetilde{U}^j_0$ and $\de_y\widetilde{\Theta}_0$ from Proposition \ref{prop:LinftyGGamma}. For the second term in \eqref{eq:V0energy_finedecomp} we have that
\begin{equation}
\begin{aligned}
  \l\widetilde{\Theta}_0, U^2_0\de_y\widetilde{U}^1_0+\widetilde{U}^3_0\de_z\widetilde{U}^1_0\r_{H^{2m}}&=\l\de_y\widetilde{\Theta}_0, U^2_0\widetilde{U}^1_0\r_{H^{2m}} +\l\de_z\widetilde{\Theta}_0,\widetilde{U}^3_0\widetilde{U}^1_0\r_{H^{2m}}\\
  &=\l\de_y\widetilde{\Theta}_0, U^2_0\widetilde{U}^1_0\r_{H^{2m}} +\l\de_y\de_z\widetilde{\Theta}_0,\de_z^{-1}U^2_0\widetilde{U}^1_0\r_{H^{2m}}\\
  &\quad+\l\de_z\widetilde{\Theta}_0,\de_z^{-1}U^2_0\de_y\widetilde{U}^1_0\r_{H^{2m}},
\end{aligned}  
\end{equation}
and since by \eqref{eq:U2Theta_recover} we have
\begin{equation}
  \norm{\de_y\de_z\widetilde{\Theta}_0}_{H^{2m}}\lesssim \norm{\nabla\Gamma_0}_{H^{2m}}  ,
\end{equation}
we can proceed as above with the decomposition from Lemma \ref{lemma:innl_decomp}, compare \eqref{eq:prod_decomp1}. 
Analogous arguments give the claim for all remainder terms in \eqref{eq:V0energy_finedecomp}: This is clear for all terms involving $R^2_\alpha$, and since $\widetilde{U}_0^3=-\de_z^{-1}\de_yU^2_0$ also for all terms involving $R^3_\alpha$ whenever $\abs{\alpha}<2m$, or when $\abs{\alpha}=2m$ and $\gamma_1\neq \alpha$. To conclude it suffices to observe that upon integration by parts in the last remaining term (with $\alpha=\gamma_1$, $\gamma_2=0$) we obtain
\begin{equation}
\begin{aligned}
  -\l\de^\alpha\widetilde{V}_0,\de^\alpha \de_z^{-1}\de_y U^2_0\de_zF\r= \l\de^\alpha\de_y\widetilde{V}_0,\de^\alpha \de_z^{-1} U^2_0\de_z F\r+\l\de^\alpha\widetilde{V}_0,\de^\alpha \de_z^{-1} U^2_0\de_y\de_z F\r.
\end{aligned}  
\end{equation}
\end{proof}

\begin{proof}[Proof of \eqref{eq:energy_V00}]
 By construction (see \eqref{eq:V_0_dynamics}), $\overline{V}_0$ satisfies
 \begin{equation}\label{eq:V00_dynamics}
 \de_t \overline{V}_0+\de_y(\overline{U^2V})_0=\nu\de_{yy}\overline{V}_0.
 \end{equation}
 As in \eqref{eq:noselfdouble0} we have with \eqref{eq:ulU20=0} that
 \begin{equation}\label{eq:noselfdouble0V}
  (\overline{U^2V})_0=(\overline{U^2_0\widetilde{V}_0})+(\overline{U^2_{\neq} V_{\neq}})_0,
 \end{equation}
 and energy estimates together with \eqref{eq:double0idV} imply that
 \begin{equation}
   \norm{\overline{V}_0(t)}_{H^{2m}}^{2}+\nu\int_0^t\norm{\de_y\overline{V}_0}_{H^{2m}}^2\dd \tau \leq \nu^{-1}\int_0^t  \norm{(\overline{U^2_0\widetilde{V}_0})(\tau)}_{H^{2m}}^2 +\norm{(\overline{U^2_{\neq} V_{\neq}})_0(\tau)}_{H^{2m}}^2\dd\tau.
 \end{equation}
Decomposing $U^2_0 = U^{2,in}_0+U^{2,nl}_0$ as in Lemma \ref{lemma:innl_decomp}, it follows that
\begin{equation}
\begin{aligned}
 \norm{(\overline{U^2_0\widetilde{V}_0})(\tau)}_{H^{2m}}&\lesssim \norm{\l\nabla\r^{2m}U_0^{2,in}(\tau)}_{L^\infty}\norm{ \widetilde{V}_0(\tau)}_{H^{2m}}+\norm{U_0^{2,nl}(\tau)}_{H^{2m}}\norm{\widetilde{V}_0(\tau)}_{H^{2m}}\\
&\lesssim (\beta t)^{-\frac13}\e^{-\nu t}\eps^2 + \eps\left[\nu^{-\frac23}\eps+\nu^{-\frac53}\eps^2\right]\norm{\widetilde{V}_0(\tau)}_{H^{2m}}\\
&\qquad +\beta^{-\frac13}\eps\left[\nu^{-\frac23}\eps+\nu^{-\frac53}\eps^2+\nu^{-\frac83}\eps^3\right]\norm{\widetilde{V}_0(\tau)}_{H^{2m}}.
\end{aligned} 
\end{equation}
Integrating in time and using dissipation (since $l\neq0$) leads to 
$$
    \nu^{-1}\int_0^\infty \norm{(\overline{U^2_0\widetilde{V}_0})(\tau)}^2_{H^{2m}}
   \lesssim  \eps^2\left[\nu^{-\frac53}\eps^2+\nu^{-\frac83}\eps^3\right]^2+\beta^{-\frac23}\eps^2\left[\nu^{-\frac23}\eps+\nu^{-\frac53}\eps^2+\nu^{-\frac83}\eps^3+\nu^{-\frac{11}{3}}\eps^4\right]^2.
$$
Observing that
\begin{equation}
 \norm{(\overline{U^2_{\neq} V_{\neq}})_0(\tau)}_{H^{2m}}\lesssim \norm{\cA \GG_{\neq}(\tau)}^2+\norm{\cA \Gamma_{\neq}(\tau)}^2 +\norm{\cA U^1_{\neq}(\tau)}^2
\end{equation}
and integrating in time (using also \eqref{eq:enhanced_diss}) gives the claim.
\end{proof}

\subsection{Control of zero modes -- proof of Proposition \ref{prop:zero}}\label{ssec:btstrap_zero}
In this section we show how the simple zero modes $\GG_0$, $\Gamma_0$ and (as a consequence) $\widetilde{U}^3_0$ can be controlled. Invoking again the energy structure for the symmetrized variables $\GG_0$ and $\Gamma_0$, we have that
\begin{equation}\label{eq:G0Gamma0_energy_eq}
\begin{aligned}
\frac12\ddt\left(\norm{\GG_0}^2_{H^{2m}}+\norm{\Gamma_0}^2_{H^{2m}}\right) &=-\nu\left[\norm{\nabla\GG_0}^2_{H^{2m}}+\norm{\nabla\Gamma_0}^2_{H^{2m}}\right]  + \l\GG_0,\,\cT_\star(U,U^2)\r_{H^{2m}}\\
&\qquad + \langle \GG_0,\,\de_y\cP_\star(U,U)\rangle_{H^{2m}}+ \langle \Gamma_0,\,\cT_\circ(U,U^2)\rangle_{H^{2m}}.
\end{aligned}
\end{equation}
The nonlinear interactions satisfy
\begin{equation}
 \cB(U,F)_0=\cB(U_0,F_0)_0 +\cB(U_{\neq},F_{\neq})_0,\qquad \cB\in\{\cT_\star,\cT_\circ,\cP_\star\}, F\in\{U,U^2\}, 
\end{equation}
which is analogous to the decomposition employed in \eqref{def:cTstar_decomposition}, \eqref{def:cTcirc_decomposition} and \eqref{def:cPstar_decomposition}.
We will show:
\begin{lemma}\label{lem:G0Gamma0_energy}
  Under the assumptions of Theorem \ref{thm:bootstrap_step}, there holds that
  \begin{equation}\label{eq:allneqs}
    \int_0^\infty \l\GG_0,\,\cT_\star(U_{\neq},U^2_{\neq})\r_{H^{2m}}+\langle \GG_0,\,\de_y\cP_\star(U_{\neq},U_{\neq})\rangle_{H^{2m}}+\langle \Gamma_0,\,\cT_\circ(U_{\neq},U^2_{\neq})\rangle_{H^{2m}}\lesssim (\nu^{-\frac34}\eps)\eps^2,
  \end{equation}
  and
  \begin{equation}\label{eq:allzeros}  
  \begin{aligned}
    &\int_0^\infty  \l\GG_0,\,\cT_\star(U_0,U^2_0)\r_{H^{2m}}+\langle \GG_0,\,\de_y\cP_\star(U_0,U_0)\rangle_{H^{2m}}+\langle \Gamma_0,\,\cT_\circ(U_0,U^2_0)\rangle_{H^{2m}}\\
    &\qquad\qquad \lesssim \eps^2\left(\nu^{-\frac12}\eps+\nu^{-\frac53}\eps^2\right)+\beta^{-\frac13}\eps^2\left(\nu^{-\frac56}\eps+\nu^{-\frac53}\eps^2+\nu^{-\frac83}\eps^3\right).
  \end{aligned}
  \end{equation}
\end{lemma}

By \eqref{eq:G0Gamma0_energy_eq} this implies that 
\begin{equation}\label{eq:G0Gamma0_energy}
    \norm{\GG_0}^2_{L^\infty_tH^{2m}}+\norm{\Gamma_0}^2_{L^\infty_tH^{2m}}+\nu\norm{\nabla\GG_0}^2_{L^2_tH^{2m}}+\nu\norm{\nabla\Gamma_0}^2_{L^2_tH^{2m}}  \leq \eps_0^2+C_3(\nu^{-\frac56}\eps+\beta^{-\frac13}\nu^{-\frac89}\eps)\eps^2.
  \end{equation}
As a simple consequence, we also have that
\begin{equation}\label{eq:U30_energy}
  \norm{\widetilde{U}^3_0}^2_{L^\infty_tH^{2m}}+\nu\norm{\nabla \widetilde{U}^3_0}^2_{L^2_tH^{2m}} \leq \eps_0^2+C_3(\nu^{-\frac56}\eps+\beta^{-\frac13}\nu^{-\frac89}\eps)\eps^2,  
\end{equation}
which together with Lemma \eqref{eq:G0Gamma0_energy} completes the proof of Proposition \ref{prop:zero}.
\begin{proof}[Proof of \eqref{eq:U30_energy}]
 It suffices to observe that by \eqref{eq:U2Theta_recover} and incompressibility there holds that
 \begin{equation}
   \widetilde{U}^3_0=-\de_z^{-1}\de_yU^2_0=\de_z^{-1}\de_y\abs{\de_z}^{\frac12}\abs{\nabla_{y,z}}^{-\frac32}\GG_0.
 \end{equation}
\end{proof}

\begin{proof}[Proof of Lemma \ref{lem:G0Gamma0_energy}]
  We establish the bound \eqref{eq:allneqs} in Section \ref{ssec:allneqs}, while \eqref{eq:allzeros} is demonstrated in Section \ref{ssec:allzeros}.
\end{proof}

\subsubsection{Nonlinear terms analysis: $(\neq,\neq)$ interactions}\label{ssec:allneqs}
Here we prove the bound \eqref{eq:allneqs}. We make use of the decompositions \eqref{def:cTstar_decomposition}, \eqref{def:cTcirc_decomposition} and \eqref{def:cPstar_decomposition} and introduced in Section \ref{ssec:GGamma_neq}.

By definition we have that 
\begin{equation}
\langle \GG_0,\cT_\star(U_{\neq},U^2_{\neq})\rangle_{H^{2m}}=\langle \GG_0,|\de_z|^{-\frac12}|\nabla_{y,z}|^{\frac32}(U_{\neq}\cdot \nabla_L |\nabla_{x,z}|^{\frac12}|\nabla_L|^{-\frac32}\GG_{\neq})\r_{H^{2m}}.
\end{equation}
This term shares the same structure as the one in Section \ref{sssec:cTstarGGamma}, whose bounds are dictated by the behaviour of the nonzero frequencies interactions. For this reason, we can mirror the arguments presented there to deduce that 
\begin{subequations}
    \begin{align}
       \int_0^\infty|\langle \GG_0,\cT^1_\star(U_{\neq},U^2_{\neq})\rangle_{H^{2m}}| &\lesssim (\nu^{-\frac34}\eps)\eps^2,\\
       \int_0^\infty|\langle \GG_0,\cT^2_\star(U_{\neq},U^2_{\neq})\rangle_{H^{2m}}| &\lesssim (\nu^{-\frac12}\eps)\eps^2,\\
       \int_0^\infty|\langle \GG_0,\cT^3_\star(U_{\neq},U^2_{\neq})\rangle_{H^{2m}}| &\lesssim (\nu^{-\frac34}\eps)\eps^2.
    \end{align}
\end{subequations}
Similarly, the term 
\begin{align}
\l \Gamma_0, \cT_\circ(U_{\neq},\Theta_{\neq})\r_{H^{2m}}=\langle \Gamma_0, |\de_z|^{\frac12}|\nabla_{y,z}|^{\frac12}(U_{\neq}\cdot \nabla_L |\nabla_{x,z}|^{-\frac12}|\nabla_L|^{-\frac12}\Gamma_{\neq}) \rangle_{H^{2m}},
\end{align}
closely resembles $\cT_\circ(U_{\neq},\Theta_{\neq})_{\neq}$, treated in Section \ref{sssec:cTcircGGamma}, and again we deduce
\begin{subequations}
    \begin{align}
    \int_0^\infty|\l \Gamma_0, \cT^1_\circ(U_{\neq},\Theta_{\neq}) \r_{H^{2m}}| &\lesssim (\nu^{-\frac23}\eps)\eps^2,\\
    \int_0^\infty|\l \Gamma_0,\cT^2_\circ(U_{\neq},\Theta_{\neq})\r_{H^{2m}}| &\lesssim (\nu^{-\frac12}\eps)\eps^2,\\
     \int_0^\infty|\l \Gamma_0, \cT^3_\circ(U_{\neq},\Theta_{\neq}) \r_{H^{2m}}| &\lesssim (\nu^{-\frac23}\eps)\eps^2.
    \end{align}
\end{subequations}

Finally, the nonlinear term arising from the pressure is
\begin{equation}
\l \GG_0, \de_y\cP_\star(U_{\neq},U_{\neq})\r_{H^{2m}}=\langle \GG_0, |\de_z|^{-\frac12}|\nabla_{y,z}|^{-\frac12}\de_y(\de_i^LU_{\neq}^j\de_j^LU^i_{\neq})_0\rangle_{H^{2m}},
\end{equation}
for $i,j\in\{1,2,3\}$ with the usual convention $\de_1^L=\de_x$, $\de_2^L=\de_y^L$ and $\de_3^L=\de_z$.
Comparing this to $\cP_\star(U_{\neq},U_{\neq})_{\neq}$ of Section \ref{sssec:cPGGamma} and remarking that it is possible to mimic the computations done in the nonzero case, this implies the bounds
\begin{subequations}
    \begin{align}
    \int_0^\infty |\l \GG_0, \cP_\star^{1,r}(U_{\neq},U_{\neq})\r_{H^{2m}}| &\lesssim (\nu^{-\frac23}\eps)\eps^2, \quad r=1,3,\\
    \int_0^\infty |\l \GG_0, \cP_\star^{3,3}(U_{\neq},U_{\neq})\r_{H^{2m}}| &\lesssim (\nu^{-\frac23}\eps)\eps^2,\\
    \int_0^\infty |\l \GG_0, \cP_\star^{r,2}(U_{\neq},U_{\neq})\r_{H^{2m}}| &\lesssim (\nu^{-\frac34}\eps)\eps^2,\quad r=1,3,\\
    \int_0^\infty |\l \GG_0, \cP_\star^{2,2}(U_{\neq},U_{\neq})\r_{H^{2m}}| &\lesssim (\nu^{-\frac13}\eps)\eps^2.
    \end{align}
\end{subequations}

\subsubsection{Nonlinear terms analysis: $(0,0)$ interactions}\label{ssec:allzeros}
We are now left with establishing the bound \eqref{eq:allzeros} for the $(0,0)$ interactions of $\cT_\star$, $\cT_\circ$ and $\de_y\cP_\star$. These are comparatively delicate, as the terms do not dissipate at an enhanced rate or feature inviscid damping type decay. Instead we have to rely on dispersive decay.

Beginning with $\cT_\star$, since $\de_xU^2_0=0$ there holds that
\begin{equation}
\begin{aligned}
  \cT_\star(U_0,U^2_0)&=\cT^2_\star(U_0,U^2_0)+\cT^3_\star(U_0,U^2_0) \\
  &=|\de_z|^{-\frac12}|\nabla_{y,z}|^{\frac32}(U_0^2\de_y U^2_0)+|\de_z|^{-\frac12}|\nabla_{y,z}|^{\frac32}(U_0^3\de_z U^2_0).
\end{aligned}  
\end{equation}
We can treat $\cT^2_\star(U_0,U^2_0)_0$ directly by estimating
$$
\begin{aligned}
  |\l \GG_0, \cT^2_\star(U_0,U^2_0) \r_{H^{2m}}| &\lesssim \norm{|\de_z|^{-\frac12}|\nabla_{y,z}|^{\frac12} \GG_0}_{H^{2m}}\norm{U^2_0\de_yU^2_0}_{H^{2m+1}}\\
  &\lesssim \norm{\GG_0}_{H^{2m}}^{\frac12}\norm{\nabla\GG_0}_{H^{2m}}^{\frac12}\left(\norm{U_0^2}_{L^\infty}\norm{\de_yU_0^2}_{H^{2m+1}}+\norm{U_0^2}_{H^{2m+1}}\norm{\de_yU_0^2}_{L^\infty}\right)\\
  &\lesssim \norm{\GG_0}_{H^{2m}}^{\frac12}\norm{\nabla\GG_0}_{H^{2m}}^{\frac32}\norm{U_0^2}_{L^\infty}+\norm{\GG_0}_{H^{2m}}^{\frac32}\norm{\nabla\GG_0}_{H^{2m}}^{\frac12}\norm{\de_yU_0^2}_{L^\infty},
\end{aligned}  
$$
which upon integration in time (invoking the crude bound $\norm{U^2_0}_{L^\infty}\lesssim\norm{G_0}_{H^{2m}}$ and Proposition \ref{prop:LinftyGGamma} for small and large times, respectively\footnote{
We have that
$$
     \int_0^\infty \min\{1,\beta^{-\frac13}t^{-\frac13}\e^{-\nu t}\}\norm{\nabla\GG_0}^{\frac32}_{H^{2m}}\lesssim \int_0^\nu \norm{\nabla\GG_0}^{\frac32}_{H^{2m}} +\int_\nu^\infty t^{-\frac13}\e^{-\nu t}\norm{\nabla\GG_0}^{\frac32}_{H^{2m}}\lesssim (\nu^{-\frac12}+\beta^{-\frac13}\nu^{-\frac56})\eps^\frac32.
$$
}) yields that
\begin{equation}
    \int_0^\infty|\l\GG_0, \cT^2_\star(U_0,U^2_0) \r_{H^{2m}}| \lesssim \left(\nu^{-\frac12}\eps+\beta^{-\frac13}(\nu^{-\frac56}\eps+\nu^{-\frac53}\eps^2)\right)\eps^2.
\end{equation}

For $\cT^3_\star(U_0,U^2_0)$ we start by decomposing
\begin{equation}
  \cT^3_\star(U_0,U^2_0)=\cT^3_\star(\widetilde{U}_0,U^2_0)+\cT^3_\star(\overline{U}_0,U^2_0).  
\end{equation}
Using incompressibility to write $\widetilde{U}^3_0=-\de_z^{-1}\de_yU^2_0$, we have 
\begin{equation}
  \l \GG_0,\cT^3_\star(\widetilde{U}_0,U^2_0)\r_{H^{2m}}=-\l \GG_0,|\de_z|^{-\frac12}|\nabla_{y,z}|^{\frac32}(\de_z^{-1}\de_yU^2_0\de_zU^2_0)\r_{H^{2m}}, 
\end{equation}
which can be bounded as above to yield
\begin{equation}
    \int_0^\infty|\l\GG_0, \cT^3_\star(\widetilde{U}_0,U^2_0) \r_{H^{2m}}| \lesssim \left(\nu^{-\frac12}\eps+\beta^{-\frac13}(\nu^{-\frac56}\eps+\nu^{-\frac53}\eps^2)\right)\eps^2.
\end{equation}
Next, we have that
\begin{align*}
  \abs{\l\GG_0,\cT^3_\star(\overline{U}_0,U^2_0)\r_{H^{2m}}}&=\abs{\l \GG_0,|\de_z|^{-\frac12}|\nabla_{y,z}|^{\frac32}(\overline{U}^3_0\de_z|\de_z|^{\frac12}|\nabla_{y,z}|^{-\frac32}\GG_0)\r_{H^{2m}}}\\
  &\lesssim \norm{|\de_z|^{-\frac12}|\nabla_{y,z}|^{\frac12}\GG_0}_{H^{2m}}\left(\norm{\de_y\overline{U}^3_0}_{H^{2m}}\norm{\GG_0}_{H^{2m}}+\norm{\overline{U}^3_0}_{H^{2m}}\norm{\nabla \GG_0}_{H^{2m}}\right)\\
  &\lesssim \norm{\nabla\GG_0}_{H^{2m}}^{\frac12}\norm{\GG_0}_{H^{2m}}^{\frac12}\left(\norm{\de_y\overline{U}^3_0}_{H^{2m}}\norm{\GG_0}_{H^{2m}}+\norm{\overline{U}^3_0}_{H^{2m}}\norm{\nabla \GG_0}_{H^{2m}}\right).
\end{align*}  
Using the energy estimates for the double zero mode $\overline{U}^3_0$ from Lemma \ref{lemma:zzeroU3theta} and integrating in time gives
\begin{equation}
    \int_0^\infty|\l\GG_0, \cT^3_\star(\overline{U}_0,U^2_0) \r_{H^{2m}}| \lesssim (\beta^{-\frac13}\nu^{-\frac83}\eps^3 +\nu^{-\frac53}\eps^2)\eps^2.
\end{equation}

To treat $\cT_\circ(U_0,\Theta_0)$, we note that
\begin{equation}
  \cT_\circ(U_0,\Theta_0)=\cT^2_\circ(U_0,\Theta_0)+\cT^3_\circ(U_0,\Theta_0) , 
\end{equation}
and since $\Gamma_0=\widetilde{\Gamma}_0$ self-interactions of double zero modes do not contribute to $\l \Gamma_0,\cT_\circ(U_0,\Theta_0)\r$. That is,
\begin{equation}
  \l \Gamma_0,\cT^2_\circ(U_0,\Theta_0)\r_{H^{2m}}=\l \Gamma_0,\cT^2_\circ(U_0,\widetilde{\Theta}_0)\r_{H^{2m}}+\l \Gamma_0,\cT^2_\circ(U_0,\overline{\Theta}_0)\r_{H^{2m}}  
\end{equation}
and (since $\de_z\overline{\Theta}_0=0$)
\begin{equation}
  \l \Gamma_0,\cT_\circ^3(U_0,\Theta_0)\r_{H^{2m}}=\l \Gamma_0,\cT^3_\circ(U_0,\widetilde{\Theta}_0)\r_{H^{2m}}=\l \Gamma_0,\cT^3_\circ(\overline{U}_0,\widetilde{\Theta}_0)\r_{H^{2m}}+\l \Gamma_0,\cT^3_\circ(\widetilde{U}_0,\widetilde{\Theta}_0)\r_{H^{2m}},  
\end{equation}
which can all be treated as above.

Finally for $\cP_\star$ we have, using that $\de_x U_0=0$, the symmetry \eqref{eq:pressure_symmetry} and incompressibility,
\begin{align}
\l \GG_0, \de_y\cP_\star(U_0,U_0)\r_{H^{2m}}&=\langle \GG_0, |\de_z|^{-\frac12}|\nabla_{y,z}|^{-\frac12}\de_y(\de_iU_0^j\de_jU^i_0)\rangle_{H^{2m}}\notag\\
&=2\l \GG_0, \de_y\cP^{2,2}_\star(U_0,U_0)\r_{H^{2m}} +2\l \GG_0, \de_y\cP^{2,3}_\star(U_0,U_0)\r_{H^{2m}}.
% &=\l\GG_0, |\de_z|^{-\frac12}|\nabla_{y,z}|^{-\frac12}\de_y(\de_yU_0^2\de_yU^2_0)\r_{H^{2m}} \\
% &\quad +\l\GG_0, |\de_z|^{-\frac12}|\nabla_{y,z}|^{-\frac12}\de_y(\de_yU_0^3\de_zU^2_0)\r_{H^{2m}}\\
% &\quad +\l\GG_0, |\de_z|^{-\frac12}|\nabla_{y,z}|^{-\frac12}\de_y(\de_zU_0^3\de_zU^3_0)\r_{H^{2m}}\\
% &=\l \GG_0, \de_y\cP^{2,2}_\star(U_0,U_0)\r_{H^{2m}} +\l \GG_0, \de_y\cP^{2,3}_\star(U_0,U_0)\r_{H^{2m}} +\l \GG_0, \de_y\cP^{3,3}_\star(U_0,U_0)\r_{H^{2m}}.
\end{align}
The first term can simply be bounded as
\begin{equation}
|\l \GG_0, \de_y\cP^{2,2}_\star(U_0,U_0)\r_{H^{2m}}|\lesssim\norm{|\de_z|^{-\frac12}|\nabla_{y,z}|^{-\frac12}\de_y\GG_0}_{H^{2m}}\norm{\de_yU_0^2}_{L^\infty}\norm{\de_yU^2_0}_{H^{2m}}.
\end{equation}
Invoking the decay from Proposition \ref{prop:LinftyGGamma} and integrating in time gives
\begin{equation}
    \int_0^\infty |\l\GG_0, \de_y\cP^{2,2}_\star(U_0,U_0) \r_{H^{2m}}| \lesssim \beta^{-\frac13}(\nu^{-\frac23}\eps+\nu^{-\frac53}\eps^2)\eps^2.
\end{equation}
On the other hand, we note that double zero modes appear only linearly in $\cP^{2,3}_\star(U_0,U_0)$, i.e.\
\begin{equation}
  \cP^{2,3}_\star(U_0,U_0)=\cP^{2,3}_\star(\overline{U}_0,U_0)+\cP^{2,3}_\star(\widetilde{U}_0,U_0).  
\end{equation}
As before we thus have
\begin{equation}
    \int_0^\infty |\l\GG_0, \de_y\cP^{2,3}_\star(\widetilde{U}_0,U^2_0) \r_{H^{2m}}| \lesssim \beta^{-\frac13}(\nu^{-\frac23}\eps+\nu^{-\frac53}\eps^2)\eps^2.
\end{equation}
Finally, 
\begin{equation}
\begin{aligned}
|\l \GG_0, \de_y\cP^{2,3}_\star(\overline{U}_0,U_0)\r_{H^{2m}}|&\lesssim\norm{|\de_z|^{-\frac12}|\nabla_{y,z}|^{-\frac12}\de_y\GG_0}_{H^{2m}}\norm{\de_y\overline{U}_0^3}_{H^{2m}}\norm{\de_zU^2_0}_{H^{2m}}\\
&\lesssim \norm{\GG_0}_{H^{2m}}^{\frac32}\norm{\nabla\GG_0}_{H^{2m}}^{\frac12}\norm{\de_y\overline{U}_0^3}_{H^{2m}},
\end{aligned}
\end{equation}
and we use the energy bounds from Lemma \ref{lemma:zzeroU3theta} to conclude that
\begin{equation}
    \int_0^\infty|\l\GG_0, \de_y\cP^{2,3}_\star(\overline{U}_0,U^2_0) \r_{H^{2m}}| \lesssim (\beta^{-\frac13}\nu^{-\frac83}\eps^3 +\nu^{-\frac53}\eps^2)\eps^2.
\end{equation}

\addtocontents{toc}{\protect\setcounter{tocdepth}{0}}
\section*{Acknowledgments}
The authors would like to thank J.\ Bedrossian and M.\ Dolce for illuminating discussions.
The research of MCZ was partially supported by the Royal Society URF\textbackslash R1\textbackslash 191492 and EPSRC Horizon Europe Guarantee EP/X020886/1. KW gratefully acknowledges support of the SNSF through grant PCEFP2\_ 203059.

\addtocontents{toc}{\protect\setcounter{tocdepth}{1}}

\bibliographystyle{abbrv}
\bibliography{CZDZW-Nonlinear3DBoussinesq.bib}

\end{document}